\documentclass[12pt]{article}
\usepackage[margin=1.25in]{geometry}
\usepackage{amsmath, amssymb,amsthm, euscript, tabularx, hyperref}
\usepackage{cleveref}
\usepackage{graphicx,tikz, etoolbox}
\usepackage{wrapfig}
\newtheorem{THM}{Theorem}
\newtheorem{LMA}[THM]{Lemma}
\newtheorem{PROP}[THM]{Proposition}
\newtheorem{CORO}[THM]{Corollary}
\newtheorem{CONJ}[THM]{Conjecture}

\newtheorem{DEF}[THM]{Definition}

\newtheorem*{RMK}{Remark}

\newtheorem{claim}[THM]{Claim}
\newtheorem{subclaim}[THM]{Subclaim}

\numberwithin{equation}{section}
\numberwithin{THM}{section}

\newcommand\opn[1]{ {\operatorname{ #1 } }}
\newcommand\ov[1]{\overline{#1}}
\newcommand\wt[1]{\widetilde{#1}}
\newcommand\offplus{\hspace{1em}}

\newcommand{\tildx}{\widetilde{x}}
\newcommand\Gclone[2]{G_{#1 \rightarrow #2} }
\newcommand\clone[3]{{#1}_{#2 \rightarrow #3}}
\newcommand{\pk}{\ov{p}}

\newcommand{\lb}{\lbrack}
\newcommand{\rb}{\rbrack}

\newcommand{\buc}{\Delta}
\newcommand{\gap}{\Gamma}

\crefname{CONJ}{conjecture}{conjectures}
\crefname{THM}{theorem}{theorems}
\crefname{LMA}{lemma}{lemmas}
\crefname{CORO}{corollary}{corollaries}
\crefname{CONJ}{conjecture}{conjectures}
\crefname{claim}{claim}{claims}

\begin{document}
%%%%%%%%%%%%%%%%%%%%%%%%%%%%%%%%%%%%%%%%%
%\baselineskip=12pt
%\phantom{a}\vskip .2in
\centerline{{\bf  On the Minimal Edge Density of $K_4$-free 6-critical Graphs}}%
\vskip .3in
%authors
\centerline{{\bf Wenbo Gao}}
\smallskip
\centerline{Department of Industrial Engineering and Operations Research}
\centerline{Columbia University}
\centerline{New York, NY}
\centerline{USA 10027}
\centerline{{\tt wg2279@columbia.edu}}
\bigskip
\centerline{and}
\bigskip
\centerline{{\bf Luke Postle}%
\footnote{Partially
supported by NSERC under Discovery Grant No. 2014-06162, the Ontario Early 
Researcher Awards program and the Canada Research Chairs program.}}
\smallskip
\centerline{Department of Combinatorics and Optimization}
\centerline{University of Waterloo}
\centerline{Waterloo, ON}
\centerline{Canada N2L 3G1}
\centerline{{\tt lpostle@uwaterloo.ca}}
%end of authors

\vskip 0.4in \centerline{\bf ABSTRACT}
\bigskip

\noindent Kostochka and Yancey resolved a famous conjecture of Ore on the asymptotic density of $k$-critical graphs by proving that every $k$-critical graph $G$ satisfies $|E(G)| \geq (\frac{k}{2} - \frac{1}{k-1})|V(G)| - \frac{k(k-3)}{2(k-1)}$. The class of graphs for which this bound is tight, $k$-Ore graphs, contain a notably large number of $K_{k-2}$-subgraphs. Subsequent work attempted to determine the asymptotic density for $k$-critical graphs that do \emph{not} contain large cliques as subgraphs, but only partial progress has been made on this problem. The second author showed that if $G$ is 5-critical and has no $K_3$-subgraphs, then for $\varepsilon = 1/84$, $|E(G)| \geq (\frac{9}{4} + \varepsilon)|V(G)| - \frac{5}{4}$. It has also been shown that for all $k \geq 33$, there exists $\varepsilon_k > 0$ such that $k$-critical graphs with no $K_{k-2}$-subgraphs satisfy $|E(G)| \geq (\frac{k}{2} - \frac{1}{k-1} + \varepsilon_k)|V(G)| - \frac{k(k-3)}{2(k-1)}$. In this work, we develop general structural results that are applicable to resolving the remaining difficult cases $6 \leq k \leq 32$. We apply our results to carefully analyze the structure of 6-critical graphs and use a discharging argument to show that for $\varepsilon_6 = 1/1050$, 6-critical graphs with no $K_4$ subgraph satisfy $|E(G)| \geq ( \frac{k}{2} - \frac{1}{k-1} + \varepsilon_6 ) |V(G)| - \frac{k(k-3)}{2(k-1)}$.

\section{Introduction}
A \emph{$k$-coloring} of a graph $G$ is an assignment $\varphi: V(G) \rightarrow \left\{1,\ldots,k \right\}$ of one of $k$ colors to each vertex of $G$, and is a \emph{proper coloring} if $\varphi(u) \neq \varphi(v)$ for every edge $uv$ of $G$. The \emph{chromatic number} $\chi(G)$ is the smallest integer $k$ for which $G$ has a proper $k$-coloring. A graph $G$ is said to be \emph{$k$-critical} if $\chi(G) = k$, and for every proper subgraph $H$ of $G$, $\chi(H) < k$.

Let $f_k(n)$ denote the minimum number of edges in a $k$-critical graph with $n$ vertices. It has been a long-standing problem to determine the \emph{asymptotic density} of $k$-critical graphs, that is, the value of $\lim\limits_{n \rightarrow \infty} \frac{f_k(n)}{n}$. Early work by Dirac \cite{D1957LMS} established that $k$-critical graphs satisfy $|E(G)| \geq \frac{k-1}{2}|V(G)| + \frac{k-3}{2}$, which was later improved to $|E(G)| \geq (\frac{k-1}{2} + \frac{k-3}{2(k^2-3)})|V(G)|$ by Gallai \cite{G1963HAS,KSW1996DM}. Ore conjectured \cite{O1967} that $f_k(n+k-1) = f_k(n) + (k-1)(\frac{k}{2} - \frac{1}{k-1})$, which would imply that the asymptotic density is $\frac{k}{2} - \frac{1}{k-1}$. This was established by Kostochka and Yancey \cite{KY2014JCTb}, who obtained the following strengthened lower bound.

\begin{THM}[Kostochka and Yancey, \cite{KY2014JCTb}]\label{KY1} If $G$ is $k$-critical, then
$$|E(G)| \geq \left( \frac{k}{2} - \frac{1}{k-1} \right) |V(G)| - \frac{k(k-3)}{2(k-1)}.$$
\end{THM}

\Cref{KY1} implies the asymptotic density conjectured by Ore. Furthermore, Kostochka and Yancey showed that the \emph{only} graphs for which the inequality is tight are a special class of graphs known as \emph{$k$-Ore graphs}, which can be constructed via a sequence of operations known as \emph{Ore-compositions}.

\begin{DEF}\label{DefOreComp}
An \emph{Ore-composition} of two graphs $G_1, G_2$ with respect to an edge $xy \in E(G_1)$ and a vertex $z \in V(G_2)$ is the graph obtained by deleting the edge $xy$, splitting $z$ into two vertices $z_1, z_2$ of positive degree, and identifying $x$ with $z_1$ and $y$ with $z_2$. We refer to $G_1$ as the \emph{edge-side} of the Ore-composition, $G_2$ as the \emph{vertex-side}, and $z$ as the \emph{split-vertex}.

A graph $G$ is \emph{$k$-Ore} if it can be obtained from performing repeated Ore-compositions starting with $K_k$.
\end{DEF}

The class of $k$-Ore graphs has the notable structural property that its members contain many large cliques. Since the $k$-Ore graphs are precisely the $k$-critical graphs of minimum edge density, it is then natural to ask whether the edge density of $k$-critical graphs \emph{not} containing any large cliques is strictly greater than the lower bound (in the limit). This leads to the following conjecture, which was previously posed in \cite{P2015EDM}.
\begin{CONJ}\label{MAINCONJ}
For every $k \geq 4$, there exists $\varepsilon_k > 0$ such that if $G$ is $k$-critical and does not contain a $K_{k-2}$ subgraph, then
$$|E(G)| \geq \left( \frac{k}{2} - \frac{1}{k-1} + \varepsilon_k \right)|V(G)| - \frac{k(k-3)}{2(k-1)}.$$
\end{CONJ}
\Cref{MAINCONJ} has been proved for certain values of $k$. It is vacuously true for $k = 4$, as every 4-critical graph must contain an edge. More is known for $k = 4$; the second author proved in \cite{P2018Manuscript} that \Cref{MAINCONJ} holds for 4-critical graphs of girth 5, and \Cref{KY1} was strengthened for such graphs in \cite{LP2017JGT}. The conjecture has also been proved for the case $k = 5$ in \cite{P2015EDM}. Curiously, the conjecture is also known to be true when $k$ is large; in the thesis \cite{L2015PHD_EMORY} (see also \cite{GLP}), it is shown that \Cref{MAINCONJ} holds for all $k \geq 33$.

This paper aims to lay the groundwork for resolving the remaining cases $6 \leq k \leq 32$, which appear to require more sophisticated and careful analysis. The main elements of this analysis are a modified \emph{potential function} and an operation known as \emph{cloning}, which are used to analyze the local structure of the graph using coarse, global information obtained from measuring the potential. This local information can then be used to obtain global bounds on the number of edges and vertices via \emph{discharging}.

The \emph{potential function} is a key tool used by Kostochka and Yancey \cite{KY2014JCTb} to pass between local and global structure. Their original potential, which we denote $\pk(\cdot)$ and refer to as the \emph{KY-potential},\footnote{to distinguish it from the new $(\varepsilon,\delta)$-potential used in our paper, defined below.} is defined to be
$$\pk(G) = (k-2)(k+1)|V(G)| - 2(k-1)|E(G)|.$$
For $R \subseteq V(G)$, the potential of the subgraph induced by $R$ is defined as $\pk(R) = \pk(G \lb R \rb)$. By rearranging, we immediately observe that \Cref{KY1} is equivalent to the statement that $\pk(G) \leq k(k-3)$ for all $k$-critical graphs $G$.

Now, to account for the presence of large cliques, we define
$$T(G) = \max \{2a(H) + b(H): H \subseteq G \text{ a union of vertex-disjoint cliques} \}$$
where $a(H)$ is the number of components isomorphic to $K_{k-1}$, and $b(H)$ the number of components isomorphic to $K_{k-2}$. We define a modification of the potential function as follows. For fixed $\varepsilon > 0, \delta > 0$, the \emph{$(\varepsilon,\delta)$-potential} (or simply the \emph{potential}) is defined to be
$$p(G) = ( (k-2)(k+1) + \varepsilon)|V(G)| - 2(k-1)|E(G)| - \delta T(G)$$
For $R \subseteq V(G)$, the potential of the subgraph is $p_G(R) = p(G\lb R \rb)$ as before. Note that $p(G) = \pk(G) + \varepsilon |V(G)|-\delta T(G)$. We make the following conjecture.
\begin{CONJ}\label{OURCONJ}
For all $k \geq 6$, there exist $\varepsilon_k, \delta_k, P_k  > 0$ such that the $(\varepsilon_k, \delta_k)$-potential satisfies
\begin{enumerate}
\item $p(K_k) = k(k-3) + k\varepsilon_k - 2\delta_k$, and
\item $p(G) \leq k(k-3) + |V(G)|\varepsilon_k - \left( 2 + \frac{|V(G)| - 1}{k-1}\right)\delta_k$ if $G$ is $k$-Ore and $G \neq K_k$, and
\item $p(G) \leq k(k-3) - P_k$ if $G$ is $k$-critical and not $k$-Ore.
\end{enumerate}
\end{CONJ}
The non-trivial content of this conjecture is the third statement; the first two statements follow from the properties of $k$-Ore graphs and routine calculations. \Cref{OURCONJ} clearly implies \Cref{MAINCONJ}, as a $k$-critical graph $G$ that does not contain a $K_{k-2}$ clique satisfies $T(G) = 0$ and is not $k$-Ore. Rearranging the inequality $p(G) \leq k(k-3) - P_k$ when $T(G) = 0$ yields
$$ |E(G)| \geq \left( \frac{k}{2} - \frac{1}{k-1} +\frac{\varepsilon_k}{2(k-1)} \right) |V(G)| - \frac{k(k-3)}{2(k-1)} + \frac{P_k}{2(k-1)}$$
Moreover, this implies that the asymptotic density of $k$-critical graphs not containing a $K_{k-2}$ subgraph increases to at least $\frac{k}{2} - \frac{1}{k-1} + \frac{\varepsilon_k}{2(k-1)}$.

In this paper, we first develop general techniques for proving \Cref{OURCONJ}, and then apply them to the case $k = 6$. Our main theorem proves \Cref{OURCONJ} when $k=6$ as follows.
\begin{THM}\label{MAINTHM}
For $\varepsilon = \frac{1}{105}, \delta = \frac{10}{105}, P = \frac{20}{21}$, the $(\varepsilon, \delta)$-potential satisfies
\begin{enumerate}
\item $p(K_6) = 18 + 6\varepsilon - 2\delta$, and
\item $p(G) \leq 18 + |V(G)|\varepsilon - \left( 2 + \frac{|V(G)| - 1}{5}\right)\delta$ if $G$ is $6$-Ore and $G \neq K_6$, and
\item $p(G) \leq 18 - P$ if $G$ is 6-critical and not 6-Ore.
\end{enumerate}
\end{THM}

\subsection{Organization}

This paper is organized as follows. In \Cref{sec:kOre}, we prove several facts about $k$-Ore graphs and Ore-compositions. In \Cref{sec:kiersteadrabern}, we use list colorings to study independent sets of degree $k-1$ vertices. In \Cref{sec:potential}, we use the potential function to study the general structural properties of $k$-critical graphs that are `close' to violating \Cref{OURCONJ}. In \Cref{sec:cloning}, we define a notion called \emph{cloning} and develop its properties. The results in \Cref{sec:potential,sec:cloning} apply to all $k \geq 6$. In \Cref{sec:6critical}, we apply the general results of the previous sections to 6-critical graphs. In \Cref{sec:discharging}, we complete the proof of \Cref{MAINTHM} using discharging.

\subsection{Notation and Conventions}

We adopt the following conventions. Unless stated otherwise, a \emph{graph} is a simple graph, and a \emph{coloring} is a proper coloring. For $R \subseteq V(G)$, $G \lb R \rb$ denotes the subgraph induced by the vertices in $R$ and $\partial_G R$ denotes the \emph{boundary} of $R$ in $G$ that is $\{v\in R: N(v)\setminus R\ne \emptyset\}$. The edge between vertices $u,v$ will be denoted $uv$. The set of neighbors of the vertex $x$ is denoted $N(x)$, and the closed neighborhood $\{x\} \cup N(x)$ is denoted $N\lb x \rb$. 

For a set $S$ consisting of pairs $\{u,v\}$ of vertices of $G$, $G + S$ denotes the graph obtained by adding the edges in $S$ to $G$. When $S$ consists of a single edge $uv$, we simply write $G + uv$. For a vertex $v \in V(G)$, $G - v$ denotes the graph obtained by deleting the vertex $v$. For $A, B \subseteq V(G)$, we use $E_G(A,B)$ to denote the set of edges $ab$ with $a \in A, b \in B$. When $G$ is unambiguous, we simply write $E(A,B)$.

If $G$ is an Ore-composition of two graphs $G_1, G_2$ on the vertices $\{a,b\}$, then we refer to $a,b$ as \emph{overlap vertices}, and $\underline{ab}$ denotes the vertex that was split in $G_2$ (the \emph{split-vertex}).

\section{$k$-Ore Graphs and Gems}\label{sec:kOre}

We first study the properties of $k$-Ore graphs and a related class of graphs called \emph{gems}, which appear when considering counterexamples to \Cref{OURCONJ}. We will prove the first two (easy) statements of \Cref{OURCONJ} and \Cref{MAINTHM}, as well as several lemmas which will be needed later.

\subsection{$k$-Ore Graphs}

A graph $H$ obtained from a $k$-Ore graph $G$ by splitting a vertex $v$ of $G$ into two vertices $a,b$ of positive degree will be called a \emph{split $k$-Ore graph} and we call $a$ and $b$ the \emph{split vertices} of $H$. 

Recall that a $k$-Ore graph is obtained by a sequence of Ore-compositions (\Cref{DefOreComp}) applied to $K_k$. Suppose $G$ is an Ore-composition of $G_1, G_2$ on $xy \in E(G_1), z \in V(G_2)$. We can view $G$ as a graph $G'$ isomorphic to $G_1$, with the edge $xy \in E(G')$ corresponding to the subgraph of $G$ which is isomorphic to $G_2$ with $z$ split. The edge $xy \in E(G')$ is called a \emph{replacement edge} of the Ore-composition $G$ when viewing $G$ as $G'$. See \Cref{fig:frame} for an example where $G_1 = K_k$ and the edge $v_3v_4$ has been replaced.

\begin{figure}
	\centering
	\begin{tikzpicture}[transform shape]
	\foreach \x in {1,...,6}{%
		\pgfmathparse{(\x+2)*360/6+30}
		\pgfmathtruncatemacro\lab{7-\x}
		\node[draw,circle,inner sep=0.05cm] (N-\x) at (\pgfmathresult:1.6cm) {$v_\lab$};
	} 
	\node[inner sep=0.05cm] (G2) at (\the\numexpr 6*360/6:1.4cm){\tiny $G_2$};
	\foreach \x [count=\xi from 1] in {1,...,6}{%
		\foreach \y in {\x,...,6}{%
			\ifboolexpr{ test {\ifnumcomp{\x}{=}{3}} and test {\ifnumcomp{\y}{=}{4}}}
			{}
			{\path (N-\x) edge[ultra thin,-] (N-\y);}
		}
	}
	%\path (N-3) edge[ultra thin, draw=red,-] (N-4);
	\draw[bend right=25, -, dashed, red] (N-3) to (N-4);
	\draw[bend left=25, -, dashed, red] (N-3) to (N-4);
	\draw[thin] (N-3) to (G2);
	\draw[bend left=15, thin] (N-3) to (G2);
	\draw[bend right=15, thin] (N-3) to (G2);
	\draw[thin] (N-4) to (G2);
	\draw[bend left=15, thin] (N-4) to (G2);
	\draw[bend right=15, thin] (N-4) to (G2);
	\end{tikzpicture}
	\caption{The frame of an Ore-composition.}
	\label{fig:frame}
\end{figure}
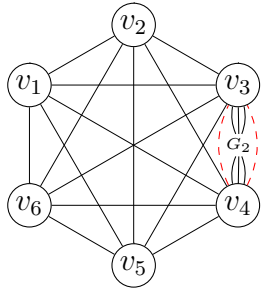

When $G_1 = G_2 = K_k$, the Ore-composition of $G_1$ and $G_2$ is isomorphic to $K_k$ with a replacement edge. This is true for $k$-Ore graphs in general, which motivates the following definition.
\begin{DEF}
A \emph{frame} $F$ for the $k$-Ore graph $G$ is a set of $k$ vertices such that $G$ is isomorphic to the complete graph $K_k$ on $F$, with some edges possibly being replacement edges.
\end{DEF}

\begin{LMA}\label{kOreAsClique}
Every $k$-Ore graph has a frame.
\end{LMA}
\begin{proof}
We proceed by induction on $|V(G)|$. If $|V(G)| = k$, the only $k$-Ore graph is $G = K_k$ which is a frame of itself as desired. So we may assume that $G$ is the Ore-composition of $k$-Ore graphs $G_1, G_2$ on $xy \in E(G_1)$. Since $|V(G_1)| < |V(G)|$, the induction hypothesis implies $G_1$ has a frame $F = \{v_1,\ldots,v_k\}$. 
	
We claim that $F$ is a frame for $G$ as follows. If $x,y \in F$, then $F$ is still a frame for $G$, with an additional replacement edge between $x,y$. Otherwise, $xy$ is an edge contained in the split $k$-Ore graph between $v_i,v_j$, which remains a split $k$-Ore graph when $xy$ is replaced. Since replacing an edge in a split $k$-Ore graph results in another $k$-Ore graph, $F$ is still a frame for $G$.
\end{proof}

Given a $k$-Ore graph $H$, we use $V^\ast(H)$ to denote a frame of $H$, which may be a particular frame if specified, or if not, an arbitrarily chosen frame of $H$.

Ore-compositions preserve (almost all) cliques, so we have a strong bound on the number of large cliques in the resulting graph as the following lemma shows.

\begin{LMA}\label{Tsuperadd}
If $G$ is an Ore-composition of $G_1, G_2$, then $T(G) \geq T(G_1) + T(G_2) - 2$. Moreover, if $G_1 = K_k$ or $G_2 = K_k$, then $T(G) \geq T(G_1) + T(G_2) - 1$. 
\end{LMA}
\begin{proof}
Let $e$ be the replaced edge of $G_1$, and $z$ the split-vertex of $G_2$. We clearly have $T(G) \geq T(G_1 - e) + T(G_2 \setminus \{z\}) \geq T(G_1) - 1 + T(G_2) - 1$. The second statement follows as $T(K_k - e ) = T(K_k \setminus \{z\}) = T(K_k)$.
\end{proof}

\begin{LMA}\label{kOreTBound}
If $G$ is a $k$-Ore graph and $G \neq K_k$, then $T(G) \geq 2 + \frac{|V(G)| - 1}{k-1}$. 
\end{LMA}
\begin{proof}
We proceed by induction on $|V(G)|$. Since $G \neq K_k$, $G$ is the Ore-composition of two $k$-Ore graphs $G_1, G_2$. If $G_1 = G_2 = K_k$, then $T(G) = 4 \geq 2 + \frac{2k-2}{k-1}$ as desired. If $G_1 \neq K_k$ and $G_2 = K_k$, the induction hypothesis applies to $G_1$, and we have $T(G) \geq T(G_1) + 1 = 2 + \frac{ |V(G)| - 1}{k-1}$ by \Cref{Tsuperadd} as desired. Similarly this also holds if $G_1 = K_k$ and $G_2 \neq K_k$.

So we may assume that $G_1, G_2 \neq K_k$. By induction $T(G_i) \geq 2 + \frac{ |V(G_i)| - 1}{k-1}$. \Cref{Tsuperadd} implies that $T(G) \geq \left( 2 + \frac{ |V(G_1)| - 1}{k-1} \right) + \left( 2 + \frac{ |V(G_2)| - 1}{k-1} \right) - 2 = 2 + \frac{ |V(G)| - 1}{k-1}$ as desired.
\end{proof}

Note that the second statements of \Cref{MAINTHM} and \Cref{OURCONJ} follow from \Cref{kOreTBound}.

\subsection{Gems}
In minimal counterexamples to \Cref{OURCONJ}, we are interested in studying certain subgraphs called \emph{gems}, which occur if the graph $G$ resembles an Ore-composition.

\begin{DEF}
A subgraph $D \subseteq G$ is a \emph{diamond} if $G\lb D\rb = K_k - uv$ and $d_G(x) = k-1$ for every $x \in V(D) \setminus \{u,v\}$. The vertices $u,v$ are called the endpoints of the diamond.

A subgraph $D \subseteq G$ is an \emph{emerald} if $G\lb D\rb = K_{k-1}$ and $d_G(x) = k - 1$ for every $x \in V(D)$.

A \emph{gem} is a diamond or emerald.
\end{DEF}

\begin{LMA}\label{kOreGem}
If $G$ is a $k$-Ore graph and $v \in V(G)$, then there exists a gem of $G$ not containing $v$.
\end{LMA}
\begin{proof}
\begin{proof}
We proceed by induction on $|V(G)|$. If $G = K_k$, then $G \setminus \{v\}$ is an emerald of $G$ not containing $v$. Otherwise, $G$ is the Ore-composition of two $k$-Ore graphs $G_1, G_2$ ($G_1$ the edge-side) with overlap vertices $\{a,b\}$. Let $G_1, G_2$ be chosen such that $|E(G_1)|$ is minimized. If $v \in V(G_1)$, then by induction, there is a gem $F$ in $G_2$ not containing the vertex $\underline{ab} \in V(G_2)$, but then $F$ is also a gem of $G$ as desired.
	
Hence, we may assume that $v \in V(G_2) \setminus V(G_1)$. If $G_1 = K_k$, then $G \lb V(G_1) \rb$ is a diamond of $G$ not containing $v$. Otherwise, $G_1$ itself is an Ore-composition of two $k$-Ore graphs $H_1, H_2$ with overlap vertices $\{c,d\}$ where $H_1$ denotes the edge-side of the composition. If the edge $ab \in E(G_1)$ is an edge of $H_2$, then there is an Ore decomposition of $G$ with $H_1$ as the edge-side, contradicting the minimality of $G_1$. Thus $ab \in E(H_1)$. By induction, there is a gem $F_2$ in $H_2$ not containing $\underline{cd} \in V(H_2)$. Since $ab \in E(H_1)$, $F_2$ is also a gem of $G$ and $F_2$ does not contain $v$ as desired.
\end{proof}
\end{proof}

\begin{LMA}\label{kOreExcludeGem}
If $G$ is $k$-Ore and $D = K_{k-1}$ is a subgraph of $G$, then either $G = K_k$ or there exists a diamond or emerald of $G$ disjoint from $D$.
\end{LMA}
\begin{proof}
We proceed by induction on $|V(G)|$. If $G = K_k$ we are done, so suppose that $G$ is the Ore-composition of two $k$-Ore graphs $G_1, G_2$ with overlap vertices $\{a,b\}$. Let the composition be chosen such that $|V(G_1)|$ is minimized. Note that since $D$ is a clique, $D$ must lie entirely on either the edge-side or the vertex-side of the Ore-composition. If $V(D) \subseteq V(G_1)$, then by \Cref{kOreGem}, there is a gem $F$ in $G_2$ not containing the vertex $\underline{ab} \in V(G_2)$. Now $F$ is a gem of $G$ disjoint from $D$ as desired.

Thus we may assume that $D$ is on the vertex-side. There are two cases, depending on whether $V(D)$ is disjoint from $\{a,b\}$. First, suppose that $V(D) \cap \{a,b\} = \emptyset$. If $G_1 = K_k$, then $G_1$ is a diamond disjoint from $D$ as desired. Otherwise, $G_1$ is the composition of two $k$-Ore graphs $H_1, H_2$ with overlap vertices $\{c,d\}$. If the edge $ab \in E(G_1)$ is in $E(H_2)$, then there is an Ore decomposition of $G$ with $H_1$ as the edge-side, contradicting the minimality of $G_1$. Thus $ab \in E(H_1)$ and by \Cref{kOreGem}, there is a gem $F_2$ in $H_2$ not containing $\underline{cd} \in V(H_2)$. Now $F_2$ does not contain $a,b$ as $a,b \in V(H_1)$, and is therefore a gem in $G$ disjoint from $D$ as desired.

So we may suppose that $V(D)$ and $\{a,b\}$ are not disjoint. At most one of $a,b$ can be a vertex of $D$, so we may assume without loss of generality that $a \in V(D)$. If $G_2 \neq K_k$ then by induction, there is a gem $F_2'$ in $G_2$ disjoint from $D$.  Now $F_2'$ is disjoint from $D$ in the graph $G$ as desired. 

Thus we may assume that $G_2 = K_k$. As $d_{G_2}(\underline{ab}) = k-1$, it follows that $b \in V(G)$ has exactly one neighbor in $V(G_2)$. By \Cref{kOreGem}, there is a gem $D'$ in $G_1$ disjoint from $a$. If $D'$ is a diamond, then $D'$ is also a diamond of $G$ disjoint from $D$, as $b \in V(D')$ implies that $b$ is an endpoint of $D'$. If $D'$ is an emerald that does not contain $b$, then $D'$ is an emerald in $G$ disjoint from $D$. If $D'$ is an emerald with $b \in V(D')$, then $D'$ is still an emerald of $G$, as $b$ is not adjacent to $a$ in $G$ and has exactly one neighbor on the vertex-side $V(G_2)$. In any case then, $D'$ is an emerald of $G$ disjoint from $D$ as desired.
\end{proof}

\section{Colorings and Independent Sets}\label{sec:kiersteadrabern}
In this section, our goal is to obtain an upper bound on the number of vertices of degree $k$ which are adjacent to vertices of degree $k-1$ in a $k$-critical graph. This will be important when discharging.

We briefly introduce several new definitions which are needed for this section. Given a function $L(v)$ which associates a list of colors to each vertex $v$, an assignment $\varphi: V(G) \rightarrow \mathbb{N}$ is an \emph{$L$-coloring} if $\varphi(v) \in L(v)$ for all $v$, and $\varphi(u) \neq \varphi(v)$ for every edge $uv$. Given a function $f: V(G) \rightarrow \mathbb{N}$, the graph $G$ is \emph{$f$-choosable} if an $L$-coloring exists for every function $L(v)$ with $|L(v)| = f(v)$ for all $v$.

In our discharging argument (\Cref{subs:global}), it is necessary to bound the number of edges emanating from independent sets consisting of degree $k-1$ vertices. To do so, we employ the following lemma of Kierstead and Rabern \cite{KR2017JGT}:
\begin{LMA}[Kierstead and Rabern, \cite{KR2017JGT}]\label{KRmain}
Let $G$ be a non-empty graph and let $f:V(G) \rightarrow \mathbb{N}$ with $f(v) \leq d_G(v) + 1$ for all $v \in V(G)$. If there is an independent set $A \subseteq V(G)$ such that
$$|E(A,V(G))| \geq \sum_{v \in V(G)} d_G(v) + 1 - f(v)$$
then $G$ has a non-empty induced subgraph $H$ that is $f_H$-choosable, for $f_H(v) = f(v) + d_H(v) - d_G(v)$.
\end{LMA}

We are interested in the case where $G$ is $k$-critical and $A$ consists of vertices of degree $k - 1$.

\begin{LMA}\label{indsetedges}
Let $G$ be a $k$-critical graph, and let $A \subseteq V(G)$ be an independent set with $d_G(v) = k- 1$ for all $v \in A$. Let $B_i$ be the set of vertices $v$ that have a neighbor in $A$ such that $d_G(v) = k-1+i$. Then for any $\ell \geq 1$, defining $B = \bigcup_{i=1}^\ell B_i$, we have
$$|E_G(A,B)| \leq |A| + \sum_{i = 1}^\ell (i+1) |B_i|.$$
\end{LMA}
\begin{proof}
Suppose to the contrary that $G$ contains an independent set $A$ violating the desired inequality. Let $G'$ be the subgraph induced by $A$ and $B$, so $G' = G\lb A \cup B\rb$. Define $f$ on $V(G')$ by $f(v) = d_{G'}(v)$ for $v \in A$, and $f(v) = \max\{d_{G'}(v) - i, 0\}$ for $v \in B_i$. Then we have
$$|E_{G'}(A,B)| > |A| + \sum_{i=1}^\ell (i+1)|B_i| \geq \sum_{v \in V(G')} d_{G'}(v) + 1 - f(v).$$
Hence the \Cref{indsetedges} applies to $G'$. Thus, we can find a non-empty induced subgraph $H$ of $G'$ (and hence of $G$) which is $f_H$-choosable for $f_H(v) = f(v) + d_H(v) - d_{G'}(v)$.

Since $G$ is $k$-critical, there exists a $(k-1)$-coloring $\varphi$ of $G \setminus H$. For each $v \in V(H)$, let $g(v)$ be the number of colors in $\{1,\ldots,k-1\}$ that are not used by the neighbors of $v$ in $\varphi$. It suffices to show that $g(v) \geq f_H(v)$ for all $v \in V(H)$, since then $\varphi$ could be extended to a $(k-1)$-coloring of $G$ by using the $f_H$-choosability of $H$, thus contradicting that $G$ is $k$-critical. Let $v \in V(H)$ be a vertex in $B_i$. Clearly $v$ has at most $d_G(v) - d_H(v)$ neighbors in $G \setminus H$, so $g(v) \geq \max \{ (k-1) - (d_G(v) - d_H(v)), 0\} = \max \{ d_H(v) - i, 0\}$. Since $d_H(v) - d_{G'}(v) \leq 0$, a routine calculation shows that $\max\{d_H(v) - i, 0\} \geq \max \{d_{G'}(v) - i, 0\} + d_H(v) - d_{G'}(v) = f_H(v)$. Therefore $g(v) \geq f_H(v)$, completing the proof.
\end{proof}

In \Cref{sec:discharging}, we will apply the following special case of \Cref{indsetedges}.
\begin{CORO}\label{A2B}
Let $G$ be a $k$-critical graph, and let $A \subseteq V(G)$ be an independent set with $d_G(v) = k-1$ for all $v \in A$. Let $B$ be the set of neighbors of $A$ of degree $k$. Then $|E_G(A,B)| \leq |A| + 2|B|$.
\end{CORO}

\section{Applications of the Potential Function}\label{sec:potential}

Recall the definition of the $(\varepsilon,\delta)$-potential function as follows.
\begin{DEF}
The \emph{potential} of a graph $G$ is given by
$$p(G) = ((k-2)(k+1) + \varepsilon)|V(G)| - 2(k-1)|E(G)| - \delta T(G)$$
For $R \subseteq V(G)$, the potential $p_G(R)$ is $p(G \lbrack R \rbrack )$.
\end{DEF}

This section is devoted to analyzing the local structure of $k$-critical graphs for $k \geq 6$, by measuring the contribution of edges and vertices to the potential in substructures of the graph. The key result is that, broadly speaking, $k$-critical graphs with high potential cannot contain any `nearly' $k$-critical graphs as substructures, which can be made $k$-critical by adding a small number of edges. To make this precise, we state a few definitions.

\begin{DEF}
We say a graph $H$ is \emph{smaller} than $G$ if any of the following hold:
\begin{enumerate}
\item $|V(H)| < |V(G)|,$ or
\item $|V(H)| = |V(G)|$ and $|E(H)| < |E(G)|,$ or
\item $|V(H)| = |V(G)|$, $|E(H)| = |E(G)|$, and $H$ has more pairs of vertices with the same closed neighborhood than $G$.
\end{enumerate}
We say a $k$-critical graph $G$ is \emph{good} if every graph smaller than $G$ satisfies \Cref{OURCONJ}. We say $G$ is \emph{$i$-tight} if $G$ is good and $p(G) > k(k-3) - P - Q + i\delta$, for an auxiliary constant $Q$ to be defined in \Cref{sec:potential_assumptions}. We say $G$ is \emph{tight} if $G$ is 0-tight.
\end{DEF}

Hence tight graphs are $k$-critical graphs which are close to being counterexamples to \Cref{OURCONJ}.

\begin{DEF}
An \emph{edge-addition} in $G$ is a non-edge $xy$ such that $G + xy$ contains a $k$-critical subgraph $H$ with $V(H) \subsetneq V(G)$.
	
An \emph{$i$-edge-addition} in $G$ is a set $S$ of at most $i$ non-edges such that $G + S$ contains a $k$-critical subgraph $H$ with $V(H) \subsetneq V(G)$.
\end{DEF}

We will show (as a part of the more general \Cref{buckets}) the following.

\begin{LMA}
Let $G$ be a tight, ungemmed graph. Then for $1 \leq i \leq (k-2)/2$, $G$ does not admit an $i$-edge-addition.
\end{LMA}

The next results demonstrate how the nonexistence of edge additions can be used to make strong deductions about the local properties of the graph. The following lemmas will be used frequently in subsequent sections.

\begin{LMA}\label{commonNeighbors}
Suppose that $G$ does not admit an $i$-edge-addition for some $i\ge 1$. Let $x$ be a vertex of degree $k-1$ and $y$ a neighbor of $x$ of degree at most $k - 2 + i$. If $|N(x)\cap N(y)|\ge k-3$, then $N\lb x\rb \subseteq N\lb y\rb$.
\end{LMA}
\begin{proof}
Let $H = G \lb N\lb x\rb \cap N\lb y\rb \rb$, the subgraph induced by $x,y$ and their common neighbors $N(x)\cap N(y)$. If $|N(x)\cap N(y)|\ge k-2$, then $N[x]=N[y]$ as desired. So we may assume that $|N(x)\cap N(y)|=k-3$. Thus $x$ has exactly one neighbor $a$ not in $H$. Let $b_1,\ldots, b_i$ be the neighbors of $y$ not in $H$. Let $S$ be the set of edges $\{ ab_1, \ldots, ab_i\}$ and suppose that $G \setminus \{x\} + S$ has a $(k-1)$-coloring $\varphi$. 
	
Let $L = \{\varphi(v): v\in N(x)\}$. If $|L| \leq k-2$, then $\varphi$ extends to a $(k-1)$-coloring of $G$ by letting $\varphi(x) \in [k-1]\setminus L$. But then $S$ is an $i$-edge-addition, contrary to our assumption. 
	
So we may suppose that $|L|=k-1$. Define a new $(k-1)$-coloring $\psi$ by letting $\psi(x) = \varphi(y), \psi(y) = \varphi(a)$, and $\psi(z) = \varphi(z)$ for all $z \in V(G) \setminus \{x,y\}$. Since $\varphi(b_1), \ldots, \varphi(b_i)$ are all distinct from $\varphi(a)$, $\psi$ is a $(k-1)$-coloring of $G$, a contradiction. But then $S$ is an $i$-edge-addition, contrary to our assumption.
\end{proof}

\begin{CORO}\label{emeraldCluster}
Suppose $G$ does not admit edge-additions. If $x$ and $y$ are adjacent vertices of degree $k-1$ with $|N(x)\cap N(y)|\ge k-3$, then $N\lb x\rb = N\lb y\rb$.
\end{CORO}

\begin{CORO}\label{kNeighbor}
Suppose $G$ does not admit 2-edge-additions. If $x$ and $y$ are adjacent vertices with $d(x) = k-1, d(y) = k$, and $|N(x)\cap N(y)| \ge k-3$, then $N\lb x\rb \subseteq N\lb y\rb$. 
\end{CORO}

\subsection{Assumptions}\label{sec:potential_assumptions}

Let $k \geq 6$ be fixed. We subsequently write $\varepsilon, \delta, P$ for the constants $\varepsilon_k, \delta_k, P_k$ in \Cref{OURCONJ}. Unless stated otherwise, $G$ is always $k$-critical.

In this section and following sections, we will assume that $\varepsilon, \delta, P$, and auxiliary constants $Q, \buc, \gap$, satisfy the conditions of the forthcoming Assumption 1. These inequalities may not be sufficient to determine precise values of $\varepsilon_k, \delta_k, P_k$ for which \Cref{OURCONJ} holds for a specific $k$, as additional restrictions on $\varepsilon_k$ may arise during the discharging step. However, so long as the following inequalities are satisfied, the results of \Cref{sec:potential,sec:cloning} will hold for any $k \geq 6$.

\textbf{Assumption 1}
All of the following inequalities hold:
\begin{enumerate}
\item $\delta = 2(k-1)\varepsilon$, and
\item $2 + \delta \leq \gap \leq k - 2$, and
\item $(\gap - 2) + Q + k\delta \leq P$, and
\item $P + Q + \frac{k}{2} \delta \leq \buc$, and
\item $\buc + (k-1)\delta \leq 2$.
\end{enumerate}
In general, we are interested in maximizing the value of $\varepsilon$. However, the calculations are easier to follow when stated in terms of these constants, so we defer expressing them in terms of $\varepsilon$ only.

\subsection{Analysis of Potential}

\begin{DEF}
Suppose $G$ is a $k$-critical graph, and $R \subsetneq V(G)$ a proper subset of the vertices. By definition, there exists a $(k-1)$-coloring $\varphi: R \rightarrow \{1,\ldots,k-1\}$ of $G \lb R\rb$. We define $G_{R,\varphi}$ to be the graph obtained by identifying all vertices of the same color in $\varphi$ to a single vertex $x_i$ for all $i\in \{1,\ldots, k-1\}$, adding the edges $x_ix_j$ for all $i\ne j\in \{1,\ldots,k-1\}$, and removing multiple edges (to obtain a simple graph). 
	
Note that $G_{R,\varphi}$ is not $(k-1)$-colorable, or else $G$ would have been $(k-1)$-colorable. Thus $G_{R,\varphi}$ contains a $k$-critical subgraph, which necessarily includes at least one of the vertices $x_1,\ldots,x_{k-1}$. Let $W$ be such a subgraph. Let $X = V(W) \cap \{x_1,\ldots,x_{k-1}\}$ and $R' = R \cup (V(W) - X)$. Note that $R \subsetneq R' \subseteq V(G)$. We say $R'$ is a \emph{critical extension} of $R$ %(with respect to $\phi$) 
with \emph{extender} $W$ and \emph{core} $X$.
\end{DEF}

\begin{DEF}
Let $R'$ be a critical extension of $R \subsetneq V(G)$ with extender $W$ and core $X$. If $R' = V(G)$, we say $R'$ is \emph{spanning}. If $|E(G \lb R' \rb)| = |E(G \lb R \rb)| + |E(W)| - |E(K_X)| + i$, we say $R'$ is \emph{$i$-incomplete}. We say $R'$ is \emph{complete} if $i = 0$.
\end{DEF}

\begin{LMA}\label{submodular}%Thesis 4.6
If $R'$ is an $i$-incomplete critical extension of $R \subsetneq V(G)$ with extender $W$ and core $X$, then
$$p_G(R') \leq p_G(R) + p(W) - 2i(k-1) - (p(K_X) + \delta T(K_X) - \delta |X|).$$
\end{LMA}
\begin{proof}
We count the contribution of vertices, edges, and $T(\cdot)$ to each side. Since $R' = R \cup(V(W) - X)$, each vertex of $R'$ is counted exactly once on the right side. The number of edges on the right is, by definition, $|E(G \lb R' \rb)| - i$, so the net contribution from edges is $-2i(k-1)$ on the right side. Note that edges within $G \lb R \rb$ and $W - X$ are counted exactly once on each side, and each edge between $X$ and $V(W) - X$ in $G_{R, \varphi}$ corresponds to at least one distinct edge between $R$ and $V(W) - X$ in $G$, so $i \geq 0$. Finally, $T(G \lb R' \rb) \geq T(G \lb R \rb) + T(W - X)$, and $T(W - X) \geq T(W) - |X|$, as $W$ contains at most $|X|$ more disjoint cliques than $W - X$. Adding these inequalities proves the lemma.
\end{proof}

Our next lemma shows that potential decreases under extensions.

\begin{LMA}\label{edgedrop}
Suppose $G$ is good, and $R \subsetneq V(G)$ with $G \lb R \rb$ not a clique. If $R'$ is an $i$-incomplete critical extension of $R$, then $p_G(R') \leq p_G(R) - 2(i+1)(k-1) - \delta$.
\end{LMA}
\begin{proof}
Since $G \lb R\rb$ is not a clique, $W \neq G$ and so $W$ is smaller than $G$. As $G$ is good, we find that $W$ satisfies \Cref{OURCONJ}.
	
By a routine calculation,
$$p(K_X) + \delta T(K_X) - \delta |X| = -|X| ( (k-1)|X| - (k^2 - 3 + \varepsilon - \delta)),$$
which is quadratic in $|X|$ with roots $\{0, \frac{k^2 - 3 + \varepsilon - \delta}{k-1}\}$. Since $\varepsilon - \delta \geq 3 - k$, this quadratic is minimized over $|X| \in \{1,\ldots, k -1\}$ at $|X| = 1$. Thus
$$p(K_X) + \delta T(K_X) - \delta|X| \ge p(K_1) - \delta = (k-1)\varepsilon - \delta.$$
	
Now, we condition on three possibilities for $W$. First suppose that $W = K_k$. Observe that \Cref{submodular} holds without adding the term $\delta |X|$, as $T(W) = T(W - x)$ for every vertex $x \in X$. We then have
$$p_G(R') \leq p_G(R) - 2i(k-1) + p(K_k) - p(K_1) = p(R) - 2(i+1)(k-1) + (k-1)\varepsilon - 2\delta.$$
Since $\delta \geq (k-1)\varepsilon$ by Assumption 1.1, the desired inequality holds.
	
Next suppose that $W$ is $k$-Ore and $W \neq K_k$. Using \Cref{kOreTBound} and the fact that $\delta = 2(k-1)\varepsilon$ by Assumption 1.1, we have
\begin{align*}
p_G(R') &\leq p_G(R) - 2i(k-1) \\
&\offplus + \left( k(k-3)  + |V(G)|\varepsilon - \left( 2 + \frac{|V(G)| - 1}{k-1}\right) \delta \right) - (k^2 - k + \varepsilon - \delta) \\
&= p_G(R) - 2(i+1)(k-1) - \delta - (|V(G)| - 1)\varepsilon \\
& \leq p_G(R) -2(i+1)(k-1) - \delta.
\end{align*}
	
Finally suppose that $W$ is not $k$-Ore. Now $p(W) \leq k(k-3) - P$ and
$$p_G(R') \leq p_G(R) -2(i+1)(k-1) - (P + \varepsilon - \delta)$$
and we obtain the desired inequality since $P \geq 2\delta$ by Assumption 1.2.
\end{proof}

\Cref{edgedrop} implies that every proper subgraph $G \lbrack R \rbrack$ has higher potential than $G$, as the potential decreases by iterating critical extensions.

\subsection{Collapsible Subsets}

The following definition is useful for characterizing the proper subsets of $G$ which have the lowest potential.

\begin{DEF}
A subset $R \subsetneq V(G), |R| \geq 2$ is \emph{collapsible} in $G$ if for every proper $(k-1)$-coloring of $G \lb R\rb$, the vertices of $\partial_G R$ receive the same color.
	
A subset $R \subsetneq V(G), |R| \geq 2$ is \emph{$i$-collapsible} in $G$ if for all $(k-1)$-colorings $\varphi: R \rightarrow C = \{1,\ldots,k-1\}$ of $G \lb R \rb$,
$$\min_{1 \leq c \leq k - 1} |E(\varphi^{-1}(C \setminus \{c\}) \cap R, V(G) \setminus R)| \leq i.$$
	
It is easy to see that collapsibility corresponds to 0-collapsibility.
\end{DEF}

We now provide a characterization of collapsible subsets.

\begin{PROP}\label{prop:COLLAPSE_EQUIV}%Thesis Proposition 4.9
A subset $R \subsetneq V(G), |R| \geq 2$ is collapsible in $G$ if and only if every critical extension has a core of size 1, is spanning, and is complete.
\end{PROP}
\begin{proof}
Suppose that $R$ is collapsible in $G$. Let $R'$ be a critical extension of $R$ with extender $W$ and core $X$ with respect to a $(k-1)$ coloring $\phi$ of $G \lb R \rb$. Without loss of generality we may assume that $\phi(\partial_G R)=1$. By definition of collapsibility, $x_1$ is a cut-vertex of $G_{R,\varphi}$. Let $A = (V(G) \setminus R) \cup \{x_1\}$. A $k$-critical graph has no cut-vertex, and so $W$ is a subgraph of $G_{R,\varphi} \lb A \rb$ and $X = \{x_1\}$.
	
Now suppose that $R' \neq V(G)$. Let $y \in V(G) \setminus R'$. Since $G$ is $k$-critical, there exists a $(k-1)$-coloring $\psi$ of $G - y$. $\psi$ induces a $(k-1)$-coloring of $R$, and we may assume without loss of generality that $\psi(\partial_G R) = 1$. Now $\psi$ induces a $(k-1)$-coloring of $W$, a contradiction. Hence we may assume that $R' = V(G)$ and thus the extension is spanning. It only remains to show that the extension is complete. 
	
Next suppose that $G \lb R' \setminus R\rb$ has an edge $f$ not in $W \lb R' \setminus R\rb$. Since $G$ is $k$-critical, there exists a $(k-1)$-coloring $\psi$ of $G - f$. This induces a $(k-1)$-coloring of $W$, a contradiction. So we may assume that $E(G \lb R' \setminus R\rb) = E(W \lb R' \setminus R\rb)$. 
	
Since the extension is not complete, we may now assume that there exists an edge $ux_1 \in E(W)$ that corresponds to two edges $uv_1, uv_2 \in E(G)$, with $v_1, v_2 \in R$. Since $G$ is $k$-critical, there exists a $(k-1)$-coloring $\psi$ of $G - uv_1$. Since $R$ is collapsible, $\psi(v_1) = \psi(v_2)$. Moreover, $\psi(u_1)$ is a distinct color from $\psi(u)$ and thus $\psi$ extends to a $(k-1)$-coloring of $G$, a contradiction. Therefore $R'$ is complete. Since $R'$ was arbitrary, it follows that every critical extension of $R$ has a core of size 1, is spanning, and is complete as desired.
	
We now prove the converse. Suppose to the contrary that every critical extension of $R$ has a core of size 1, is spanning, and is complete, but that there exists a $(k-1)$-coloring $\varphi$ of $G \lb R\rb$ with $\varphi(u) = 1, \varphi(v) = 2$ for some $u,v \in \partial_G R$. Let $W$ be a $k$-critical subgraph of $G_{R,\varphi}$ and let $R'$ be a critical extension of $R$ with extender $W$ and core $X$. Since $|X| = 1$, we may assume that $x_2 \notin X$. Since $v \in \partial_G R$ and $R'$ is spanning, there is an edge $vz \in E(G)$ between $R$ and $V(G) \setminus R$ with no corresponding edge in $W$. Thus $R'$ is at least 1-incomplete, a contradiction.
\end{proof}

One of the implications of \Cref{prop:COLLAPSE_EQUIV} holds for an analogue for $i$-collapsible subsets as follows.

\begin{LMA}\label{collapsibility} %Thesis Prop. 4.17
If $R \subsetneq V(G)$ has the property that all critical extensions of $R$ are spanning, have a core of size 1, and are at most $i$-incomplete, then $R$ is $i$-collapsible.
\end{LMA}
\begin{proof}
Let $\varphi$ be any $(k-1)$-coloring of $G \lb R\rb$. Let $W$ be a $k$-critical subgraph of $G_{R,\varphi}$ and let $R'$ be a critical extension of $R$ with extender $W$ and core $X$. By our hypothesis, $R' = V(G)$ and $|X|=1$. We may assume without loss of generality that $X=\{x_1\}$. Now an edge between $E(\varphi^{-1}(C \setminus \{1\}) \cap R$ and $V(G) \setminus R$ is in $|E(G \lb R' \rb)|$ but not in $|E(G \lb R \rb)| + |E(W)| - |E(K_X)|$. As $R'$ is at most $i$-incomplete, it follows that there are at most $i$ such edges. Then for $c = 1$,
$$|E(\varphi^{-1}(C \setminus \{c\}) \cap R, V(G) \setminus R)| \leq i,$$
which proves that $R$ is $i$-collapsible.
\end{proof}

\subsection{Edge Additions}

The following lemma shows that such a critical subgraph arising from an $i$-edge-addition yields a subset of $V(G)$ whose potential is small.

\begin{LMA}\label{addUpperBound}
Suppose $G$ is tight and admits an $i$-edge-addition. Let $H$ be a $k$-critical subgraph of $G + S$ and let $R = V(H)$. Then
$$p_G(R) \leq p(G) + P + Q + 2i(k-1) + i\delta.$$
Furthermore, if $H$ is not $k$-Ore, then $p_G(R) \leq p(G) + Q + 2i(k-1) + i\delta$.
\end{LMA}
\begin{proof}
Since $|E(H)\setminus E(G[R])|\le i$, it follows that $T(H) \leq T(G \lb R \rb) + i$. Hence $p(H) \geq p_G(R) - 2i(k-1) - i\delta$. Since $G$ is tight and $H$ is smaller than $G$, $H$ satisfies \Cref{OURCONJ}. Thus $p(H) \leq k(k-3) \leq p(G) + P + Q$. So $p_G(R) \leq p(G) + P + Q + 2i(k-1) + i\delta$ as desired. Furthermore, if $H$ is not $k$-Ore, then $p(H) \leq k(k-3) - P = p(G) + Q$ as desired.
\end{proof}

Our next lemma shows that $i$-collapsible subsets yield $(i+1)$-edge additions for $i\leq (k-3)/2$.

\begin{LMA}\label{colladd} %Thesis Prop. 4.18
If $G$ contains an $i$-collapsible subset $R$ for $i \leq (k-3)/2$, then $G$ admits an $(i+1)$-edge-addition $S$ consisting of non-edges of $G \lb R\rb$.
\end{LMA}
\begin{proof}
Suppose not. If $i = 0$, a collapsible subset clearly admits an edge-addition, as $R$ contains at least two vertices which receive the same color in every $(k-1)$-coloring of $G \lb R \rb$. Hence we may assume that $i \geq 1$.
	
Let $R \subsetneq V(G)$ be an $i$-collapsible subset for some $i$ with $1 \leq i \leq (k-3)/2$. Suppose to the contrary that $G$ does not admit an $(i+1)$-edge-addition. For each $u \in \partial_G R$, let $w(u) = |\{ uv \in E(G)| v \in V(G) \setminus R\}|$. Since $G$ is $(k-1)$-edge connected, $\sum_{u \in \partial_G R} w(u) \geq k-1$. Let $\partial_G R = \{u_1,\ldots, u_s\}$ and assume without loss of generality that $w(u_1) \geq w(u_j)$ for all $j \geq 2$.
	
\begin{claim}
$w(u_2) + \ldots + w(u_s) \leq i + 1$.
\end{claim}
\begin{proof}
Suppose not, that is $w(u_2) + \ldots + w(u_s) \geq i + 2$. We may assume without loss of generality that $w(u_1) \geq w(u_2) \geq \ldots \geq w(u_s)$. Choose the largest index $j$ such that $w(u_j) + \ldots + w(u_s) \geq i + 1$, and set $\alpha = i + 1 - (w(u_{j+1}) + \ldots + w(u_s))$. Since $w(u_1) + \ldots + w(u_s) \geq k - 1$, we also have that $w(u_1) + \ldots + w(u_j) \geq (k-1) - (i+1-\alpha) = k-2-i-\alpha$, which is at least $i + 1 + \alpha$ as $i \leq (k-3)/2$. 
		
By the choice of $j$ and the ordering of the vertices, $0 < \alpha \leq w(u_j) \leq w(u_1)$. Define a multigraph graph $H$ on the vertices $V(H) = \partial_G R$ with the following edges. Add $\alpha$ edges between $u_1$ and $u_j$, and $i + 1 - \alpha$ edges between the vertex sets $\{u_{j+1},\ldots,u_s\}$ and $\{u_1,\ldots,u_j\}$ so that for each $\ell$, the degree of $u_\ell$ in $H$ is at most $w(u_\ell)$. 
		
As $G$ admits no $(i+1)$-edge-addition, there is a $(k-1)$-coloring $\varphi$ of $G\lb R \rb + E(H)$, which is also a $(k-1)$-coloring of $G \lb R \rb$. For any color $c$ and $M = \varphi^{-1}(c)$, we have that
$$\sum_{u \in (\partial_G) \setminus M} w(u) \geq \sum_{u \in (\partial_G R) \setminus M} d_H(u) \geq \frac{1}{2} \sum_{u \in \partial_G R} d_H(u) = i + 1.$$
This contradicts the $i$-collapsibility of $R$.
\end{proof}

Now let $S = \{ u_1u_j| 2 \leq j \leq s\}$, and note that $|S| \leq \sum_{2\le j \le s} w(u_j) \le i + 1$. Hence $G\lb R \rb + S$ has a $(k-1)$-coloring $\varphi$. We may assume that $\varphi(u_1) = 1$ and hence $\varphi(u_j)\ne 1$ for all $j\ge 2$.

Note that $w(u_1) \geq k-1-(i+1) \ge i + 1$. Thus for all $c$ with $2\le c\le k-1$, we have that $|E(\varphi^{-1}(C \setminus \{c\}) \cap R, V(G) \setminus R)| \ge i+1$. Since $R$ is $i$-collapsible, it follows that $|E(\varphi^{-1}(C \setminus \{1\}) \cap R, V(G) \setminus R)|\le i$.  Thus $w(u_2) + \ldots + w(u_s) \leq i$. 

Let $\psi$ be a $(k-1)$-coloring of $G\lb (V(G) \setminus R) \cup \{u_1\} \rb$ such that $\psi(u_1) = 1$. Consider the \emph{improper} coloring $\ov{\psi}$ induced on $G$ by the union of $\varphi$ and $\psi$. Now let us choose $\psi$ so that the number of edges between $R, V(G) \setminus R$ with endpoints having the same color in $\ov{\psi}$ is minimized. Without loss of generality, assume that $\varphi(u_2) = 2$ and that some neighbor $z$ of $u_2$ in $V(G) \setminus R$ also has $\ov{\psi}(z) = 2$.

To obtain a contradiction, we show that the colors on $V(G) \setminus R$ can be permuted so as to contradict the minimality of $\psi$. Let $\ov{C}$ be the set of colors
$$\{1\} \cup \{ \varphi(u_j): 2 \leq j \leq s\} \cup \{ \varphi(v): v \in V(G) \setminus R, \exists u_jv \in E(G), \varphi(u_j) = 2\}.$$
Now $\ov{C}$ contains at most $i + 1$ colors, as there are at most $i$ edges joining $u_2,\ldots,u_s$ to $V(G) \setminus R$. Hence, there exists a color $c' \in \{1,\ldots, k-1\} \setminus \ov{C}$. The coloring $\psi'$ obtained by permuting the colors $2$ and $c'$ in $\psi$ contradicts the minimality of $\psi$, and this completes the proof.
\end{proof}

Now we are ready to prove the following lemma which shows that $G$ does not admit $i$-edge-additions for $i\leq (k-2)/2$ and in turn provides a lower bound on the potential of proper subsets of $G$.

\begin{LMA}\label{buckets}
Let $G$ be a tight, ungemmed graph. Then for all $i$ with $1 \leq i \leq (k-2)/2$, $G$ does not admit an $i$-edge-addition, and there is no subset $R \subsetneq V(G)$, $G \lb R\rb$ not a clique, with $p_G(R) < p(G) + 2i(k-1) + \buc$.
\end{LMA}
\begin{proof}
We will show by induction on $i$ that for $0 \leq i \leq (k-2)/2$, $G$ does not admit an $i$-edge-addition, and there is no subset $R \subsetneq V(G)$, $G\lb R\rb$ not a clique, with $p_G(R) < p(G) + 2i(k-1) + \buc$. The base case $i = 0$ is true since \Cref{edgedrop} implies that $p(G) \leq p_G(R) - 2(k-1)$.
	
Now suppose it holds for $i - 1$ and suppose to the contrary it does not hold for $i$. 
	
\begin{claim}\label{claim:coll}
If $R \subsetneq V(G)$, $G \lb R\rb$ not a clique, with $p_G(R) < p(G) + 2i(k-1) + \buc$, then every critical extension of $R$ is spanning, at most $(i-1)$-incomplete, and has a core of size one.
\end{claim}
\begin{proof}
Let $R'$ be a critical extension of $R$ with extender $W$. By \Cref{edgedrop}, if $R'$ is $\ell$-incomplete, then
$$p_G(R') \leq p_G(R) - 2(\ell+1)(k-1) - \delta < p(G) + 2(i - 1 - \ell)(k-1) + \buc - \delta.$$
This is a contradiction unless $R' = V(G)$ and $\ell \leq i-1$. Hence $R'$ is spanning and $(i-1)$-incomplete. Next suppose that $R'$ has size greater than one. Then by \Cref{submodular},
\begin{align*}p_G(R') &\leq p_G(R) + p(W)  - (p(K_{k-1}) -(k-3)\delta) \\
&= p_G(R) + p(W) - (2(k-1)(k-2) - (k-1)(\delta - \varepsilon)).
\end{align*}
Since $p_G(R) < p(G) + 2i(k-1) + \buc \leq p(G) + (k-1)(k-2) + \buc$ and $p_G(R') \geq p(G)$, we have $p(W) > k(k-3) + 2  - (\buc + (k-1)(\delta - \varepsilon)) \geq k(k-3)$.
However, $G$ is tight, so $W$ satisfies \Cref{OURCONJ}, in which case $p(W) \leq k(k-3)$, a contradiction. Hence $R'$ has a core of size one as desired.
\end{proof}
	
Since the statement does not hold for $i$, there exists $R\subsetneq V(G)$, $G \lb R\rb$ not a clique, with $p_G(R) < p(G) + 2i(k-1) + \buc$. Now let us assume that $R$ is such a subset of minimal size. Note that by \Cref{claim:coll}, every critical extension of $R$ is spanning, at most $(i-1)$-incomplete and has a core of size one. By \Cref{collapsibility}, $R$ is $(i-1)$-collapsible. 
	
Since $i - 1 \leq (k-3)/2$, \Cref{colladd} implies that $G$ admits an $i$-edge-addition $S$ consisting of non-edges of $G\lb R \rb$. By induction, $G$ does not admit an $(i-1)$-edge-addition, so $|S| = i$. Let $S$ be chosen so that the $k$-critical graph $H \subseteq G \lb R\rb + S$ has the minimum number of vertices, and let $R_0 = V(H)$. By the same calculation as in \Cref{addUpperBound}, $p_G(R_0) \leq p(H) + 2i(k-1) + i\delta$. As $G$ is good and $H$ is smaller than $G$, $p(H) \leq k(k-3)$. But then 
\begin{align*}p_G(R_0) &\leq k(k-3) + 2i(k-1) + i\delta \\
&< p(G) + 2i(k-1) + (P + Q + i\delta) \\
&< p(G) + 2i(k-1) + \buc,
\end{align*}
where the second inequality follows since $p(G) \ge k(k-3) - P - Q$ and the third inequality follows since $P+Q+i\delta \le \Delta$ by Assumption 1.4. Since $R$ is a minimum subset with this property, we have that $R_0 = R$.
	
By assumption, $G$ does not admits an $(i-1)$-edge-addition and hence by \Cref{colladd}, $R$ is not $(i-2)$-collapsible. It follows then from \Cref{claim:coll} that there exists a critical extension $R'$ of $R$ that is not $(i-2)$-incomplete.
	
\begin{claim}\label{claim:kOre}
$H$ is $k$-Ore.
\end{claim}
\begin{proof}
Suppose not. As $H$ is smaller than $G$, $p(H) \leq k(k-3) - P < p(G) + Q$. Since $R'$ is spanning and has core of size one, \Cref{submodular,addUpperBound} yield
\begin{align*}
p(G) &\leq p_G(R) + p(W) - 2(i-1)(k-1) - (k^2 - k - 2 + \varepsilon - \delta) \\
&\leq (p(H) + 2i(k-1) + i\delta) + p(W) - 2(i-1)(k-1) - (k^2 - k - 2 + \varepsilon - \delta) \\
&< p(G) + p(W) + Q +  2(k-1) + i\delta - (k^2 - k + \varepsilon - \delta)
\end{align*}
and therefore $p(W) > k(k-3) - (Q + (i+1)\delta - \varepsilon)$. Since $Q + k\delta \leq P$, this implies that $W$ is $k$-Ore.
		
Since $|X| = 1$, by \Cref{kOreGem}, there exists a gem $D$ in $W$ with $x_1 \notin V(D)$. If $i = 1$, then as $R'$ is complete, $D$ is also a gem in $G$, a contradiction that $G$ is ungemmed. Hence we may assume that $i > 1$. If $D$ is a diamond, then $G$ admits an edge-addition, which contradicts the induction hypothesis. 
		
Thus we may assume that $D$ is an emerald, with vertices $\{u_1,\ldots,u_{k-1}\}$. Since $R'$ is $(i-1)$-incomplete, there are at most $i - 1$ edges incident to $D$ in $G$ but not in $W$, so there are at least $k - i$ vertices in $D$ of degree $k-1$ in $G$, which we may assume are $\{u_1,\ldots,u_\ell \}$, $\ell \leq k - i$. For each $1 \leq i \leq \ell$, let $v_i$ denote the unique neighbor of $u_i$ outside $D$. If $v_i \neq v_j$ for any $i\ne j$, then $v_iv_j$ is an edge-addition in $G$. To see this, suppose there exists a $(k-1)$-coloring $\varphi$ of $G + v_iv_j - \{u_i,u_j\}$. Since $\varphi(v_i) \neq \varphi(v_j)$, $\varphi$ extends to a $(k-1)$-coloring of $G$, a contradiction. Since $i > 1$, $G$ does not admit an edge-addition, so $v_i = v_j$ for all $i,j$. But then adding the edges $v_1u_j$ for $j = \ell + 1,\ldots, k - 1$ yields a $K_k$-subgraph; so $G$ admits an $(i-1)$-edge-addition, a contradiction. Therefore $W$ is not $k$-Ore, a contradiction.
\end{proof}
	
\begin{claim}\label{claim:S1}
If $i = 1$, then $S = \{xy\}$ for a single edge $xy$ and $H = G\lb R\rb + S$.
\end{claim}
\begin{proof}
Suppose not. Recall that $V(H) = R$. As $E(G[R])\setminus E(H)\ne \emptyset$, then $p_G(R) \leq p(H) + \delta$. In which case
\begin{align*}
p(G) &\leq p_G(R) + p(W) - (k^2 - k - 2 + \varepsilon - \delta) \\
&\leq p(H) + \delta + (k(k-3) - P) - (k^2 - k - \varepsilon + \delta) \\
&\leq p(H) - 2(k-1) - P - \varepsilon + 2\delta.
\end{align*}
Rearranging yields $p(H) \geq p(G) + 2(k-1) + P + \varepsilon - 2\delta$. Since $p(G) > k(k-3) - P - Q$, we have that $p(H) > k(k-3)$, which is a contradiction, since we assumed $H$ to be $k$-Ore.
\end{proof}
	
\begin{claim}\label{claim:Kk}
$H=K_k$.
\end{claim}
\begin{proof}
Suppose not. Since $H\ne K_k$ and $H$ is $k$-Ore, we have that $H$ is an Ore-composition of two $k$-Ore graphs $H_1, H_2$. We will work towards a contradiction showing that $H$ cannot be an Ore-composition. Let $\{a,b\}$ denote the overlap vertices, and $\underline{ab}$ the split-vertex of $H_2$. The proof for $i = 1$ differs from the general case, so we first resolve $i = 1$.

For $i = 1$ and $S = \{xy\}$, we show that $\partial_G R$ intersects both $V(H_1)\setminus \{a,b\}, V(H_2)-\underline{ab}$. 

\begin{subclaim}\label{claim:boundaryNotOneSide}
There exists $u,v \in \partial_G R$ with $u \in V(H_1) \setminus V(H_2), v \in V(H_2) \setminus V(H_1)$.
\end{subclaim}
\begin{proof}
Suppose to the contrary that $\partial_G R$ is contained in $V(H_1) \setminus V(H_2)$. By \Cref{kOreGem}, there exists a gem $D$ in $H_2$ with $\underline{ab} \notin V(D)$. Since there are no edges between $V(H_2)$ and $G \setminus R$, $D$ is a gem in $G$, a contradiction. Now suppose that $\partial_G R$ was contained in $V(H_2) \setminus V(H_1)$. Notice that in any $(k-1)$-coloring of $G \lb R\rb$, $a$ and $b$ receive different colors, as otherwise there exists a $(k-1)$-coloring of $H_2$, a contradiction. Thus, at most one of $a,b$ is in $\partial_G R$; we may assume without loss of generality that $b \notin \partial_G R$. By \Cref{kOreGem}, there exists a gem $D$ in $H_1$ with $a \notin V(D)$, in which case $ab \notin E(D)$. $D$ does not contain the edge $xy$, or else one of the edges $x,y$ would have degree strictly less than $k - 1$ in $G$. Then $D$ is a gem in $G$, a contradiction.
\end{proof}

The remainder of the argument for $i = 1$ will appear again, so we label it for convenience.

\begin{subclaim}\label{claim:i1GivenBoundary}
If \Cref{claim:boundaryNotOneSide} holds, then for $i = 1$, we obtain the desired result that $H$ cannot be an Ore-composition of two $k$-Ore graphs.
\end{subclaim}
\begin{proof}
From \Cref{claim:boundaryNotOneSide}, there exists $u,v \in \partial_G R$ with $u \in V(H_1) \setminus V(H_2), v \in V(H_2) \setminus V(H_1)$. Let $\varphi$ be a $(k-1)$-coloring of $G \lb R\rb$; since $R$ is collapsible, $\varphi(u) = \varphi(v)$. Since $H_2$ is a split $k$-Ore subgraph of $G$, we have that $\varphi(a)\ne \varphi(b)$. So let us assume without loss of generality that $\varphi(a) = 1$ and $\varphi(b) = 2$.
		
Suppose $\varphi(u)$ is distinct from $\varphi(a), \varphi(b)$. Suppose without loss of generality that $\varphi(u)=3$. Hence, by permuting the colors $3,4$ on $V(H_1)$, we obtain a $(k-1)$-coloring of $G \lb R\rb$ in which $u,v$ have distinct colors, a contradiction. So we may assume $\varphi(u) = \varphi(v) = 1$ in every $(k-1)$-coloring of $G \lb R \rb$. This implies that $H_2 + av$ is not $(k-1)$-colorable, but then $G \lb H_2 \rb$ admits an edge-addition, which contradicts the minimality of $H$. This proves the subcase $i=1$.
\end{proof}
		
Now we consider $i > 1$. Let $\ov{S} = S \cap E(H_1)$ be the edges of $S$ added on the vertex-side. If $\ov{S} \leq i - 1$, then $\ov{S} \cup \{ab\}$ is an $i$-edge-addition contained in $G\lb V(H_1)\rb$, which contradicts the minimality of $H$. Hence $S \subseteq E(H_1)$, so $H_2 - \underline{ab}$ is a subgraph of $G$. By \Cref{kOreGem}, there exists a gem $D$ in $H_2$ with $\underline{ab} \notin V(D)$. If $D$ is a diamond, then $G$ admits an edge-addition, a contradiction. Thus $D$ is an emerald, with vertices $\{u_1,\ldots,u_{k-1}\}$.
		
In a $(k-1)$-coloring of $G\lb R\rb$, the vertices of $D$ receive different colors. Let $X = \{x_1\}$ be the core of $R'$. Since $R'$ is $(i-1)$-incomplete, there are at most $i$ vertices in $V(D) \cap \partial_G R$, as otherwise there are at least $i$ edges from $\{x_2,\ldots,x_{k-1}\}$ to $V(G) \setminus R$ not included in $W$. Now let $\{u_1,\ldots,u_\ell\}$, be the vertices in $V(D) \setminus \partial_G R$. As before, the vertices $u_1,\ldots,u_\ell$ have a common neighbor $v$ outside $D$, or else $G$ would admit an edge-addition. Adding $i$ edges between $v$ and $u_{\ell + 1}, \ldots, u_{k-1}$ yields a $K_k$, which contradicts the minimality of $S$ as an $i$-edge-addition.
\end{proof}
	
Now suppose that $i = 1$. By \Cref{claim:S1}, $S = \{xy\}$ and $H = G \lb R \rb + S$. The vertices $x,y$ are in $\partial_G R$, or else one of $x$ or $y$ has degree less than $k - 1$ in $G$, a contradiction as $G$ is $k$-critical. But $R'$ has a core of size one and is complete, so $\partial_G R = \{x,y\}$. By \Cref{claim:Kk}, $H=K_k$. Hence $G \lb R \rb$ is a diamond in $G$, a contradiction.
	
So we may assume that $i > 1$. Let $V(H) = \{u_1,\ldots,u_k\}$, and note that $G \lb R \rb + S = H$ by the minimality of $R$. Without loss of generality, assume that $u_1u_k \in S$. Let $\varphi: G \lb R\rb \rightarrow \{1,\ldots,k-1\}$ be the improper coloring $\varphi(u_i) = i$ for $1 \leq i \leq k-1$ and $\varphi(u_k) = 1$. Since $R$ is $(i-1)$-collapsible, we have that
$$\min_{c \in C} |E(\varphi^{-1}(C \setminus \{c\}) \cap R, V(G) \setminus R)| \leq i - 1.$$
	
Observe that as $G$ is $k$-critical, $d_G(u_i) \geq k-1$, so each vertex $u_i$ has at least as many neighbors in $V(G) \setminus R$ as the number of edges of $S$ incident with $u_i$. Therefore, if $c \neq 1, k$, the vertex $v_c$ is incident to at most $i - 1$ edges of $S$, and the vertices $\{v_i: i \neq c\}$ are incident with at least $i - 1$ edges with ends in $V(G) \setminus R$, with an additional 2 edges for the edge $u_1u_k \in S$. Hence for $c \neq 1$,
$$|E(\varphi^{-1}(C \setminus \{c\}) \cap R, V(G) \setminus R)| \geq i+ 1.$$
	
Since $R$ is $(i-1)$-collapsible, it follows that $|E(\varphi^{-1}(C \setminus \{1\}) \cap R, V(G) \setminus R)| \leq i-1$. By the pigeonhole principle, every edge in $S$ is incident with either $u_1$ or $u_k$. If the edges $u_1u_i$ and $u_ku_j$ are both in $S$ for some $u_i \neq u_j$, then let $\psi$ be the improper coloring obtained from $\varphi$ by setting $\psi(u_i) = 1, \psi(u_k) = \varphi(u_i)$ (that is, we swap the colors on $u_i$ and $u_k$). Then $\psi$ induces a proper coloring of $G \lb R\rb$ in which no color class covers all but $i - 1$ endpoints of $S$, which contradicts the $(i-1)$-collapsibility of $R$. Hence the edges of $S$ either form a star, or $i = 3$ and $S$ forms a triangle.
	
Suppose first that $S$ forms a star, with $u_1$ the center. In this case, for $R$ to satisfy the $(i-1)$-collapsibility condition, every leaf of the star has exactly one edge going to $V(G) \setminus R$, and every $u_i \in R$ not incident with $S$ has no edges to $V(G) \setminus R$. Then $d_G(u_i) = k - 1$ for $2 \leq i \leq k$ and $G \lb R \setminus \{u_1\} \rb$ is an emerald in $G$, a contradiction.
	
Suppose instead that $S$ forms a triangle on $\{u_1,u_2,u_3\}$. Then each of $u_1,u_2,u_3$ are incident with exactly 2 edges with ends in $V(G) \setminus R$ while for every $i > 3$, $u_i$ is not incident with an edge with an end in $V(G) \setminus R$. However, $|S| = 3$ occurs only when $(k-2)/2 \geq 3$, or $k \geq 8$. Thus $|E(R,V(G) \setminus R)|\le 6$. Yet $G$ is at least 7-edge-connected as $G$ is $k$-critical with $k\ge 8$, a contradiction.
	
We conclude that there is no subset $R \subsetneq V(G)$, $G\lb R\rb$ not a clique, with $p_G(R) < p(G) + 2i(k-1) + \buc$. By  \Cref{addUpperBound}, if $G$ admitted an $i$-edge-addition with $R = V(G)$, we would have $p_G(R) \leq  p(G) + 2i(k-1) + P + Q + i\delta$. Since $P + Q + i\delta < P + Q + \frac{k}{2}\delta \leq \buc$, no such subset $R$ exists, and thus no $i$-edge-addition exists. This completes the induction.
\end{proof}

\section{Cloning}\label{sec:cloning}
In this section, we introduce an operation known as \emph{cloning} and use it to study the structure of tight, ungemmed graphs. Applied to vertices of low degree, cloning produces subgraphs that resemble $K_k$, which we refer to as \emph{clusters}. Using the potential function and the results of \Cref{sec:potential}, we can infer facts about the structure of the cluster, and therefore of the original graph, which are of great importance when we later perform discharging.

We begin by defining cloning and clusters, and analyzing them. A key result is that a tight, ungemmed graph cannot have a cluster of size greater than $k - 4$. We then turn our attention to a particular subgraph called a \emph{gadget} which arises from cloning, and has a certain recursive structure that allows for carefully targeted discharging. Finally, we obtain several results about minimal counterexamples.

\subsection{Cloning and Clusters}
In this section, we define an operation called \emph{cloning}, which creates a copy of a vertex of degree $k-1$ and deletes one of its neighbors. We then use it to derive structural properties of a minimum counterexample. Unless stated otherwise, $G$ is assumed to be $k$-critical. The constants $\varepsilon, \delta, P, Q, \buc, \gap$ satisfy Assumption 1.

\begin{DEF}
Let $x,y$ be vertices with $d(x) = k-1$ and $xy \in E(G)$. The graph $\Gclone{y}{x}$ is obtained by letting $V(\Gclone{y}{x}) = V(G) \setminus \{y\} \cup \{ \tildx \}$ and $E(\Gclone{y}{x}) = E(G - y) \cup \{\tildx v ~|~ v \in N_G(x) \} \cup \{\tildx x\} $. This operation is referred to as \emph{cloning} $x$ with $y$.
\end{DEF}
\begin{DEF}
A \emph{cluster} is a maximal set $R \subseteq V(G)$ such that $d(x) = k-1$ for every $x \in R$, and $N\lb x\rb = N\lb y\rb$ for every $x,y \in R$.

For a cluster $C$, $\Gclone{y}{C}$ denotes the cloning of a vertex of $C$ with $y$. It is easy to see that this is independent of the vertex of $C$ chosen. When we use the notation $\Gclone{y}{C}$, we use $\wt{y}$ to denote the new vertex of the clone (note this is the opposite of the notation $\wt{x}$ used when we clone $\Gclone{y}{x}$).
\end{DEF}

\begin{LMA}\label{criticalClone}
If $x$ and $y$ are vertices such that $x$ is in a cluster $C$ of size $s$, $xy \in E(G)$, and $d(y) \leq k - 2 + s$, then $\Gclone{y}{x}$ is not $(k-1)$-colorable.
\end{LMA}
\begin{proof}
Suppose $\varphi$ is a $(k-1)$-coloring of $\Gclone{y}{x}$. Note $|\{\varphi(v):v\in N(y)\setminus C\}|\ge k-2$. Thus there exists a color $c \in [k-1]\setminus \{\varphi(v): v\in N(y) \setminus C\}$. If there does not exist a vertex $z \in C$ such that $\varphi(z) = c$, then $\varphi$ extends to a $(k-1)$-coloring of $G$ by letting $\varphi(y)=c$, a contradiction. So we may assume there exists a vertex $z\in C$ such that $\varphi(z)=c$; moreover, such a vertex is unique since $G[C]$ is a clique. Now define a coloring $\psi$ on $G$ by letting $\psi(y) = c$, $\psi(z) = \varphi(\tildx)$, and $\psi = \varphi$ for all other vertices of $G$. But then $\psi$ is a $(k-1)$-coloring of $G$, a contradiction.
\end{proof}

\begin{LMA}\label{cloning}
Suppose that $G$ is a tight, ungemmed graph, and $xy \in E(G)$ such that
\begin{enumerate}
\item $x$ is in a cluster $C_x$ of size $s$
\item $d(y) \leq k - 2 + s$
\item If $y$ is in a cluster $C_y$, then $C_y \neq C_x, |C_y| \leq s$.
\end{enumerate}
If $H \subseteq \Gclone{y}{x}$ is a $k$-critical graph, then either $H$ is $k$-Ore, or, $H = \Gclone{y}{x}$ and $d_G(y) = k-1$.
\end{LMA}
\begin{proof}
Suppose not. By conditions 1 and 2 and \Cref{criticalClone}, there exists a $k$-critical subgraph $H$ of $\Gclone{y}{x}$. By condition 3, $H$ is smaller than $G$.
	
Let $R = V(H) \setminus \{\tildx\}$. If $G\lb R \rb$ is a clique, then $H = K_k$ and hence $G \lb R \cup \{y\} \rb$ is a $K_k$-clique in $G$, a contradiction. So we may assume that $G[R]$ is not a clique. The vertex $\tildx$ has degree $k-1$ in $H$. Thus the potential of $R$ satisfies
$$p_G(R) \leq p(H) - ((k-2)(k+1) + \varepsilon) + 2(k-1)(k-1) + \delta.$$
Let $R'$ be a critical extension of $R$ with extender $W$. Since $G\lb R \rb$ is not a clique, $W$ is smaller than $G$. Let $X$ be the core of $R'$.
	
\begin{claim}\label{claim:X1}
$|X|=1$.
\end{claim}
\begin{proof}
Suppose not. First suppose that $1 < |X| < k - 1$. By \Cref{submodular},
\begin{align*}
p(G) \leq p_G(R') &\leq p_G(R) + p(W) - (p(K_2) - 2\delta) \\
&\leq p(H) + p(W) - (k^2 - k - 6) + 3(\delta - \varepsilon).
\end{align*}
As $G$ is good and $H$ is not $k$-Ore, $p(H) \leq k(k-3) - P$. Hence
$$p(W) \geq k(k-3) + 2k - 6 - (Q + 3(\delta - \varepsilon)) > k(k-3).$$
But $W$ is smaller than $G$, so $p(W) \leq k(k-3)$. This is a contradiction.
		
So we may assume that $|X| = k - 1$. It follows from a similar calculation using \Cref{submodular} as above that
$$p(W) \geq k(k-3) - (Q + k(\delta - \varepsilon)) > k(k-3) - P,$$
and hence $W$ is $k$-Ore. Furthermore, if $R'$ is incomplete or not spanning, then by \Cref{edgedrop}, $p_G(R') \geq p(G) + 2(k-1)$ and hence
$$p(W) \geq k(k-3) + 2k - 2 - (Q + k(\delta - \varepsilon)) > k(k-3),$$
a contradiction. So we may assume that $R'$ is spanning and complete.
		
First suppose that $W = K_k$. Then as $|X| = k -1$ and $R'$ is spanning and complete, $\partial_G (V(G) \setminus R) = \{y\}$. This implies that $V(G) = R \cup \{y\}$. Now $p(H) \leq p(G) + Q$, $|V(H)| = |V(G)|$, and $|E(H)| \leq |E(G)|$, and hence we have that $|E(H)| = |E(G)|$. Therefore $d_G(y) = k-1$ and $H = \Gclone{y}{x}$, a contradiction.
		
So we may assume that $W \neq K_k$. Then \Cref{kOreExcludeGem} implies that there exists a gem $D$ in $W$ disjoint from $X$. If $D$ is a diamond, then $G$ admits an edge-addition, contradicting \Cref{buckets}. So we may assume that $D$ is an emerald. Since $R'$ is complete, every vertex $v$ of $D$ satisfies $d_G(v) = d_W(v) = k-1$. But then $D$ is an emerald of $G$, a contradiction.
\end{proof}
	
\begin{claim}\label{claim:RSI} 
$R'$ is spanning and $i$-incomplete for some $i \le (k-4)/2$.
\end{claim}
\begin{proof}
Suppose not. Hence $R'$ is $j$-incomplete for some $j \geq (k-3)/2$. By \Cref{submodular},
\begin{align*}
p_G(R') &\leq p(R) + p(W) - 2j(k-1) - (p(K_1) - \delta) \\
&\leq p(H) + p(W) - 2j(k-1) - (2k - 6 - 2(\delta - \varepsilon)).
\end{align*}
If $R' \subsetneq V(G)$ is proper, \Cref{buckets} implies that $p_G(R') \geq p(G) + (k-1)(k-2) + \buc$. Then, as $p(H) \leq p(G) + Q$, we have that
$$p(W) \geq k(k-3) + 2j(k-1) + 4(k-2) + \buc - (Q + 2(\delta - \varepsilon)) > k(k-3),$$
a contradiction. Thus, $R'$ is spanning and so $p(G) = p_G(R')$. But then we have that
$$p(W) \geq 2j(k-1) + (2k - 6 - Q) - 2(\delta - \varepsilon).$$
Since $j \geq (k-3)/2$, we have that $p(W) \geq k(k-3) + (k - 3 - Q - 2\delta - \varepsilon) \ge k(k-3)$ by Assumption 1, a contradiction. 
\end{proof}
	
By \Cref{claim:X1,claim:RSI}, every critical extension of $R$ is spanning, has a core of size one, and is at most $(k-4)/2$-incomplete. Thus by \Cref{collapsibility}, $R$ is $(k-4)/2$-collapsible. Hence by \Cref{colladd}, $G$ admits a $(k-2)/2$-edge-addition, contradicting \Cref{buckets}.
\end{proof}

\begin{LMA}\label{singleVertex}
Let $G$ be a tight, ungemmed graph. If $R \subsetneq V(G), |R| \geq k$ is a proper subset, then $p_G(R) > p(G) + (k-1)(k-2) + (\gap + Q + \delta)$ unless $G - R$ is a single vertex of degree at most $k - 1 + (k-4)/2$ in $G$.
\end{LMA}
\begin{proof}
Suppose not. Let $R \subsetneq V(G), |R| \geq k$ be a proper subset with $p_G(R) \leq p(G) + (k-1)(k-2) + (\gap + Q + \delta)$. 
	
\begin{claim}\label{claim:RSI2}
If $R'$ is a critical extension of $R$, then
$R'$ is spanning and $i$-incomplete for some $i \le (k-3)/2$.
\end{claim}
\begin{proof}
Suppose not. First suppose that $R'$ is not spanning. Then \Cref{edgedrop} implies that
\begin{align*}
p_G(R') \leq p_G(R) - 2(k-1) - \delta &\leq p(G) + (k-1)(k-2) + (\gap + Q + \delta) - 2(k-1) - \delta \\
&= p(G) + (k-1)(k-2) - (2(k-1) - \gap - Q - \delta) \\
&< p(G) + (k-1)(k-2) + \buc
\end{align*}
This contradicts \Cref{buckets} as $R' \neq V(G)$. So we may assume that $R'$ is spanning. 
		
So we may assume that $R'$ is not $i$-incomplete for some $i\le (k-4)/2$. For even $k$, if $R'$ is at most $(k-2)/2$-incomplete, then \Cref{edgedrop} implies
$$p(G) = p_G(R') \leq p_G(R) - k(k-1) - \delta \leq p(G) - (2(k-1) - \gap - Q)$$
a contradiction. For odd $k$, if $R'$ is at most $(k-3)/2$-incomplete, then the same calculation using \Cref{edgedrop} implies that $\gap + Q \geq k-1$, a contradiction.
\end{proof}
	
It follows from \Cref{buckets} that $R$ is not $(k-4)/2$-collapsible. But then by \Cref{collapsibility} and \Cref{claim:RSI2}, there exists a critical extension $R'$ of $R$ that has a core of size greater than one. Let $W$ be the extender of $R$ and let $X$ be the core of $R$.
	
First suppose that $2 \leq |X| < k-1$. Then by \Cref{submodular},
\begin{align*}
p(G) = p_G(R') &\leq p(R) + p(W) - (p(K_2) - 2\delta) \\
&\leq p(G) + p(W) - (k(k-3) + 4k - 6 - \gap - Q + 2\varepsilon - 3\delta).
\end{align*}
Thus $p(W) > k(k-3)$, a contradiction. 
	
So we may assume that $|X|=k-1$. Repeating the previous calculation with $|X| = k-1$ in \Cref{submodular}, we find that
\begin{align*}
p(G) = p_G(R') &\leq p(R) + p(W) - (p(K_{k-1}) - (k - 3)\delta) \\
&\leq p(G) + p(W) - (k(k-3) - (\gap - 2 + Q + k \delta - (k-1)\varepsilon)).
\end{align*}
Since $(\gap - 2) + Q + k\delta - (k-1)\varepsilon < P$ by Assumption 1.3, we have that $p(W) > k(k-3) - P$. Hence $W$ is $k$-Ore. 
	
Now suppose that $W \neq K_k$. Then by \Cref{kOreExcludeGem}, there exists a gem $D$ with vertices in $V(W) \setminus X$. If $D$ is a diamond, then $G$ admits an edge-addition, contradicting \Cref{buckets}. So we may assume that $D$ is an emerald. As $R'$ is spanning and at most $(k-4)/2$-incomplete by \Cref{claim:RSI2}, there are at most $(k-4)/2$ vertices of $D$ incident with an edge in $E(R') \setminus E(W)$. Hence at least $k - 1 - (k-4)/2 \geq k/2 + 1$ vertices of $D$ have degree $k - 1$ in $G$. Let $u_1,\ldots,u_\ell$ denote the vertices in $D$ with $d_G(u_i) = k-1$. For each $i$ with $1\le i \le \ell$, let $v_i$ denote the neighbor of $u_i$ their neighbors not in $D$. By \Cref{emeraldCluster}, $\ov{N}(u_i) = \ov{N}(u_j)$ for all $i,j$. Hence, there exists a vertex $v\in V(G)\setminus V(D)$ such that $v$ is adjacent to all of $\{u_1,\ldots,u_\ell\}$. Adding at most $(k-4)/2$-edges between $v$ and $u_{\ell + 1},\ldots,u_{k-1}$ yields a $K_k$-subgraph of $G$. Thus $G$ admits a $(k-4/)2$-edge-addition, contradicting \Cref{buckets}.
	
So we may assume that $W = K_k$. A $R'$ is spanning and $|X| = k-1$, $|G - R| = 1$. Since $R'$ is at most $(k-4)/2$-incomplete, the single vertex of $G - R$ has degree at most $k - 1 + (k-4)/2$, a contradiction.
\end{proof}

Next, we show that tight and ungemmed graphs have no clusters of size $k - 3$ or larger. We will need the following theorem of Kostochka and Yancey.

\begin{THM}[Theorem 6, \cite{KY2018Combinatorica}]\label{strong_KY}
Let $k \geq 4$ and $G$ be a $k$-critical graph. Then $G$ is $k$-extremal (i.e, $G$ satisfies \Cref{KY1} with equality) if and only if it is a $k$-Ore graph. Moreover, if $G$ is not a $k$-Ore graph, then
$$|E(G)| \geq \frac{(k-2)(k+1)|V(G)| - y_k}{2(k-1)}$$
where $y_4 = 2, y_5 = 4$, and $y_k = k^2 - 5k + 2$ for $k \geq 6$.
\end{THM}

Restated in terms of the potential function, \Cref{strong_KY} shows that for $k \geq 6$, $k$-critical graphs $G$ that are not $k$-Ore have smaller potential as follows.

\begin{CORO}\label{strong_KY_cor}
If $k \geq 6$ and $G$ is $k$-critical graph that is not $k$-Ore, then $\pk(G) \leq k(k-3) - 2(k-1)$.
\end{CORO}

\begin{LMA}\label{k3Cluster}
Let $G$ be a tight and ungemmed graph. The maximum size of a cluster in $G$ is $k - 4$.
\end{LMA}
\begin{proof}
Suppose to the contrary that there exists a cluster $C$ of size at least $k - 3$. If $|C| \geq k - 2$, then $G[C\cup N(C)]$ is a diamond of $G$, a contradiction. So we may assume that $|C| = k - 3$. Let $C = \{u_1,\ldots,u_{k-3}\}$, and let $v_1, v_2, v_3$ be the common neighbors of vertices in $C$. Adding the edges $v_1v_2, v_2v_3, v_1v_3$ produces a $K_k$-subgraph, so $G$ admits a 3-edge-addition, a contradiction to \Cref{buckets} when $k \geq 8$. Hence, we may assume that $k \in \{6,7\}$. Furthermore, as $G$ does not admit a $2$-edge-addition by \Cref{buckets}, we may assume that $v_1v_2,v_1v_3,v_2v_3\not\in E(G)$. For each $i \in \{1,2,3\}$, let $G_i$ denote the graph obtained from $G-C-v_i$ by identifying the two vertices of $\{v_1,v_2,v_3\}\setminus \{v_i\}$.
	
\begin{claim}\label{claim:DegOrOre}
For each $i\in\{1,2,3\}$, at least one of the following holds:
\begin{enumerate}
\item $v_i$ has degree at most $k$, or
\item $G_i$ has a $k$-Ore subgraph.
\end{enumerate}
\end{claim}
\begin{proof}
Suppose not. We may assume without loss of generality that the claim does not hold for $i=3$. Let $w$ denote the identified vertex of $G_3$. Note that $G_3$ is not $(k-1)$-colorable, as a $(k-1)$-coloring $\varphi$ of $G_3$ extends to a $(k-1)$-coloring of $G$ by letting $\varphi(v_1)=\varphi(v_2)=\varphi(w)$ and coloring the vertices in $C$ with the remaining $k - 3$ colors of $[k-1]\setminus \{\varphi(w),\varphi(v_3)\}$. Let $K$ be a $k$-critical subgraph of $G_3$ and note that $w \in V(K)$. Let $R = V(K) - w + \{v_1, v_2\} + C$. Observe that $|V(R)| = |V(K)| + k - 2$, $|E(G \lb R\rb)| \geq |E(K)| + \binom{k-1}{2} - 1, T(G \lb R \rb) \geq T(K) - 1$. Thus,
\begin{align*}
p_G(R) &\leq p(K) + ((k-2)(k+1) + \varepsilon)(k-2) - 2(k-1)\left( \binom{k-1}{2}-1\right) + \delta \\
&\leq p(K) + k^2 - 3k + 4 + 2\delta.
\end{align*}
Furthermore, if $v_3 \in V(K)$, then $|E(G \lb R\rb)| \geq |E(K)| + \binom{k}{2} - 3$; in which case, $p_G(R) \leq p(K) + (-k^2 + 5k - 2) + 2\delta$.
		
For all $k \geq 5$, $-k^2 + 5k - 2 \leq -2$. As $G$ is good, $K$ satisfies \Cref{OURCONJ} and hence $p(K) \leq k(k-3)$. Since $R$ has a non-edge $v_1v_2$, $G \lb R\rb$ is not a clique, and \Cref{edgedrop} implies that $p(G) \leq p_G(R)$. Thus, if $v_3 \in V(K)$, $p(G) \leq p_G(R) \leq k(k-3)  +( -k^2 + 5k - 2) \leq k(k-3) - 2$, which contradicts the fact that $G$ is tight. We deduce that $v_3 \notin V(K)$.
		
Now $K$ is not $k$-Ore as otherwise 2 holds. But then
$$p_G(R) \leq p(G) + (k-1)(k-2) + 2 + Q + 2\delta \leq p(G) + (k-1)(k-2) + \gap + Q + \delta.$$
Now \Cref{singleVertex} implies that $G - K = \{v_3\}$ is a single vertex of degree at most $k - 1 + (k-4)/2$. Since $k \leq 7$, it follows that $v_3$ has degree at most $k$ and 1 holds, a contradiction.
\end{proof}
	
\begin{claim}\label{claim:AtMostOne}
At most one of the vertices $v_1,v_2,v_3$ has degree at most $k$.
\end{claim}
\begin{proof}
Suppose not. Suppose without loss of generality that both $v_2$ and $v_3$ have degree at most $k$. Since $G$ is $k$-critical, $v_2$ and $v_3$ have degree at least $k-1$. We have two cases to consider.
		
First, suppose that both $v_2$ and$v_3$ have degree $k-1$. Let $\varphi: V(G) \setminus (C \cup \{v_2, v_3\}) \rightarrow \{1,\ldots,k-1\}$ be a $(k-1)$-coloring of $G \setminus (C \cup \{v_2, v_3\})$. But then  $\varphi$ extends to $G$, a contradiction, as follows. For each $i\in\{2,3\}$, let $L_i=\{\varphi(v):v\in N(v_i)\}$. As $v_2$ has degree $k-1$, $v_2$ has at most 2 neighbors not in $C \cup \{v_2, v_3\}$ and hence $|L_2| \leq 2$. Similarly $|L_3|\leq 2$. 
		
If $\varphi(v_1) \notin L_2$, then let $\varphi(v_2) = \varphi(v_1)$, let $\varphi(v_3)\in [k-1]\setminus L_3$ and then color the vertices in $C$ from $[k-1]\setminus \{\varphi(v_1),\varphi(v_3)\}$. So we may assume that $\varphi(v_1)\in L_2$ and similarly that $\varphi(v_1)\in L_3$. But then $|(L_2 \cup L_3) \setminus \{\varphi(v_1)\}| \leq 2 < k - 2$ since $k\ge 5$. But then we let $\varphi(v_2)\in [k-1]\setminus (L_2\cup L_3)$, let $\varphi(v_3) = \varphi(v_2)$, and color the vertices of $C$ with colors from $[k-1]\setminus \{\varphi(v_1),\varphi(v_2)\}$.
		
Hence, we may assume without loss of generality that $v_2$ has degree $k$. Since $k \geq 6$, we have that $k \leq k - 2 + (k-3)$. By \Cref{cloning}, $\Gclone{v_2}{C}$ contains a $k$-Ore subgraph $H$. Since $d_G(v_2) = k$, we have that $|E(G)| > |E(H)|$. If $|V(H)| = |V(G)|$, then
$$p(G) \leq p(H) - 2(k-1) + \delta \leq k(k-3) - 2(k-1) + \delta < k(k-3) - P - Q,$$
which contradicts that $G$ is tight. 
		
So we may assume that $|V(H)| < |V(G)|$. Let $R = (V(H) \setminus \{\widetilde{v_2}\}) \cup \{v_2\}$, which is a proper subset of $V(G)$. Observe that $H\subseteq G \lb R\rb + v_2v_1 + v_2v_3$, and therefore $G$ admits a 2-edge-addition, contradicting \Cref{buckets}.
\end{proof}
	
By \Cref{claim:DegOrOre,claim:AtMostOne}, for at least two $i\in \{1,2,3\}$, $G_i$ has a $k$-Ore subgraph. We may assume without loss of generality that both $G_1$ and $G_2$ contain a $k$-Ore subgraph. For each $i\in\{1,2\}$, let $K_i$ be a $k$-Ore subgraph of $G_i$ and let $R_i = V(K_i) \setminus \{w_i\} \cup \{v_i, v_3\}$. Notice that $R_1$ is disjoint from $C \cup \{v_2\}$ and $R_2$ is disjoint from $C\cup \{v_1\}$. Let $R = R_1 \cup R_2 \cup C$. 
	
\begin{claim}
$$\pk_G(R) \le k(k-3)+2(k-1).$$
\end{claim}
\begin{proof}
Suppose not. Let $H = R_1 \cap R_2$. For each $i\in \{1,2\}$, let
$$E_i = \{ xy: x, y \in V(K_i) \cup \{v_3\}, xy \in E(G \lb R_i\rb), xy \notin E(K_i)\}.$$
That is, $E_i$ is the set of edges which are present in $G$ but not in $K_i$, and have both endpoints in $V(K_i) \cup \{v_3\}$. Finally, define
$$E_H = \{ xy: x, y \in R_1 \cap R_2, xy \in E(G \lb R_1 \cap R_2 \rb), xy \notin E(K_1)\}.$$
Note that $E_H \subseteq E_1$.
		
For each $i\in\{1,2\}$, $|R_i| \ge |V(K_i)|+1$ and $|E(G[R_i])|\ge |E(K_i)|+|E_i|$, and hence we have that $\pk_G(R_i) \leq \pk(K_i) + (k-2)(k+1) - 2(k-1)|E_i|$. 
		
Let $S = E( K_1\lb H \rb)$. Observe that  $(k-2)(k+1)|H| - 2(k-1)|S| = \pk_{K_1}(H) \geq k(k-3) + 2(k-1)$, by \Cref{edgedrop}. Then $\pk_G(H) = \pk_{K_1}(H) - 2(k-1)|E_H|$. We calculate that
\begin{align*}\pk(R_1 \cup R_2)
&\leq \pk_G(R_1) + \pk_G(R_2) - \pk_G(H) \\
&\leq \pk(K_1) + \pk(K_2)  - \pk_{K_1}(H) + 2(k-2)(k+1)-2(k-1)(|E_1| + |E_2|) +2(k-1)|E_H| \\
&\leq k(k-3) + 2(k-2)(k+1) - 2(k-1)(|E_1| + |E_2| - |E_H| + 1) \\
&\leq k(k-3) + 2(k-2)(k+1) - 2(k-1).
\end{align*}
Next, adding the vertices and edges of the cluster $C$, we find that
\begin{align*}\pk_G(R) &\leq k(k-3) + 2(k-2)(k+1) - 2(k-1) + (k-2)(k+1)(k-3) - (k-1)(k-3)(k+2) \\
&= k(k-3) + 2(k-1).
\end{align*}
\end{proof}
	
By \Cref{kOreTBound}, for each $i\in\{1,2\}$, $T(K_i) \geq 2 + \frac{|V(K_i)| - 1}{k-1}$. Note that $T(G \lb R \rb) \geq \max\{ T(K_1), T(K_2)\} - 1$. Since $|V(K_1)| + |V(K_2)| \geq |R| - (k-2)$, we may assume without loss of generality that $|V(K_1)| \geq \frac{1}{2}(|R| - (k-2))$. Therefore $T(G \lb R\rb) \geq 1+ \frac{|R| - (k-1)}{2(k-1)}$. But then
\begin{align*}
p_G(R) &\leq \pk_G(R) + \varepsilon |R| - \delta \left( 1 + \frac{|R| - (k-1)}{2(k-1)} \right) \\
&\leq k(k-3) + 2(k-1) + \varepsilon |R| - \delta \left( 1 + \frac{|R| - (k-1)}{2(k-1)} \right) \\
&= k(k-3) + 2(k-1) + (\varepsilon - \frac{1}{2(k-1)} \delta) |R| - \frac{1}{2}\delta.
\end{align*}
It follows that $R = V(G)$ by \Cref{buckets}. As $G$ is not $k$-Ore, \Cref{strong_KY_cor} implies that $\pk(G) \leq k(k-3) - 2(k-1)$. Thus, substituting $R = V(G)$ above, we have that
\begin{align*}
p(G) &\leq \pk (G) + \varepsilon |V(G)| - \delta \left( 1 + \frac{ |V(G)| - (k-1)}{2(k-1)} \right) \\
&\leq k(k-3) - 2(k-1) \\
&< k(k-3) - P - Q,
\end{align*}
which contradicts the fact that $G$ is tight.
\end{proof}

\begin{LMA}\label{ungemmedClone}
If $G$ is a tight, ungemmed graph, $x$ is a vertex of degree $k-1$, and $y$ is a neighbor of $x$, then $\Gclone{y}{x}$ is ungemmed.
\end{LMA}
\begin{proof}
Suppose not. Let $\tildx$ denote the new vertex in $\Gclone{y}{x}$. First, suppose that $\Gclone{y}{x}$ contains a diamond $D$. If $\tildx \notin V(D)$, then $G \lb V(D) \rb$ is a $K_k - e$ subgraph of $G$, then $G$ admits an edge-addition, contradicting \Cref{buckets}. Hence $\tildx \in V(D)$, and therefore $x \in V(D)$ as well. Let $C$ be the set of vertices in $D - \tildx$ of degree $k-1$ in $\Gclone{y}{x}$. Note that $|C|\ge k-3$. Every vertex $v\in C$ is adjacent to $y$ in $G$, as otherwise $v$ has degree strictly less than $k - 1$ in $G$. Hence $C$ is a cluster in $G$ of size $k - 3$, contradicting \Cref{k3Cluster}.
	
So we may assume that $\Gclone{y}{x}$ contains an emerald $D$. Further suppose that $\tildx \in V(D)$. If $x \notin V(D)$, then as every vertex of $D$ is a neighbor of $x$, $G \lb (V(D) \setminus \{\tildx\}) \cup \{x\} \rb = K_k$. Note that every vertex $v\in D$ is adjacent to $y$ in $G$ since $v$ has degree $k-1$ in $\Gclone{y}{x}$. Then it follows that $G \lb (V(D) \setminus \{\tildx\}) \cup \{x, y\} \rb = K_k$, a contradiction. So we may assume that both $x, \tildx \in V(D)$. But then $G \lb (V(D) \setminus \{\tildx\}) \cup \{y\}\rb = K_{k-1}$, and \Cref{emeraldCluster} implies that the $k - 2$ vertices of degree $k-1$ in $D - \tildx$ are in the same cluster in $G$. This contradicts \Cref{k3Cluster}.
	
So we may assume that $x, \tildx \notin V(D)$. As $G \lb V(D) \rb = K_{k-1}$, \Cref{emeraldCluster} implies that the vertices of degree $k-1$ in $D$ are in the same cluster $C$. The only edges present in $G$ that are not present in $\Gclone{y}{x}$ are those incident with $y$. Hence if a vertex $v$ in $D$ has degree greater than $k - 1$ in $G$, then $v$ is adjacent to $y$. Therefore the vertices of $D - C$ have degree $k$ in $G$. If $C = \emptyset$, then $G \lb V(D) \cup \{y\} \rb = K_k$, a contradiction. Hence there exists a vertex $z \in V(D)$ with $d_G(z) = k-1$. Let $w$ be the unique neighbor of $z$ not in $D$. \Cref{kNeighbor} implies that every vertex in $D - C$ is also adjacent to $w$. But then $G \lb V(D) \cup \{w\} \rb = K_k$, a contradiction.
\end{proof}

\begin{LMA}\label{tightClone}
If $C$ is a cluster of size $s$, $x$ is a vertex of $C$, and $y$ is a neighbor of $x$ with degree at most $k - 2 + s$, then $p(\Gclone{y}{x}) \geq p(G) - \delta$.
\end{LMA}
\begin{proof}
Observe that $|V(\Gclone{y}{x})| = |V(G)|$, $|E(\Gclone{y}{x})| \leq |E(G)|$, and $T(\Gclone{y}{x}) \leq T(G) + 1$. The result follows by evaluating $p(\Gclone{y}{x})$.
\end{proof}

\subsection{Gadgets}

Several structures resembling $k$-Ore subgraphs arise from cloning, and their properties will be crucial for analyzing the discharging. Suppose that $G$ is tight and ungemmed, and $x,y$ are vertices satisfying the conditions of \Cref{cloning}, such that $\Gclone{y}{x}$ contains a $k$-Ore subgraph $H$. We first show that $H$ has a unique frame which contains $x,\tildx$.

\begin{LMA}\label{cloneKey}
Let $G$ be tight and ungemmed, and suppose that
\begin{enumerate}
\item $x,y$ satisfy the conditions of \Cref{cloning}
\item $\Gclone{y}{x}$ contains a $k$-Ore subgraph $H$
\end{enumerate}
Then $H$ has a unique frame $V^\ast(H) = (N(x) \setminus \{y\}) \cup \{\wt{x}\}$.
\end{LMA}
\begin{proof}
We claim that any frame $V^\ast(H)$ must be the closed neighborhood of $x$ in $\Gclone{y}{x}$. Suppose to the contrary that $x$ is inside a replacement edge with endpoints $a,b$. Note that $\tildx$ is then also within the replacement edge $a,b$. Then $G \lb V(H) \setminus \{\tildx\} \rb + ab$ is a $k$-Ore subgraph, as we have added a real edge over the replacement edge in $ab$, so $G$ admits an edge-addition, contradicting \Cref{buckets}. Therefore $x, \tildx \in V^\ast(H)$, and the other vertices of $V^\ast(H)$ are neighbors of $x$.
\end{proof}

This lemma shows that the clone $\Gclone{y}{x}$ has a $K_k$-subgraph with replacement edges, and with $x,\tildx$ vertices of $K_k$. Then upon deleting $\tildx$, $G$ contains a subgraph that can be viewed as $K_{k-1}$ with replacement edges, with $x$ a vertex of $K_{k-1}$. This motivates the following definitions.

\begin{DEF}
A \emph{gadget} is a subgraph $H$ obtained from a $k$-Ore graph by deleting a vertex $x$ of degree $k-1$ in a cluster of size at least 2.

A \emph{proto-gadget} is a subgraph $H$ isomorphic to $K_{k-1}$ with replacement edges, along with a distinguished frame $V^\ast(H) = V(K_{k-1})$.

A \emph{kite} is a subgraph $H$ isomorphic to $K_{k-2}$ with replacement edges, along with a distinguished frame $V^\ast(H) = V(K_{k-2})$.
\end{DEF}

\begin{LMA}\label{splitProto}
Let $G$ be tight and ungemmed, and let $H$ be a split $k$-Ore subgraph of $G$, with split vertices $\{a,b\}$. If $x \in V(H) \setminus \{a,b\}$, then $x \in V^\ast(H')$ for some proto-gadget $H'$ of $G$ with $H' \subseteq H - \{a,b\}$.
\end{LMA}
\begin{proof}
We argue by induction on $|V(H)|$. If $|V(H)| = k + 1$, then $H - \{a,b\} = K_{k-1}$, which is a proto-gadget containing $x$ as desired. Note that $V^\ast(H)$ is unique for $H = K_{k-1}$.

Proceeding inductively, let $H^\bullet$ be the $k$-Ore graph obtained by identifying $a,b$ to a single vertex $\underline{ab}$, and let $V^\ast(H^\bullet)$ be a frame of $H^\bullet$. First suppose that $\underline{ab} \notin V^\ast(H^\bullet)$. But then $\underline{ab}$ is contained within a replacement edge with endpoints $u,v$. Then $G + uv$ contains a $k$-Ore graph and thus $G$ admits an edge-addition, a contradiction to \Cref{buckets}. 

Now, if $x$ is not in a replacement edge, then delete $\underline{ab}$ and its incident replacement edges to obtain a proto-gadget $H'$ with $x \in V(H')$ and $H' \subseteq H - \{a,b\}$ as desired.

So we may assume that $x$ is in a replacement edge of $H^\bullet$ with endpoints $u,v$. Let $H^{\bullet \bullet}$ be the split $k$-Ore graph contained in the replacement edge. As $\underline{ab} \in V^\ast(H)$, either $\underline{ab} \notin V(H^{\bullet \bullet})$, or $\underline{ab} \in \{u,v\}$. Clearly $|V(H^{\bullet \bullet})| < |V(H)|$, so the induction hypothesis implies that there exists a proto-gadget $H' \subseteq H^{\bullet \bullet} - \{u,v\}$ with $x \in V^\ast(H')$. Since $\underline{ab} \notin V(H')$, $H'$ is a proto-gadget of $G$ and $a,b \notin V(H')$ as desired.
\end{proof}

Recall that a replacement edge $e$ between the vertices $a,b$ is precisely a split $k$-Ore subgraph $H$ with split vertices $\{a,b\}$. Abstractly, we wish to be able to treat the graph $G$ as containing an edge $ab$ instead of the subgraph $H$, and more generally, given a recursive description of the graph in terms of successive replacement edges, we should be able to treat higher-level representations of the graph without reference to the internals of the replacements. This motivates the next definition.

\begin{DEF}\label{defInsideEdge}
Given a replacement edge $e$ corresponding to a split $k$-Ore subgraph $H$ with split vertices $\{a,b\}$, we say that a set of vertices $S$ is \emph{inside (the replacement edge) $e$} if $S \subseteq V(H) \setminus \{a,b\}$.
\end{DEF}

Stated in this parlance, \Cref{splitProto} shows that if $e$ is a replacement edge and $x$ is inside $e$, then there is an entire proto-gadget $H'$ inside $e$ with $x \in V(H')$.

\subsection{Structural Properties}
\begin{LMA}\label{3path}
If $G$ is a tight, ungemmed graph, then $G$ does not contain an induced path of length 3 consisting of vertices of degree $k-1$.
\end{LMA}
\begin{proof}
Suppose to the contrary that there exists an induced path $P=v_1v_2v_3v_4$ such that each vertex of $P$ has degree $k-1$. Let $\{z_1,\ldots,z_{k-3}\}$ be the neighbors of $v_2$ not in $P$. Consider $\Gclone{v_2}{v_3}$, that is cloning $v_2$ with $v_3$. As $v_3$ has degree $k-2$ in $\Gclone{u}{y}$, $\Gclone{v_2}{v_3}$ is not $k$-critical. Hence, \Cref{cloning} implies that $\Gclone{v_2}{v_3}$ contains a $k$-Ore subgraph $H$, and \Cref{cloneKey} implies that $V^\ast(H) = \{v_1, \widetilde{v_1}, v_2, z_1,\ldots,z_{k-3}\}$.
	
We claim that $v_1z_i \in E(G)$ for all $1 \leq i \leq k - 3$. Suppose not. We may assume without loss of generality that $v_1z_1\notin E(G)$ and hence $v_1z_1$ is a replacement edge in $H$ instead of a real edge. Let $T$ be the set of neighbors of $v_1$ within the replacement edge $v_1z_1$, and let $S$ be the set of edges $\{ wz_1: w \in T\}$. 
	
Observe that $G + S$ contains a $k$-Ore subgraph, as $z_1$ now has the same neighbors as $z_1$ identified with $v_1$. Thus $|T| \geq 3$, as $G$ does not admit 2-edge-additions by \Cref{buckets}. But $v_1$ has at least one real edge for each of the vertices $v_1, z_2,\ldots,z_{k-3}$, so $d_G(x) \geq k-3 + 3 = k$, a contradiction. It follows that $v_1z_i$ is an edge of $G$ for all $i$ with $1 \leq i \leq k - 3$. But then $v_1,v_2$ have $k - 3$ common neighbors. Hence \Cref{emeraldCluster} implies that $v_1,v_2$ are in the same cluster. So $v_1$ is adjacent to $v_3$, which contradicts the assumption that $P$ is induced.
\end{proof}

\begin{LMA}\label{min2Cut}
If $G$ is a smallest counterexample to \Cref{OURCONJ}, then $G$ is 3-connected.
\end{LMA}
\begin{proof}
Suppose to the contrary that $G$ has a 2-cut $\{x,y\}$, separating $G$ into two edge-disjoint subgraphs $G_1, G_2$. As $G$ is $k$-critical, $x$ and $y$ as otherwise there exists a $(k-1)$-coloring of $G$ by taking the union of a $(k-1)$-coloring of $G_1$ and a $(k-1)$-coloring of $G_2$ as necessary (cf. Theorem 3 of \cite{D1953FM}). As noted in \cite{D1953FM}, either $\varphi_1(x) = \varphi_1(y), \varphi_2(x) \neq \varphi_2(y)$ for all possible $(k-1)$-colorings $\varphi_1, \varphi_2$ of $G_1, G_2$ respectively, or vice versa, as otherwise, by permuting colors if necessary, the union of $\varphi_1, \varphi_2$ is a $(k-1)$-coloring of $G$, a contradiction. So without loss of generality, assume that $\varphi_1(x) = \varphi_1(y)$ for all $(k-1)$-colorings $\varphi_1$ of $G_1$.
	
We claim that $G$ is an Ore-composition of $H_1 = G_1$ and $H_2 = G_2 - \{x,y\} + \underline{xy}$, with $H_1$ being the edge-side. Suppose not. Then there exists a vertex $z\in V(G_2)$ such that $z$ is adjacent to both $x$ and $y$. Let $\varphi$ be a $(k-1)$-coloring of $G - xz$. Since $\varphi$ induces a $(k-1)$-coloring of $G_1$, it follows that $\varphi(x) = \varphi(y)$. Since $yz\in E(G-xz)$, we have that $\varphi(x) \neq \varphi(z)$. But then $\varphi$ is a $(k-1)$-coloring of $G$, a contradiction. This proves the claim.
	
Since $G$ is not $k$-Ore, at least one of $H_1$ or $H_2$ is not $k$-Ore. Note that both $H_1$ and $H_2$ are smaller than $G$ and so satisfy \Cref{OURCONJ}. Observe that $|V(G)| = |V(G_1)| + |V(G_2)| - 1, |E(G)| = |E(H_1)| + |E(H_2)| - 1$, and by \Cref{Tsuperadd}, $T(G) \geq T(H_1) + T(H_2) - 2$. Thus
$$p(G) \leq p(H_1) + p(H_2) - (k-2)(k+1) - \varepsilon + 2(k-1) + 2\delta.$$
As $G$ is a minimal counterexample, $p(G) > k(k-3) - P$. 
	
First suppose that $H_2$ is not $k$-Ore. Then $p(H_2) \leq k(k-3) - P$, and
$$k(k-3) - P < p(G) \leq p(H_1) - P - \varepsilon + 2\delta,$$
whence $p(H_1) > k(k-3) - 2\delta + \varepsilon$. It then follows that $H_1 = K_k$, as $\delta \geq 2(k-1)\varepsilon$ implies
$$k(k-3) - 2\delta + \varepsilon > k(k-3) + \varepsilon |V(G)| -  \left( 2 + \frac{|V(G)| - 1}{k-1} \right)\delta.$$
But then the sharper form of \Cref{Tsuperadd} when $H_1 = K_k$ implies that $T(G) \geq T(H_1) + T(H_2) - 1$, and so
$$p(G) \leq k(k-3) - P - (\delta - (k-1)\varepsilon) < k(k-3) - P,$$
which contradicts the assumption that $G$ is a counterexample. So we may assume that $H_2$ is $k$-Ore and $H_1$ is not. But then the same argument applies using again the sharper form of \Cref{Tsuperadd}. 
\end{proof}

\begin{LMA}\label{minNoEdgeAdd}
If $G$ is a smallest counterexample to \Cref{OURCONJ}, then $G$ does not admit an edge-addition.
\end{LMA}
\begin{proof}
The proof is the same as $i = 1$ in \Cref{buckets}, with a small modification to account for the fact that we are no longer assuming $G$ to be ungemmed. Let $R \subsetneq V(G)$ be a subset of minimum size such that $G\lb R \rb + xy $ contains a $k$-critical subgraph $H$. The same argument as \Cref{claim:coll} in \Cref{buckets} shows that $R$ is collapsible. Since $H$ is proper, by \Cref{edgedrop} we have that $k(k-3) - P < p(G) < p(H)$ and hence $H$ is $k$-Ore. Thus we have established the analogue of \Cref{claim:kOre} in \Cref{buckets}. The same argument as in \Cref{claim:S1} holds here. Together, these facts imply that $H$ is $k$-Ore and $H = G\lb R \rb + xy$, from which we deduce that $H$ cannot be $K_k$, since then $\{x,y\}$ would form a 2-cut in $G$, contradicting \Cref{min2Cut}.

Having shown that $H = K_k$ is impossible, it only remains to show that $H$ cannot be an Ore-composition of two $k$-Ore graphs $H_1, H_2$. Clearly there must exist $u,v \in \partial_G R$ with $u \in V(H_1) \setminus V(H_2), v \in V(H_2) \setminus V(H_1)$, as otherwise the overlap vertices of $H_1, H_2$ would be a 2-cut in $G$. Thus, the conditions for \Cref{claim:i1GivenBoundary} hold, and \Cref{claim:i1GivenBoundary} implies that $H$ cannot be an Ore-composition of $H_1, H_2$, which completes the argument.
\end{proof}

\section{6-Critical Graphs}\label{sec:6critical}

We now prove \Cref{MAINTHM} for $k = 6$. The proof is divided into two parts. In this section, we first analyze the local, structural properties of 6-critical graphs by applying the results of \Cref{sec:potential} and \Cref{sec:cloning}. We then combine this with discharging in \Cref{sec:discharging} to complete the proof.

Let
$$\varepsilon = \frac{1}{105},~ \delta = 10\varepsilon = \frac{10}{105},~ P = \frac{20}{21},~ Q = \frac{2}{7},~ \buc = \frac{32}{21},~ \gap = 2 + \frac{10}{105}$$
One can verify by routine calculations that these values satisfy Assumption 1 (see \Cref{sec:potential_assumptions}), and that the inequalities arising in the discharging analysis of \Cref{sec:discharging} are satisfied for $\varepsilon = \frac{1}{105}$.

\begin{LMA}\label{minIsTightUngemmed}
If $G$ is a smallest counterexample to \Cref{MAINTHM}, then $G$ is 3-tight and ungemmed.
\end{LMA}
\begin{proof}
Suppose not. Since $p(G) > k(k-3) - P$ and $Q \geq 3\delta$, $G$ is 3-tight. By \Cref{minNoEdgeAdd}, $G$ does not admit an edge-addition and hence $G$ does not contain a diamond $D$. If $G$ contains an emerald $D$, then by \Cref{minNoEdgeAdd} and \Cref{emeraldCluster}, the vertices of $D$ have a common neighbor $v$ not in $D$. But then $G \lb V(D) \cup \{v\}\rb = K_k$, a contradiction since $G$ is not $k$-Ore. Hence $G$ is $3$-tight and ungemmed, a contradiction.
\end{proof}

In \Cref{sec:discharging}, we will assign charge to $V(G)$ in such a way that vertices of degree 5 have negative initial charge, and all other vertices have positive initial charge. Our goal is to discharge the vertices so that every vertex has non-negative charge, which will produce a contradiction. To that end, we first study the structural properties of degree 5 vertices.

\begin{DEF}
For each $i\ge 5$, $D_i(G)$ denotes the subgraph of $G$ induced by vertices of degree $i$.
\end{DEF}

\begin{LMA}\label{3inE}
If $H \subseteq G$ is a split $k$-Ore subgraph with $V(H) \subsetneq V(G)$ and split vertices $\{a,b\}$, then $d_H(a), d_H(b) \geq 3$.
\end{LMA}
\begin{proof}
Suppose not. We may assume without loss of generality that $d_H(a) \le 2$. Let $G'=G \lb V(H) \rb + \{bu: u\in N_H(a)\}$. Since $N_{G'}(b)$ contains $N_H(a)\cup N_H(b)$, it follows that $G'$ contains a 6-Ore subgraph and hence $G$ admits a 2-edge-addition, contradicting by \Cref{buckets}.
\end{proof}

We will frequently make use of \Cref{3inE}. In particular, we have the following.

\begin{LMA}\label{RepDeg}
Let $x \in V^\ast(H)$ for a gadget $H$ of $G$. If $d_G(x) = 5$, then $x$ is not incident to a replacement edge of $H$, and if $d_G(x) \in \{6,7\}$, then $x$ is incident to at most 1 replacement edge of $H$.
\end{LMA}
\begin{proof}
It follows from \Cref{3inE} that $d_G(x) \ge 4 + 2r$ where $r$ is the number of replacement edges of $H$ that $x$ is incident to. Hence if $d_G(x)=5$, then $r=0$ and if $d_G(x) \le 7$, then $r\le 1$.
\end{proof}

\begin{CORO}\label{oneCinH}
Let $H$ be a proto-gadget of $G$. If $x,y \in V^\ast(H)$ and $d_G(x) = d_G(y) = 5$, then $x$ and $y$ are in the same cluster of $G$.
\end{CORO}
\begin{proof}
By \Cref{RepDeg}, neither $x$ nor $y$ are incident to a replacement edge of $H$. Hence $x$ and $y$ are adjacent and $|N(x)\cap N(y)|\ge 3$. By \Cref{emeraldCluster}, $x$ and $y$ are in the same cluster of $G$ as desired.
\end{proof}

If $C$ is a cluster in $G$, $x\in C$ and $y\in N(x)\setminus C$, then we say $y$ is a \emph{neighbor of $C$}.

\begin{LMA}\label{clusterNeighbor6}
If $G$ is a 1-tight, ungemmed graph and $C$ is a cluster of size 2 in $G$, then $C$ has at most 1 neighbor having degree at most 6. Furthermore, if $C$ has a neighbor having degree at most 6, then $C$ is contained in a proto-gadget of $G$.
\end{LMA}
\begin{proof}
Let $\{v_1,v_2,v_3,v_4\}$ denote the neighbors of $C$. Suppose $C$ has a neighbor having degree at most $6$. Suppose without loss of generality $d(v_1) \leq 6$. By \Cref{cloning}, $\Gclone{v_1}{C}$ is $6$-critical or $\Gclone{v_1}{C}$ contains a $6$-Ore subgraph. \Cref{tightClone} implies that $\Gclone{v}{C}$ is tight and \Cref{ungemmedClone} implies that $\Gclone{v}{C}$ is ungemmed. But $\Gclone{v}{C}$ has a cluster of size $3$, and therefore is not $6$-critical, or else it would contradict \Cref{k3Cluster}. Hence $\Gclone{v}{C}$ contains a $6$-Ore subgraph $H_1$. By \Cref{cloneKey}, $V^\ast(H_1) = C \cup \{\wt{v_1},v_2,v_3,v_4\}$, and $H_1 - \wt{v_1}$ is a proto-gadget of $G$ containing $C$. This proves the second statement.
	
To prove the first statement, let us further suppose for the purposes of contradiction that $C$ has a second neighbor having degree at most $6$. Suppose without loss of generality that $d(v_2) \leq 6$. Since $v_2$ has degree at most 7 in $\Gclone{v_1}{C}$, \Cref{3inE} implies that $v_2$ is adjacent to at least one of $v_3$ and $v_4$. Without loss of generality, we may assume that $v_2$ is adjacent to $v_3$. Suppose $d(v_2) = 5$. Hence $v_2$ has degree 6 in $\Gclone{v_1}{C}$ and \Cref{RepDeg} implies that $v_2$ is adjacent to both $v_3,v_4$. But then \Cref{emeraldCluster} implies that $v_2$ is in the cluster $C$, a contradiction. 
	
So we may assume $d(v_2) = 6$. Now consider $\Gclone{v_2}{C}$. By the same reasoning as above, $\Gclone{v_2}{C}$ contains a 6-Ore subgraph $H_2$ with $V^\ast(H_2) = C \cup \{\wt{v_2}, v_1, v_3, v_4\}$. Using the replacement edges of $H_2$, we find that $G \lb (V(H_2) \setminus \{\wt{v_2}\})\cup \{v_2\}\} \rb + v_2v_1 + v_2v_4$ contains a 6-Ore subgraph. If $(V(H_2)\setminus \{\wt{v_2}\}) \cup \{v_2\} \subsetneq V(G)$, then $G$ admits a 2-edge-addition, contradicting \Cref{buckets}. So we may assume that $(V(H_2) \setminus \{\wt{v_2}\}) \cup \{v_2\} = V(G)$. But now $|V(G)| = |V(H_2)|$ and $|E(G)| \geq |E(H_2)| + 1$. Since $H_2$ is 6-Ore, we find that the potential of $G$ is at most
$$p(G) \leq p(H_2) - 2(k-1) + \delta < k(k-3) - P - Q,$$
which contradicts that $G$ is 1-tight.
\end{proof}	 

\begin{LMA}\label{K4}
If $G$ is a 2-tight, ungemmed graph, then $D_5(G)$ does not contain a subgraph isomorphic to $K_4$.
\end{LMA}
\begin{proof}
Suppose to the contrary that $v_1,v_2,v_3,v_4$ have degree 5 and $G \lb \{v_1,v_2,v_3,v_4\}\rb = K_4$. Let $u_1,u_2$ be the other neighbors of $v_1$. By \Cref{cloning}, $\Gclone{v_2}{v_1}$ is $6$-critical or $\Gclone{v_2}{v_1}$ contains a 6-Ore subgraph. 
	
First suppose $\Gclone{v_2}{v_1}$ has a 6-Ore subgraph $H$. Then by \Cref{cloneKey}, $V^\ast(H) = \{v_1, \wt{v_1}, v_3, v_4\}$. By \Cref{RepDeg} as $v_3,v_4$ have degree 5, we find that $v_3,v_4$ are not incident with a replacement edges of $H$. Hence $v_3,v_4$ are each adjacent to both of $u_1, u_2$. But then \Cref{emeraldCluster} implies that $v_1,v_3,v_4$ are in the same cluster, contradicting \Cref{k3Cluster}. 
	
So we may assume that $\Gclone{v_2}{v_1}$ is 6-critical. By \Cref{tightClone,ungemmedClone}, $\Gclone{v_2}{v_1}$ is tight and ungemmed. But $\Gclone{v_2}{v_1}$ contains a cluster of size 2 with two neighbors having degree 5, which contradicts \Cref{clusterNeighbor6}.
\end{proof}

\begin{LMA}\label{2path}
Let $G$ be a 2-tight, ungemmed graph. If $P=v_1v_2v_3$ is an induced path such that $d_G(v_1)=d_G(v_2)=d_G(v_3)=5$, then $\Gclone{v_1}{v_2}$ and $\Gclone{v_3}{v_2}$ are 1-tight and ungemmed.
\end{LMA}
\begin{proof}
Suppose not. First suppose that $\Gclone{v_3}{v_2}$ is not 6-critical. By \Cref{cloning} and \Cref{cloneKey}, $\Gclone{v_3}{v_2}$ contains a 6-Ore subgraph $H$ with $V^\ast(H) = (N(v_2) \setminus \{v_3\}) \cup \{\wt{v_2}\}$. The vertex $v_1$ has degree 6 in $\Gclone{v_3}{v_2}$. By \Cref{RepDeg}, $v_1$ is not incident with a replacement edge of $H$. Hence $v_1$ is adjacent to the other 3 neighbors of $v_2$. But then \Cref{emeraldCluster} implies that $v_1,v_2$ are in the same cluster and hence $v_1$ is adjacent to $v_3$, a contradiction since $P$ is induced. 
	
So we may assume that $\Gclone{v_3}{v_2}$ is 6-critical. By \Cref{tightClone} $p(\Gclone{v_3}{v_2}) \geq p(G) - \delta > p(G) - P - Q$ and hence $\Gclone{v_3}{v_2}$ is tight. By \Cref{ungemmedClone}, $\Gclone{v_3}{v_2}$ is ungemmed. By symmetry $\Gclone{v_1}{v_3}$ is also tight and ungemmed, a contradiction.
\end{proof}

\begin{LMA}\label{loneK3}
Let $G$ be a 2-tight, ungemmed graph. If a component of $D_5(G)$ contains a triangle, then that component is a triangle.
\end{LMA}
\begin{proof}
Suppose not. Thus we may assume there exists a triangle $T=v_1v_2v_3$ in $D_5(G)$ and $u \in D_5(G)$ adjacent to $v_1$. Note that in $\Gclone{v_2}{v_1}$, $v_3$ has degree 5 and $u$ has degree at most 6. Hence we deduce from \Cref{clusterNeighbor6} that $\Gclone{v_2}{v_1}$ is not 6-critical. Thus \Cref{cloning} and \Cref{cloneKey} imply that $\Gclone{v_2}{v_1}$ contains a 6-Ore subgraph $H$ with $v_3, u \in V^\ast(H)$. Since $v_3$ has degree 5 in $G$, it follows from \Cref{RepDeg} that $v_3$ is not incident to a replacement edge of $H$. Hence $v_3u \in E(G)$. Therefore, both $v_2$ and $u$ have degree 5 in $\Gclone{v_3}{v_1}$. 
	
Similarly by \Cref{clusterNeighbor6}, $\Gclone{v_3}{v_1}$ is not 6-critical. Hence $\Gclone{v_3}{v_1}$ has a 6-Ore subgraph $H$ with $v_2,u \in V^\ast(H)$. Since both $v_2$ and $u$ have degree 5 in $G$, it follows from \Cref{RepDeg} that neither $v_2$ nor $u$ is incident to a replacement edge of $H$. Hence $v_2u \in E(G)$. But then $G \lb \{v_1,v_2,v_3,u \} \rb = K_4$, contradicting \Cref{K4}.
\end{proof}

\begin{LMA}\label{noC4}
If $G$ is a 2-tight, ungemmed graph, then $D_5(G)$ does not contain a 4-cycle.
\end{LMA}
\begin{proof}
Suppose to the contrary that $v_1v_2v_3v_4v_1$ is a 4-cycle in $D_5(G)$. By \Cref{loneK3}, $v_1v_3, v_2v_4 \notin E(G)$. By \Cref{2path}, $\Gclone{v_3}{v_2}$ is 1-tight and ungemmed. But $v_4$ has degree 4 in $\Gclone{v_3}{v_2}$, a contradiction.
\end{proof}

\begin{LMA}\label{noDeg3}
If $G$ is a 2-tight, ungemmed graph, then the maximum degree of $D_5(G)$ is at most $2$.
\end{LMA}
\begin{proof}
Suppose to the contrary that $u\in D_5(G)$ is a vertex with neighbors $v_1,v_2,v_3\in D_5(G)$. By \Cref{loneK3}, we find that $v_1$ is not adjacent to $v_2$ and hence $v_1uv_2$ is an induced path in $G$. Then by \Cref{2path}, $\Gclone{v_1}{u}$ is a 1-tight and ungemmed graph with a cluster $\{u, \wt{u}\}$ that has two neighbors $v_2,v_3$ having degree at most 6 in $\Gclone{v_1}{u}$, contradicting \Cref{clusterNeighbor6}.
\end{proof}

\begin{LMA}\label{d5struct}
Let $G$ be a 2-tight, ungemmed graph. Every component of $D_5(G)$ has size at most 3.
\end{LMA}
\begin{proof}
From \Cref{3path}, \Cref{loneK3}, \Cref{noC4}, and \Cref{noDeg3}, we find that no component of $D_5(G)$ can have more than 3 vertices.
\end{proof}

The possible components in $D_5(G)$ are shown in \Cref{fig:D5G}. Vertices depicted as open circles are in a cluster of size 2.

\begin{figure}
\centering
\newcommand\gOffset{2.5}
\begin{tikzpicture}[]
\tikzstyle{vertex}=[circle,fill=black!100,minimum size=6pt,inner sep=2pt]
\tikzstyle{cluster}=[circle, draw, inner sep=2pt, minimum size=6pt]
 \node[vertex] (A_1) at (1/2*\gOffset,0) {}; \draw (A_1);
 
\node[vertex] (B_1) at (\gOffset, 0) {}; \node[vertex] (B_2) at (\gOffset+3/5,1/3) {}; \node[vertex] (B_3) at (\gOffset+6/5,0) {};
\draw (B_1) -- (B_2) -- (B_3);

\node[cluster] (C_1) at (2*\gOffset,0) {}; \node[cluster] (C_2) at (2*\gOffset+1, 0) {};
\draw (C_1) -- (C_2); 
 
\node[cluster] (D_1) at (3*\gOffset,-1/4) {}; \node[cluster] (D_2) at (3*\gOffset+1, -1/4) {}; \node[vertex] (D_3) at (3*\gOffset + 1/2, 1.2/2) {};
\draw (D_1) -- (D_2) -- (D_3) -- (D_1);

\node[vertex] (E_1) at (4*\gOffset, 0) {}; \node[vertex] (E_2) at (4*\gOffset+ 1, 0) {};
\draw (E_1) -- (E_2);

\node[vertex] (F_1) at (5*\gOffset,-1/4) {}; \node[vertex] (F_2) at (5*\gOffset+1,-1/4) {}; \node[vertex] (F_3) at (5*\gOffset+1/2, 1.2/2) {};
\draw (F_1) -- (F_2) -- (F_3) -- (F_1);
\end{tikzpicture}
\caption{Components of $D_5(G)$. Open circles represent vertices in the same cluster.}
\label{fig:D5G}
\end{figure}
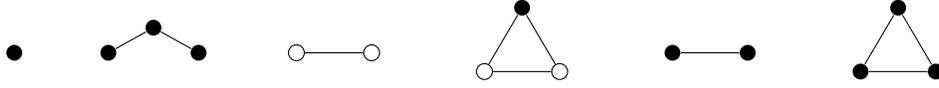

\begin{LMA}\label{noCommonProto7}
Let $G$ be a 2-tight and ungemmed graph, and let $H \subseteq G$ be a proto-gadget. If $x \in V^\ast(H)$ has degree 5 in $G$, and $y \notin V^\ast(H)$ is a neighbor of $x$ having degree at most 7 in $G$, then there does not exist a proto-gadget $H'$ with $x,y \in V^\ast(H')$.
\end{LMA}
\begin{proof}
Suppose not. That is, there exists a proto-gadget $H'$ with $x,y\in V^\ast(H')$. Since $x$ has degree 5 in $G$, by \Cref{RepDeg}, $x$ is not incident with a replacement edge of $H$ or $H'$. Let $\{z_1,\ldots, z_4\}$ denote the neighbors of $x$ in $V^\ast(H)$. Since $x \in V^\ast(H')$, exactly one of $z_1,\ldots, z_4$ is not in $V^\ast(H')$; we may assume without loss of generality that $z_1 \notin V^\ast(H')$. 
	
First suppose $y$ has degree 5 in $G$. Then by \Cref{RepDeg}, $y$ is not incident with a replacement edge of $H'$. Hence $y$ is adjacent to all of $z_2, z_3,$ and $z_4$. By \Cref{emeraldCluster}, $x$ and $y$ are in the same cluster. But then $G'=G[H\cup \{y\}]$ is a $k$-Ore subgraph of $H$. Since $G$ is $6$-critical, we find that $G=G'$. By \Cref{kOreGem}, it follows that $G$ contains a gem, a contradiction.

So we may assume that $y$ has degree 6 or 7 in $G$. Since $y$ has degree at most 7 in $G$, by \Cref{RepDeg} $y$ is incident with at most 1 replacement edge in $H'$. Thus, $y$ is adjacent to at least two of the vertices $\{z_2, z_3, z_4\}$. We may assume without loss of generality that $y$ is adjacent to $z_3$ and $z_4$. Adding the edges $yz_1, yz_2$ to the proto-gadget $H$ yields a 6-Ore subgraph $K$ of $G$. 
	
We claim that $V(K) = V(H) \cup \{y\}$ is a proper subset of $V(G)$. Let $S = E(G) \setminus E(K)$. Since $y$ has degree 6 or 7, $|S| \geq 3$. Therefore, since $E(K) = E(H) \cup \{yx, yz_1, yz_2, yz_3, yz_4\}$, we have that $|E(G)| \geq |E(K)| + 1$. If $|V(G)| = |V(K)|$, then $p(G) \leq p(K) - 2(k-1) + \delta < k(k-3) - P - Q$,
which contradicts that $G$ is 2-tight. Therefore $V(K) \subsetneq V(G)$. But then $G$ admits a 2-edge-addition, contradicting \Cref{buckets}.
\end{proof}

\begin{CORO}\label{noProto7}
Let $G$ be a 2-tight and ungemmed graph, and let $H \subseteq G$ be a proto-gadget. Suppose $x \in V^\ast(H)$ has $d_G(x) = 5$ and $y$ is a neighbor of $x$ having degree at most 7. If $y \notin V^\ast(H)$, then $y$ is the \emph{unique} neighbor of $x$ with the property that there is no proto-gadget $H'$ with $x,y \in V^\ast(H')$.
\end{CORO}

In \Cref{sec:discharging}, we will see that proto-gadgets have large initial charge. Vertices $v$ of degree 6 or 7 incident to a replacement edge $e$ should avoid sending charge to neighbors inside $e$, as those neighbors can receive sufficient charge from proto-gadgets inside $e$. This allows us to distribute the charge from $v$ more efficiently. To make this precise, we introduce the following notion.

\begin{DEF}
Let $x$ be a vertex of degree 5, and $y$ a neighbor of $x$ having degree 6 or 7. We say $x$ is a \emph{downward neighbor} of $y$ if there exists a proto-gadget $H$ with $x \in V^\ast(H), y \notin V^\ast(H)$. If $x$ is not a downward neighbor of $y$, then we say that $x$ is an \emph{upward neighbor} of $y$.
	
Similarly if $C$ is a cluster and $y$ is a neighbor of $C$ having degree 6 or 7, then we say $y$ is a \emph{downward (resp. upward) neighbor} of $C$ if $y$ is a downward (resp. upward) neighbor of every vertex $v\in C$.
\end{DEF}

Note that for the last definition with clusters, a neighbor is either downward or upward. This follows since if $y$ is downward of some vertex $v\in C$, then $y$ is downward of every vertex in $C$. To see this, note that there exists a proto-gadget $H$ with $v\in V^\ast(H), y\notin V^\ast(H)$, but then $C\setminus \{v\} \subseteq N(v)\setminus \{y\}\subseteq V^\ast(H)$ by \Cref{RepDeg} as $v$ has degree $5$ in $G$.

\begin{CORO}\label{downward}
Suppose $H \subseteq G$ is a split $k$-Ore subgraph with split vertices $\{y,z\}$. If $y$ has degree $6$ or $7$ and is adjacent to a vertex $x \in V(H)$ having degree 5, then $x$ is a downward neighbor of $y$.
\end{CORO}
\begin{proof}
By \Cref{splitProto}, there exists a proto-gadget $H' \subseteq H - \{y,z\}$ with $x \in V^\ast(H')$. Clearly $y \notin V(H')$. So by definition, $x$ is a downward neighbor of $y$.
\end{proof}

\begin{CORO}\label{upward}
If $H \subseteq G$ is a proto-gadget with $x,z \in V^\ast(H)$ such that $x$ has degree $5$ in $G$ and $z$ has degree $6$ or $7$, then $x$ is an upward neighbor of $z$.
\end{CORO}
\begin{proof}
Since $x$ has degree $5$ in $G$, by \Cref{RepDeg}, $x$ is not incident with a replacement edge of $H$. Hence there exists a unique neighbor $y$ of $x$ with $y \notin V^\ast(H)$. By \Cref{noProto7}, $y$ is the unique neighbor of $x$ such that there does not exist a proto-gadget $H'$ with $x,y \in V^\ast(H')$. Suppose to the contrary that $x$ is a downward neighbor of $z$. Then there exists a proto-gadget $H'$ with $x \in V^\ast(H'), z \notin V^\ast(H')$. Since $d_G(x) = 5$ and $d_{H'}(x)\ge 4$, we find that $y \in V^\ast(H')$, a contradiction.
\end{proof}

\section{Discharging}\label{sec:discharging}
Suppose that $G$ is a minimal counterexample to \Cref{MAINTHM}. By \Cref{minIsTightUngemmed}, $G$ is 3-tight and ungemmed. Recall that $\varepsilon = \frac{1}{105}$; in this section, $\varepsilon$ is the only constant which appears, and the reader can verify that the numerical inequalities are satisfied by $\varepsilon = \frac{1}{105}$.
\begin{DEF}
The initial charge $ch_0: V(G) \rightarrow \mathbb{R}$ is
$$ch_0(v) = (k-1)d(v) - (k-2)(k+1) - \varepsilon.$$
In particular, when $k = 6$, the initial charge is $5d(v) - 28 - \varepsilon$, which is equal to $-3-\varepsilon, 2-\varepsilon, 7-\varepsilon, 12-\varepsilon,\ldots$, etc. resp. for vertices of degree $5,6,7,8,\ldots$, etc.
\end{DEF}
The total charge of the vertices of the graph is
$$\sum_{v \in V(G)} ch_0(v) = 2(k-1)|E(G)| - ( (k-2)(k+1) + \varepsilon)|V(G)| = -p(G) - \delta T(G).$$
As $G$ is a counterexample to \Cref{MAINTHM}, $-p(G) - \delta T(G) < -(k(k-3) - P) - \delta T(G) < 0$. Our objective is to transfer charge to the vertices of degree 5, which have negative initial charge, until every vertex has non-negative charge, yielding a contradiction.

We will move charge between vertices in five stages, each composed of discharging rules. The charge after the termination of \textsc{stage} $i$ will be denoted $ch_i(v)$. Several of the stages involve successive rounds of discharging, which `trigger' when conditions are met.

\begin{DEF}
For a subset $S \subseteq V(G)$, define $ch_i(S) = \sum_{v \in S} ch_i(v)$. A cluster $C$ is \emph{satisfied} if $ch_i(C) \geq 0$, and \emph{happy} if $ch_i (C) \geq 2 + 2\varepsilon$. We say $C$ is \emph{unsatisfied} (resp. \emph{unhappy}) if it is not satisfied (resp. happy). A vertex $v$ in the cluster $C$ is satisfied (resp. happy) if $C$ is satisfied (resp. happy). A component $S$ of $D_5(G)$ is \emph{satisfied} if $ch_i(S) \geq 0$.
\end{DEF}

The discharging process is quite involved. For clarity, we first provide a high-level overview of the rules and objectives of each stage.
\begin{enumerate}
\item In the first two stages (\Cref{subs:trigger}), vertices with a large amount of excess positive charge release their charge to degree 5 neighbors. In \textsc{stage 1}, vertices of degree at least 8  unrestrictedly release charge equally to their degree 5 neighbors.

In \textsc{stage 2}, vertices of degree 6 and 7 release charge to upward neighbors only when certain rules are triggered. \textsc{stage 2} ends when the trigger condition is not met by any vertex. After \textsc{stage 2}, degree-5 vertices which are in gadgets, or have high-degree neighbors, are happy.

\item In \textsc{stage} 3 (\Cref{subs:remnants}), we discharge the remaining degree-6 and degree-7 vertices. All undischarged vertices of degree 7 release their charge, and the remaining vertices of degree 6 discharge when a specific condition is met. Vertices of degree 6 which have not yet discharged by the end of \textsc{stage 3} will be called \emph{reserved vertices}. 

Reserved vertices have a strict structural characterization. The degree-5 vertices which are not satisfied after \textsc{stage 3} must either be adjacent to a large number of reserved vertices, or else also have a specific structural characterization (which we call \emph{dangling} vertices).

In \Cref{subs:accounting}, we analyze the result of the previous stages to determine which vertices still require charge after \textsc{stage 3}. Clusters of size 2 and induced paths of length 2 in $D_5(G)$ are satisfied. Components of size two in $D_5(G)$ whose vertices are not in the same cluster and triangles in $D_5(G)$ will be satisfied unless they contain dangling vertices.

\item In \textsc{stages} 4, 5 and 6 (\Cref{subs:global}), we perform \emph{global discharging} by transferring the aggregate charge from the set of reserved vertices to a set of unsatisfied degree-5 vertices. This provides sufficient charge to satisfy the singletons of $D_5(G)$ and the remaining dangling vertices.
\end{enumerate}

\subsection{Triggered Discharging}\label{subs:trigger}

In \textsc{stage} 1, vertices of degree at least 8 discharge to their degree-5 neighbors as follows.
\begin{description}
\item[Rule 1] Every vertex $v$ with $d(v) \geq 8$ and $r$ neighbors of degree $5$ sends $\frac{ch_0(v)}{r}$ to each neighbor of degree $5$.
\end{description}

In \textsc{stage} 2, vertices of degree $6,7$ discharge when specific conditions are met. \textsc{stage} 2 terminates when no vertex triggers a condition to discharge.
\begin{description}
\item[Rule 2A] Every vertex $v$ with $d(v) = 7$ and $r \leq 5$ unhappy upward neighbors sends $\frac{ch_0(v)}{r}$ to each unhappy upward neighbor.
\item[Rule 2B] Every vertex $v$ with $d(v) = 6$ neighboring exactly one unhappy upward cluster sends $2 - \varepsilon$ to that cluster.
\end{description}

Define the function
$$\psi(d,r) =  \frac{5d - 28 - \varepsilon}{r}$$
to be the amount of initial charge sent by a vertex of degree $d$ to $r$ neighbors of degree 5. A simple calculation yields the following lemma, which we will repeatedly use when analyzing the discharging.

\begin{LMA}\label{chargeBounds} All of the following hold.
\begin{enumerate}
\item For fixed $d\ge 1$,	$\psi(d,r)$ is decreasing in $r$.
\item For fixed $r\ge 1$, $\psi(d,r)$ is increasing in $d$.
\item For a fixed $0 \leq i \leq 5$, $\psi(d, d-i)$ is increasing in $d$. 
\item A vertex $v$ of degree $d \geq 8$ sends at least $\frac{3}{2} - \frac{1}{8}\varepsilon$ charge in \textsc{stage 1} to each neighbor of degree 5.
\item A vertex $v$ of degree $d \geq 8$ and $r < d$ neighbors of degree 5 always sends at least $\psi(8,7) = \frac{12}{7} - \frac{1}{7}\varepsilon$ charge to each neighbor of degree 5.
\item A vertex $v$ of degree 7 sends at least $\psi(7,5) = \frac{7}{5} - \frac{1}{5}\varepsilon$ charge in \textsc{stage 2} to each neighbor of degree 5 to which $v$ sends charge.	\end{enumerate}
Note $\psi(8,8) \geq \psi(7,5)$, so the minimum amount of charge sent to a vertex by another vertex in \textsc{stages 1, 2} is $\psi(7,5)$.
\end{LMA}
\begin{proof}
Note that $\frac{\partial \psi(d,r)}{\partial r} = \frac{28 + \varepsilon - 5d}{r^2} \leq 0$ as $\varepsilon \leq 7$; so $\psi(d,r)$ is decreasing in $r$ and (1) holds. Note that $\frac{\partial \psi(d,r)}{\partial d} = \frac{5}{r} \geq 0$ as $r \geq 1$; so $\psi(d,r)$ is increasing in $d$ and (2) holds. For $0 \leq i \leq 5$, we have that $\frac{\partial \psi(d,d - i)}{\partial d} = \frac{28 + \varepsilon - 5i}{(d - i)^2} \geq 0$; so $\psi(d, d-i)$ is increasing in $d$ and (3) holds. 
	
A vertex of degree $d \geq 8$ sends $\psi(d,r)$ charge in \textsc{stage 1} each neighbor of degree $5$. By (1), $\psi(d,r) \ge \psi(d,d)$. By (3) with $i=0$, $\psi(d,d) \ge \psi(8,8)$.  Since $\psi(8,8) = \frac{3}{2} - \frac{1}{8}\varepsilon$, (4) holds as desired. 
	
So suppose $r < d$. By (1), $\psi(d,r) \geq \psi(d, d-1)$. By (3) with $i=1$, $\psi(d,d-1) \geq \psi(8, 7)$. Since $\psi(8,7) = \frac{12}{7} - \frac{1}{7}\varepsilon$, (5) holds as desired.
	
By \textsc{rule 2a}, a vertex of degree 7 discharges if it has $r \leq 5$ unhappy upward neighbors. The amount of charge sent to each such neighbor is $\psi(7,r)$. By (1) and since $r\leq 5$, $\psi(7,r) \geq \psi(7,5)$. Since $\psi(7,5) = \frac{7}{5} - \frac{1}{5}\varepsilon$, (6) holds as desired.
\end{proof}

\begin{RMK}
To simplify the presentation, we adopt the following convention. When we make a statement such as ``$v$ triggers \textsc{rule 2a} and sends charge to $C$'', it should be understood that the discharge only occurs if the receiving cluster is not already happy. This way, we will avoid repeatedly stating the condition `if $C$ is still unhappy' before every invocation of a discharging rule.
\end{RMK}

\begin{LMA}\label{dischargeProtoGadget}
Let $H$ be a proto-gadget of $G$, and $C \subseteq V(H)$ a cluster. Then $C$ is happy after \textsc{stage 2}.
\end{LMA}
\begin{proof}
Suppose not. Let $H$ and $C$ be a counterexample such that $|V(H)|$ is minimized. First suppose that $C \setminus V^\ast(H) \ne \emptyset$. Let $w\in C\setminus V^\ast(H)$. Hence $w$ is inside a replacement edge $e=uv$ of $H$. Since $e$ is a split $6$-Ore graph with split vertices $\{u,v\}$, it follows that $d_H(w) = 5 = d_G(w)$ and hence $N(w)\cap V(H) = N(w)\cap V(G)$. Since $u$ and $v$ each have a neighbor outside $e$, it follows that $u,v\notin C$ and hence $C$ is entirely inside $e$. By \Cref{splitProto}, there exists a proto-gadget $H' \subseteq H - \{a,b\}$ with $C \subseteq V(H')$. Since $|V(H')|<|V(H)|$, $H'$ and $C$ contradict the minimality of $H$ and $C$.

So we may assume that $C\subseteq V^\ast(H)$. Next, note that \Cref{oneCinH} implies that $V^\ast(H)\cap D_5(G) =C$. Furthermore, \Cref{upward} implies that $C$ is upward of each neighbor of $C$ having degree $6$ or $7$. Let $\{w\}=N(C)\setminus V^\ast(H)$. By \Cref{buckets} and \Cref{minIsTightUngemmed}, $G$ does not admit two edge-additions. This implies that $|N(w)\cap V(H)| \le 2$ and hence $|N(w) \cap V(H) \setminus C| \le 2-|C|$. Let $D$ denote the set of vertices of $V^\ast(H)\cap D_6(G)$ that are incident with a replacement edge of $H$. Note that if $v\in D$, then since $d(v)=6$, it follows that $v$ is incident with at most one replacement edge of $H$.
	
\begin{claim}\label{notD}
If $v\in V^\ast(H)\cap D_6(G) \setminus D$, then $v$ is adjacent to $w$; hence $|V^\ast(H)\cap D_6(G) \setminus D| \le 2-|C|$. 
\end{claim}
\begin{proof}
Follows from \Cref{kNeighbor}.
\end{proof}
	
\begin{claim}
$H\ne K_5$.
\end{claim}
\begin{proof}
Suppose not. Since $H=K_5$, $D=\emptyset$. Then by \Cref{notD}, it follows that $|V^\ast(H)\cap D_6(G)| \leq 2-|C|$.  Let $S=V(H)\setminus (D_6(G)\cup D_5(G))$. Since $V(H)\cap D_5(G) = C$, it follows that $|S| \ge 5-(2-|C|)-|C| \ge 3$. Moreover, each vertex in $S$ has degree at least 7. Since $S$ is a clique, it follows that for each vertex $s\in S$, $|N(s)\cap D_5(G)|\le d(s)-2$ and hence $s$ either discharges in \textsc{stage 1} or triggers \textsc{rule 2a}.  

First suppose $|C| = 2$.  Thus each vertex in $C$ receives at least 
$$\sum_{s\in S} \psi(d(s), d(s) - 2) \geq 3 \psi(7,5) = \frac{21}{5} - \frac{3}{5}\varepsilon$$
total charge from vertices of $S$. This discharge leaves $C$ with $ch_2(C) \geq -6-2\varepsilon + 2\left( \frac{21}{5} -  \frac{3}{5}\varepsilon \right) \geq 2 + \frac{2}{5} - \frac{16}{5}\varepsilon$. Since $\varepsilon \le \frac{1}{13}$, we find that $ch_2(C) \geq 2 + 2\varepsilon$ and hence $C$ is happy after \textsc{stage 2}, a contradiction.
	
So we may assume that $|C|=1$. Let $C = \{x\}$. Note that now for each vertex $s\in S$, $|N(s)\cap D_5(G)|\le d(s)-3$. Thus $x$ receives at least 
$$\sum_{s\in S}^3 \psi(d(s), d(s) - 3) \geq 3\psi(7,4)$$
charge from vertices in $S$. This discharge then leaves $x$ with at least $ch_2(x) \geq -3 - \varepsilon + \frac{21}{4} - \frac{3}{4}\varepsilon \geq 2 + \frac{1}{4} - \frac{7}{4}\varepsilon$. Since $\varepsilon \leq \frac{1}{15}$, we find that $ch_2(x)\geq 2 + 2\varepsilon$ and hence $x$ is happy after \textsc{stage 2}, a contradiction.
\end{proof}

\begin{claim}\label{C1}
$|C|=1$.
\end{claim}
\begin{proof}
Suppose not. Thus $|C|=2$. Let $C = \{x,y\}$. Let $\{v_1,v_2,v_3\} = V^\ast(H)\setminus C$. \Cref{clusterNeighbor6} implies that $V^\ast(H)$ has at most 1 vertex of degree 6.
	
Now further suppose that $V^\ast(H)$ has no vertices of degree 6. Then $\opn{deg}(v_j) \geq 7$ for each $j\in \{1,2,3\}$. By \Cref{downward}, it follows that $d(v_j) \geq 8$ or $v_j$ has at most $5$ upward neighbors. Hence each $v_j$ then discharges either in \textsc{stage 1} or \textsc{rule 2a}, sending at least $\min\{ \psi(8,8), \psi(7,5)\} = \frac{7}{5} - \frac{1}{5}\varepsilon$ charge to both of $x,y$. Thus $ch_2(C) \geq -6 - 2\varepsilon + 2(\frac{21}{5} - \frac{3}{5}\varepsilon) \geq 2 + \frac{2}{5} - \frac{16}{5}\varepsilon$. Since $\varepsilon \le \frac{1}{13}$, we find that $ch_2(C) \geq 2 + 2\varepsilon$ and hence $C$ is happy after \textsc{stage 2}, a contradiction.
	
So we may assume that $V^\ast(H)$ has a single vertex of degree 6, say $v_1$ without loss of generality. Since $|C|=2$, it follows from \Cref{notD} that $v_1\in D$ and hence $v_1$ is incident with a unique replacement edge $e$ of $H$. Suppose that $e=v_1v_2$ without loss of generality. \Cref{downward} implies that any neighbors of $v_1$ inside $e$ are not upward of $v_1$. Hence $v_1$ has no upward neighbors aside from $C$ and therefore triggers \textsc{rule 2b}, sending $2 - \varepsilon$ charge to $C$.
	
Note that $d(v_2), d(v_3)\ge 7$ as $V^\ast(H)$ has a single vertex of degree 6. Also note that $v_3$ is incident with $v_1$ and hence $v_3$ has at most $d(v_3) - 1$ neighbors of degree 5. 

\begin{subclaim}\label{chargeTogether}
$v_2$ and $v_3$ together send at least $3 + \frac{4}{35} - \frac{12}{35}\varepsilon$ charge to each vertex in $C$.
\end{subclaim}
\begin{proof}
First suppose that $d(v_3) \geq 8$. Then $v_3$ sends at least $\psi(8,7) = \frac{12}{7} - \frac{1}{7}\varepsilon$ charge to $C$, and $v_2$ sends at least $\psi(7,5)$. Then together $v_2, v_3$ send to each vertex in $C$ at least $\psi(8,7) + \psi(7, 5)= \frac{12}{7} - \frac{1}{7}\varepsilon + \frac{7}{5} - \frac{1}{5}\varepsilon = 3 + \frac{4}{35} - \frac{12}{35}\varepsilon$ charge to each vertex in $C$ as desired.

So we may assume that $d(v_3) = 7$. Suppose that $v_2v_3$ is a replacement edge of $H$. Since $v_2v_3$ is a replacement edge of $H$, we have that $d(v_2) \geq 8$ and $v_3$ has at most 3 upward neighbors. Then together $v_2, v_3$ send to each vertex in $C$ at least $\psi(8,8) + \psi(7, 3)= 3 + \frac{5}{6} - \frac{11}{24}\varepsilon$ charge, which is at least $3 + \frac{4}{35} - \frac{12}{35}\varepsilon$ charge as desired since $\varepsilon \leq 1 \leq \frac{604}{97}$.

So we may assume that $v_2v_3$ is a real edge of $H$. Since $v_3$ is adjacent to $v_1$ and $v_2$, we find that $v_3$ has at most 5 upward neighbors. Similarly since $v_2$ is adjacent to $v_3$, we find that $v_2$ has at most $d(v_2) - 1$ neighbors of degree 5; furthermore $v_2$ has at most $d(v_2)-4$ upward neighbors since by \Cref{downward} $v_2$ is upward of its neighbors inside the replacement edge $e=v_1v_2$. Then together $v_2, v_3$ send at least $\psi(7,5) + \min\{ \psi(8,7), \psi(7,3) \} = 3 + \frac{4}{35} - \frac{12}{35}\varepsilon$ charge to each vertex in $C$ as desired. 
\end{proof}

By \Cref{chargeTogether}, the total charge sent to $C$ by $v_2, v_3$ is at least $2 \left(3 + \frac{4}{35} - \frac{12}{35}\varepsilon \right) = 6 + \frac{8}{35} - \frac{24}{35}\varepsilon$. Since $v_1$ sends at least $2-\varepsilon$ charge to $C$, it follows that the charge of $C$ is then
$ch_2(C) \geq -6 - 2\varepsilon + (2 - \varepsilon) + \left(6 + \frac{8}{35} - \frac{24}{35}\varepsilon \right) \geq 2 + \frac{8}{35} - \frac{129}{35}\varepsilon$. Since $\varepsilon \leq \frac{8}{199}$, we find that $ch_2(C) \geq 2 + 2\varepsilon$ and hence $C$ is happy after \textsc{stage 2}, a contradiction.
\end{proof}
	
By \Cref{C1}, $|C| =1$. Let $C = \{x\}$. Recall that $V^\ast(H)\cap D_5(G) = C$. 

\begin{claim}\label{deg7sends}
If $v \in V^\ast(H)$ such that $d(v) \geq 7$, then $v$ sends at least $\psi(8,7)$ charge to $x$.
\end{claim}
\begin{proof}
First suppose $d(v) = 7$. Then either $v$ is incident with only real edges of $H$, in which case $v$ has at most 4 neighbors of degree 5, or $v$ is incident with exactly one replacement edge of $H$, in which case $v$ has at most 2 upward neighbors. In either case, $v$ triggers \textsc{rule 2a} and sends at least $\psi(7,4)$ charge to $x$, which is at least $\psi(8,7)$ charge as desired. Next suppose $d(v) = 8$. Then $v$ is incident with at most 2 replacement edges and hence $v$ has at most 7 neighbors of degree 5. Thus $v$ sends at least $\psi(8,7)$ charge to $x$ as desired. Finally suppose $d(v) \geq 9$. Then $v$ sends at least $\psi(9,9)$ charge to $x$, which is at least $\psi(8,7)$ charge as desired. 
\end{proof}
	
\begin{claim}\label{deg6sends}
If $v \in D$, then $v$ triggers \textsc{rule 2b} and hence sends $2-\varepsilon$ charge to $x$. 
\end{claim}
\begin{proof}
Suppose not. Since $v\in D$, $v$ is incident to a replacement edge $e$ of $H$. Since $d(v)=6$, this is the only replacement edge of $H$ incident with $v$. Note then that $N(v)\cap D_5(G) \setminus \{x\}$ is inside $e$. By \Cref{downward}, the neighbors of $v$ inside $e$ are downward neighbors, so $x$ is the only upward neighbor of $v$. Therefore $v$ triggers \textsc{rule 2b} and sends $2 - \varepsilon \geq \psi(8,7)$ to $x$ as desired. This proves the second statement.
\end{proof}
	
By \Cref{notD}, $|V^\ast(H)\cap D_6(G) \setminus D| \le 2-|C| \le 1$. By \Cref{deg6sends}, every vertex in $D$ sends at least $2-\varepsilon$ charge to $x$. By \Cref{deg7sends}, every vertex in $V^\ast(H)\setminus (\{x\}\cup D_6(G))$ sends at least $\psi(8,7) = \frac{12}{7} - \frac{1}{7}\varepsilon$ charge to $x$. Since $2 - \varepsilon \geq \psi(8,7)$ as $\varepsilon \le \frac{1}{3}$, it follows that at least 3 of the neighbors of $x$ will each send at least $\psi(8,7)$ charge to $x$. Hence $x$ receives at least $3 \psi(8,7)= 5 + \frac{1}{7} - \frac{3}{7}\varepsilon$ charge. Thus $ch_2(x) \geq -3 - \varepsilon + 5 + \frac{1}{7} - \frac{3}{7}\varepsilon = 2 + \frac{1}{7} - \frac{10}{7}\varepsilon$; since $\varepsilon \le \frac{1}{24}$, this is at least $2+2\varepsilon$ and hence $x$ is happy after \textsc{stage 2}, a contradiction.
\end{proof}

\begin{CORO}\label{unhappyUp}
If $x$ is a vertex of degree 5 and $x$ is unhappy after \textsc{stage 2}, then $x$ is not in any proto-gadget and hence $x$ is upward of all its neighbors.
\end{CORO}
\begin{proof}
By \Cref{dischargeProtoGadget}, all clusters contained in proto-gadgets are happy after \textsc{stage 2}. Thus, $x$ is not in a proto-gadget. So by the definition of downward and upward neighbors, $x$ is upward of its neighbors. 
\end{proof}

\begin{CORO}\label{6downHappy}
If $w$ is a vertex of degree 6 such that $w$ is a split vertex of a split $k$-Ore subgraph $H$ of $G$, then the downward neighbors of $w$ in $H$ are happy after \textsc{stage 2}.
\end{CORO}
\begin{proof}
Let $x$ be a degree-5 neighbor of $w$ such that $x\in V(H)$. As $G$ is not $k$-Ore, we find that $x$ is not a split vertex of $H$. By \Cref{splitProto}, there exists a proto-gadget $H'$ contained in $H$ with $x \in V^\ast(H')$. By \Cref{dischargeProtoGadget}, $x$ is happy after \textsc{stage 2} as desired.
\end{proof}

\begin{CORO}\label{2pathEnds}
If $(x,y,z)$ is an induced path in $D_5(G)$, then $x$ and $z$ are happy after \textsc{stage 2}, that is, $ch_2(x), ch_2(z) \geq 2+2\varepsilon$.
\end{CORO}
\begin{proof}
By \Cref{cloning}, $\Gclone{y}{x}$ and $\Gclone{y}{z}$ contain 6-Ore subgraphs, since $z, x$ have degree 4 after deleting $y$. Hence there exist gadgets $H_1, H_2$ of $G$ with $x \in V^\ast(H_1), z \in V^\ast(H_2)$. Since $x,z$ are not in clusters of size 2, \Cref{dischargeProtoGadget} implies that $ch_2(x), ch_2(z) \geq 2 + 2\varepsilon$ as desired.
\end{proof}

\subsection{Discharging: Second Stage}\label{subs:remnants}

We next distribute the remaining charge.
\begin{description}
\item[Rule 3A]: Every vertex $v$ with $d(v) \geq 7$ distributes all remaining positive charge equally to its degree-5 neighbors.
\item[Rule 3B]: Every vertex $v$ with $d(v) = 6$ adjacent to at most 2 unhappy neighbors distributes its charge equally to those neighbors. If the first condition is not met, and $w$ is adjacent to exactly three upward neighbors $x,y,z$ of degree 5 with $xy \in E(G)$ and $xz,yz \notin E(G)$, then $w$ sends $1 - \frac{1}{2}\varepsilon$ charge to $z$ and $\frac{1}{2} - \frac{1}{4}\varepsilon$ charge to each of $x,y$\footnote{If $xz\in E(G)$ or $yz \in E(G)$, then $\{x,y,z\}$ is a connected component of $D_5(G)$, which is handled by other means.}
\end{description}

A degree-6 vertex which has not discharged by the end of \textsc{stage 3} will be called a \emph{reserved vertex}. We define the \emph{reserve degree} of a vertex $x \in D_5(G)$ to be the number of neighbors of $x$ which are reserved and denote this by $r(x)$.

Vertices of $D_5(G)$ which do not neighbor too many reserved vertices will be satisfied after \textsc{stage 3}. The exceptions must have specific structural properties, which we now characterize.

\begin{DEF}
A \emph{dangling vertex} is a vertex $x \in D_5(G)$ which satisfies the following properties:
\begin{itemize}
\item $x$ has a neighbor $y \in D_5(G)$, but $x,y$ are not in the same cluster of $G$.
\item $\Gclone{x}{y}$ is 2-tight and ungemmed.
\item There exists a kite $H$ with $V^\ast(H) \cap D_5(G) = \{x\}$ and $y \notin V(H)$. If $w \in V^\ast(H)$ has $d_G(w)=6$, then $w$ is not adjacent to $y$.
\end{itemize}
If $x$ satisfies these conditions with respect to $y$, we say that $x$ \emph{dangles from $y$}.
\end{DEF}

\begin{LMA}\label{kite}
Suppose that $x \in D_5(G)$ dangles from $y$. Then all of the following hold:
\begin{itemize}
\item $r(x)\le 2$, and
\item if $|N(x)\cap D_5(G)| \ge 2$, then $r(x)\le 1$, and
\item if $r(x) \leq 1$ and $N(x)\cap D_5(G) = \{y\}$, then $x$ is satisfied after \textsc{stage 3} and $ch_3(x) \geq \frac{2}{5} - \frac{11}{5}\varepsilon$, and
\item if either $r(x)=2$, or $r(x)=1$ and $|N(x)\cap D_5(G)|\ge 2$, then $ch_3(x) \geq -3 - \varepsilon + 2\psi(7,5) = -\frac{1}{5} - \frac{7}{5}\varepsilon$. %change to $r(x)\ge 1$ 
\end{itemize}
\end{LMA}
\begin{proof}
By the definition of dangles, we have that $x\in D_5(G)$, $xy\in E(G)$, $\Gclone{x}{y}$ is 2-tight and ungemmed, and there exists a kite $H$ with $V^\ast(H) \cap D_5(G) = \{x\}$. Furthermore by definition, if $w \in V^\ast(H)$ has $d_G(w)=6$, then $w$ is not adjacent to $y$. Let $V^\ast(H) = \{x, w_1, w_2, w_3\}$.

Note that since $d(x)=5$ and $G$ does not admit $2$-edge additions, we have that $x$ is not incident with a replacement edge of $H$. This follows from \Cref{3inE}, since if $x$ were incident with a replacement edge $F$ of $H$, then $d_F(x) \geq 3$, in which case $d_G(x) \geq 6$ (since $w_1,w_2,y \in N(x)\setminus V(F)$), a contradiction. Hence for each $i\in\{1,2,3\}$,  $x$ is adjacent to $w_i$.

We first prove a series of claims which concerns the case where two of $w_1, w_2, w_3$ have degree 6 and are adjacent.

\begin{claim}\label{6e6plus5A}
Suppose that $d(w_1) = d(w_2) = 6$, and $w_1w_2 \in E(G)$. If $z \in (D_5(G) - H)\cap (N(w_1)\cup N(w_2))$ and $zy \in E(G)$, then $zx \notin E(G)$) and $z$ is happy after \textsc{stage 2}.
\end{claim}
\begin{proof}
Suppose not. We may assume without loss of generality that $z \in N(w_1)$. First suppose that $zx \in E(G)$. Recall that by assumption on $H$ and the definition of dangling, $y$ is not adjacent to $w_1$. Hence if we consider $D_5(\Gclone{z}{y})$, we find that $x,y,\tilde{y}, w_1 \in D_5(\Gclone{z}{y})$. Thus $D_5(\Gclone{z}{y})$ has a component of size at least 4. So by \Cref{d5struct}, $\Gclone{z}{y}$ is not both 2-tight and ungemmed.

By \Cref{criticalClone}, $\Gclone{z}{y}$ is not $5$-colorable and hence contains a $6$-critical subgraph $H'$. Since $G$ is tight and ungemmed, we have by \Cref{ungemmedClone} that $\Gclone{z}{y}$ is ungemmed. Hence if $H'=\Gclone{z}{y}$, it follows that $\Gclone{z}{y}$ is $2$-tight and ungemmed, a contradiction to what was proven above. So we may assume that $H' \ne \Gclone{z}{y}$. 

By \Cref{cloning}, $H'$ is $6$-Ore. By \Cref{cloneKey}, we have that $y,\tilde{y} \in V^\ast(H')$. By counting degrees, we deduce that $x,w_1,w_2 \in V^\ast(H')$ as well, whence $yw_1 \in E(G)$ (again counting degrees with \Cref{3inE}), a contradiction. 

So we may assume that $zx \notin E(G)$. Hence $xyz$ is an induced path in $D_5(G)$, and by \Cref{2pathEnds}, $z$ is happy after \textsc{stage 2}, a contradiction. 
\end{proof}

\begin{claim}\label{6e6plus5B}
Suppose that $d(w_1) = d(w_2) = 6$, and $w_1w_2 \in E(G)$. If $z \in (D_5(G) - H)\cap (N(w_1)\cup N(w_2))$ and $zx, zy\notin E(G)$, then every vertex in $D_5(G) \cap (N(w_1)\cup N(w_2))\setminus \{x,z\}$ is happy after \textsc{stage 2}.
\end{claim}
\begin{proof}
Suppose not. That is, there exists a vertex $z' \in D_5(G) \cap (N(w_1) \cup N(w_2)) \setminus \{x,z\}$ that is unhappy after \textsc{stage 2}. Hence the conclusion of Claim~\ref{6e6plus5A} does not apply to $z'$. This implies that $yz'\notin E(G)$. Hence $z' \in D_5(\Gclone{x}{y})$. Similarly since $zy \notin E(G)$, we have that $z \in D_5(\Gclone{x}{y})$. Thus $\{w_1,w_2, z, z'\}$ are in a component of size at least 4 in $\Gclone{x}{y}$.

Since $x$ dangles from $y$ by hypothesis, we have by definition of dangling that $\Gclone{x}{y}$ is 2-tight and ungemmed, contradicting \Cref{d5struct} as there is a component of $D_5(\Gclone{x}{y})$ of size at least 4.
\end{proof}

\begin{claim}\label{6e6plus5Two}
If $d(w_1) = d(w_2) = 6$ and $w_1w_2 \in E(G)$, then at most two vertices in $(D_5(G) - H) \cap (N(w_1) \cup N(w_2))$ are unhappy after \textsc{stage 2}, and if there are two, then at least one of them is adjacent to $x$.
\end{claim}
\begin{proof}
Suppose not. Let $Z$ be the set of vertices in $(D_5(G) - H) \cap (N(w_1) \cup N(W_2))$ that are unhappy after \textsc{stage 2}. Hence either $|Z|\ge 3$, or $|Z|=2$ and $|Z\cap N(x)|=0$. For each $z\in Z$, it follows Claim~\ref{6e6plus5A} that $zy\notin E(G)$ since $z$ is unhappy after \textsc{stage 2}. Then by Claim~\ref{6e6plus5B} since $|Z|\ge 2$, it follows that $zx\in E(G)$ for each $z\in Z$. Thus $Z\subseteq N(x)$. Since $|N(x)|=5$ and $w_1,w_2,y\in N(x)\setminus Z$, we find that $|Z|\le 2$. Hence $|Z|=2$ and yet $|Z\cap N(x)| = 2$, a contradiction. 
\end{proof}

We now proceed with the main argument. We first argue that $r(x) \leq 2$. If $x$ neighbored three reserved vertices, then at least two of $w_1, w_2, w_3$ are reserved; we may assume without loss of generality that $w_1$ and $w_2$ are reserved.

\begin{claim}\label{66adjacent}
If both $w_1, w_2$ are reserved, then $w_1, w_2$ are adjacent. 
\end{claim}
\begin{proof}
Suppose not. Hence there is a replacement edge $e=w_1w_2$  of $H$. In the latter case, \Cref{3inE} implies that $w_1$ has at least three neighbors inside $e$, and thus at most one upward neighbor of degree 5 not in $H$. By \Cref{splitProto}, the degree-5 neighbors of $w_1$ inside $e$ are contained in proto-gadgets. Therefore by \Cref{dischargeProtoGadget} and \Cref{6downHappy}, the degree 5-neighbors of $w$ inside $e$ are happy after \textsc{stage 2}. Hence $w_1$ is adjacent to at most two neighbors that are unhappy after \textsc{stage 2}, which contradicts $w_1$ being reserved. Thus, $w_1$ is not incident to a replacement edge, a contradiction.
\end{proof}

\begin{claim}\label{oneReserve}
At least one of $w_1$ or $w_2$ is not reserved.
\end{claim}
\begin{proof}
Suppose not. That is $w_1$ and $w_2$ are reserved. By Claim~\ref{66adjacent}, $w_1$ and $w_2$ are adjacent. Since $w_1$ is reserved, by definition we have that $|N(w_1)\cap D_5(G)\setminus \{x\}| \ge 2$ and at least two vertices in $N(w_1)\cap D_5(G)\setminus \{x\}$ are unhappy after \textsc{stage 2}, call them $z_1$ and $z_2$. By \Cref{6e6plus5A}, we find that $z_1,z_2\notin N(y)$. It then follows that $z_1, z_2 \in D_5(\Gclone{x}{y})$. But then $\{z_1, z_2, w_1,w_2\}$ are in a component of size at least 4 in $\Gclone{x}{y}$, contradicting \Cref{d5struct} because $\Gclone{x}{y}$ is 2-tight and ungemmed by hypothesis. 
\end{proof}

This claim completes the proof that $r(x) \leq 2$. Moreover, if $|N(x) \cap D_5(G)| \geq 2$, then the neighbor of $x$ outside $H$ has degree 5 and is therefore not a reserved vertex. At most one of the vertices $w_1, w_2, w_3 \in V^\ast(H)$ can be reserved, by \Cref{oneReserve}, so $r(x) \leq 1$.

\begin{claim}\label{2sixes}
If at least two of $w_1, w_2, w_3$ (wlog, $w_1, w_2$) satisfy $d(w_1) = d(w_2) = 6$, and $N(x) \cap D_5(G) = \{y\}$, then $r(x) \leq 1$ and $ch_3(x) \geq \frac{2}{5} - \frac{11}{5}\varepsilon$.
\end{claim}
\begin{proof}
The proof is divided into cases depending on $d(w_3)$.
\begin{description}
\item[Case A] If $d(w_3) = 6$, then by \Cref{3inE}, there can be at most one replacement edge in $H$. Hence, we deduce that $\{w_1,w_2,w_3\}$ induces a clique in $G$, and is therefore in a connected component of $D_5(\Gclone{x}{y})$. By \Cref{d5struct}, we have that $w_1w_2w_3$ is a component of $D_5(\Gclone{x}{y})$. Let $S=D_5(G)\cap ((N(w_1)\cup N(w_2)\cup N(w_3))\setminus \{w_1,w_2,w_3,x\})$. It follows that $S$ is a subset of $N(y)$, and since $|N(y)\cap D_5(G)\setminus \{x\}|\leq 1$, we find that $|S| \leq 1$.

Suppose $|S|=1$ and let $z\in S$. If $yz \notin E(G)$, then $\{w_1,w_2,w_3,z\}$ is a component of size 4 in $D_5(\Gclone{x}{y})$, a contradiction. Hence, we must have $yz \in E(G)$. Our other assumption on $N(x) \cap D_5(G)$ implies that $xz \notin E(G)$, so $(x,y,z)$ is an induced path in $D_5(G)$. By \Cref{2pathEnds}, $x$ is happy after \textsc{stage 2}.

So we may assume that $S=\emptyset$. Then each of $w_1, w_2, w_3$ is adjacent to at most one unhappy vertex (namely, $x$ itself), and $x$ receives $2-\varepsilon$ charge from each until $x$ is happy.

\item[Case B] Suppose $d(w_3) \geq 7$.

First, consider the case where a replacement edge $e$ joins $w_1, w_2$. \Cref{3inE} then implies that $w_1w_3, w_2w_3 \in E(G)$, so $w_3$ has at most 5 unhappy upward neighbors and therefore triggers in \textsc{rule 2a}. \Cref{chargeBounds} implies that $x$ receives at least $\psi(7,5)$ from $w_3$. For $w_1$ and $w_2$, \Cref{splitProto} and \Cref{6downHappy} imply that degree-5 neighbors inside $e$ are happy after \textsc{stage 2}. Counting degrees, $w_1$ and $w_2$ can have at most one degree-5 neighbor outside $H$, and thus at most two unhappy neighbors (one being $x$). It follows that the conditions for \textsc{rule 3b} apply and $w_1, w_2$ each send at least $1 - \frac{1}{2}\varepsilon$ charge to $x$. Thus, $ch_3(x) \ge 2(1-\frac{1}{2}\varepsilon)+\psi(7,5) - 3 - \varepsilon = \frac{2}{5} - \frac{11}{5}\varepsilon$.

If $w_1w_2 \in E(G)$, then by \Cref{6e6plus5Two}, there is at most one unhappy vertex in $(D_5(G) - H) \cap (N(w_1) \cup N(w_2))$ since by assumption $N(x)\cap D_5(G)\setminus \{y\}=\emptyset$. Thus, $x$ receives at least $2 - \epsilon$ charge from $w_1$ and $w_2$ together, and we again have $ch_3(x) \geq \frac{2}{5} - \frac{11}{5}\varepsilon$.
\end{description}
\end{proof}

To complete the proof, it remains to consider the case where at least two vertices of $V^\ast(H)$ have degree at least 7.

\begin{claim}\label{2sevens}
Suppose that $d(w_1), d(w_2) \geq 7$. Let $\alpha_3$ denote the charge sent to $x$ by $w_3$. Then $x$ receives at least $2\psi(7,5) + \alpha_3$ charge from $V^\ast(H)$.
\end{claim}
\begin{proof}
It suffices to show that $w_1, w_2$ trigger by \textsc{stage 2}. Since $w_1 \in V^\ast(H)$, $w_1$ can have at most $d(w_1) - 2$ upward neighbors, and therefore meets the conditions of \textsc{rule 1} or \textsc{rule 2}. It follows that $w_1$ sends at least $\psi(7,5)$ charge to $x$, and likewise for $w_2$.
\end{proof}

Let $\nu$ denote the charge sent to $x$ by its neighbor $\hat{z}$ outside $H$.

\begin{claim}
Suppose that $d(w_1),d(w_2) \geq 7$. Then we either have $2\psi(7,5) + \alpha_3 + \nu \geq \frac{2}{5} - \frac{11}{5}\varepsilon$, or one of the following two conditions holds:
\begin{enumerate}
\item $r(x) = 2$, or
\item $r(x) = 1$ and $|N(x) \cap D_5(G)| \geq 2$.
\end{enumerate}
\end{claim}
\begin{proof}
If $r(x) \leq 1$, then $x$ receives charge from at least one of $w_3$ or $\hat{z}$. By the definition of \textsc{rule 3}, we find that $\alpha_3 + \nu \geq 1 - \frac{1}{2}\varepsilon$ and hence $ch_3(x) \geq 2\psi(7,5)+1-\frac{1}{2}\varepsilon - 3 - \varepsilon = \frac{2}{5} - \frac{11}{5}\varepsilon$.

Suppose? $\alpha_3 + \nu = 0$. Then $ch_3(x) \geq 2\psi(7,5) - 3 - \varepsilon = -\frac{1}{5} - \frac{7}{5}\varepsilon$. This can only occur if \emph{neither} $w_3$ nor $\hat{z}$ sends charge to $x$, which implies that either:
\begin{enumerate}
\item Both $w_3$ and $\hat{z}$ are reserved, or
\item $w_3$ is reserved and $d_G(\hat{z}) = 5$.
\end{enumerate}
\end{proof}

Combining \Cref{2sixes} and \Cref{2sevens}, the proof of \Cref{kite} is complete.
\end{proof}

\begin{LMA}\label{edge06}
Let $S = \{x,y\}$ be a component of size 2 in $D_5(G)$ such that $S$ is not a cluster. If both $\Gclone{y}{x}$ and $\Gclone{x}{y}$ are 2-tight and ungemmed, and $|N(x)\cap N(y)\cap D_6(G)|=0$, then at least one of the following holds:
\begin{enumerate}
\item $x$ is happy after \textsc{stage 3}, or
\item Every neighbor of $x$ (other than $y$) sends at least $1 - \frac{1}{2}\varepsilon$ charge to $x$, or
\item $y$ dangles from $x$.
\end{enumerate}
The same applies to $y$.
\end{LMA}
\begin{proof}
We consider $x$. The same argument applies symmetrically to $y$. Assume that $x$ is unhappy after \textsc{stage 3} and hence after \textsc{stage 2} as well. By \Cref{unhappyUp}, $x$ is upward of its neighbors and therefore receives charge from them in \textsc{stage 2}.
	
Let $w \neq y$ be a neighbor of $x$. If $d(w) \geq 7$, then $w$ sends at least $\psi(7,7) \geq 1 - \frac{1}{2}\varepsilon$ charge to $x$ as $w$ discharges during either \textsc{stage 2} or \textsc{stage 3}. Hence, we may assume $d(w) = 6$.
	
If $w$ has at most one neighbor of degree 5 other than $x$, then $w$ sends at least $1 - \frac{1}{2}\varepsilon$ charge to $x$ by \textsc{rule 3a} and (2) holds. Thus, we may assume that $w$ has at least two other neighbors $z_1, z_2 \neq x$ of degree 5. 
	
First suppose $z_1, z_2$ are adjacent. Then $\{w, z_1, z_2\}$ is a triangle in $D_5(\Gclone{x}{y})$. By \Cref{loneK3}, it follows that $w$ has exactly three neighbors $x, z_1, z_2 \in D_5(G)$ and furthemore that $z_1z_2 \in E(G)$. Therefore $w$ triggers \textsc{rule 3b} and sends $1 - \frac{1}{2}\varepsilon$ charge to $x$.
	
So we may assume that $z_1$ and $z_2$ are not adjacent. Let $G' = \Gclone{x}{y}$. By assumption, $G'$ is 2-tight and ungemmed. Now $(z_1, w, z_2)$ is an induced path of $D_5(G')$. Yet $\clone{G'}{w}{z_1}, \clone{G'}{w}{z_2}$ have vertices of degree 4, and therefore are not 6-critical. By \Cref{cloning}, it follows that there exist gadgets $H_1, H_2$ of $G'$ containing $z_1, z_2$ respectively. 
	
Suppose that $H_1$ is not a gadget in $G$. Then $y, \wt{y} \in V(H_1)$, as $\wt{y}$ is the only vertex in $V(G') \setminus V(G)$. Clearly $y \notin V^\ast(H_1)$, as otherwise $y$ is adjacent to $z_1$. Applying \Cref{splitProto} to the replacement edge containing $y$ in $H_1$, there exists a proto-gadget $H_3 \subseteq H_1$ with $y \in V^\ast(H_3)$, so $\wt{y} \in V^\ast(H_3)$ as well. But $H_3 - \wt{y}$ is then a kite in $G$ containing $y$, with $x \notin V(H_3) \setminus \{\tilde{y}\}$ (since $x \notin V(G')$), and hence $y$ dangles from $x$ and (3) holds.
	
So we may assume that $H_1$ is a gadget of $G$ and by symmetry that $H_2$ is also a gadget of $G$. Then $w$ has at most three neighbors $z_1, z_2, z_3$ having degree 5, all of which are pairwise not adjacent, and such that there exist gadgets $H_1, H_2, H_3$ in $G$ containing $z_1, z_2, z_3$ respectively. None of $z_1, z_2, z_3$ is in a cluster of size 2. So by \Cref{dischargeProtoGadget}, $z_1, z_2, z_3$ are happy after \textsc{stage 2}. Thus $w$ sends $2 - \varepsilon$ charge to $x$ in \textsc{rule 3b} and (2) holds.
\end{proof}

\begin{LMA}\label{edge16}
Let $S = \{x,y\}$ be a component of size 2 in $D_5(G)$ such that $S$ is not a cluster. If both $\Gclone{y}{x}$ and $\Gclone{x}{y}$ are 2-tight and ungemmed, and $|N(x)\cap N(y)\cap D_6(G)|=1$, then $y$ dangles from $x$, and $x$ dangles from $y$.
\end{LMA}
\begin{proof}
By assumption, $G' = \Gclone{x}{y}$ is 2-tight and ungemmed. Let $N(x)\cap N(y)\cap D_6(G)=\{w\}$. Since $w$ is adjacent to $y$, $w$ has degree 6 in $G'$. Let $C$  be the cluster $\{y, \wt{y}\}$ in $G'$, and consider $G'' = \clone{G'}{w}{C}$. Since $G"'$ contains a cluster of size 3, \Cref{k3Cluster} and \Cref{cloning} imply that $G''$ contains a 6-Ore subgraph $H$. \Cref{cloneKey} implies that $V^\ast(H) = (N(y) \setminus \{x,w\}) \cup \{\wt{y}, \wt{w}\}$. Deleting the cloned vertices $\{\wt{y}, \wt{w}\}$ from $H$ yields a kite $H_2$ of $G$. Since $w$ is the only common neighbor of $x,y$ having degree 6 and $w \notin V^\ast(H_2)$, $x \notin V(H_2)$, $H_2$ is a kite which verifies that $y$ dangles from $x$ and vice versa.
\end{proof}

\begin{LMA}\label{trianglePrelim}
Let $S = \{x_1,x_2,x_3\}$ be a triangle in $D_5(G)$, not containing a cluster of size 2. Then $\Gclone{x_j}{x_i}$ is 2-tight and ungemmed for any two vertices of $x_i, x_j \in S$, and all of the following hold:
\begin{enumerate}
\item no two vertices of $S$ have a common neighbor having degree 6, and
\item\label{triangle1Reserve} each vertex $x$ of $S$ has $r(x)\le 1$, and
\item\label{triangleResSingleton} if $w$ is a reserved vertex neighboring a vertex of $S$, then the other degree-5 neighbors of $w$ are singletons in $D_5(G)$, and
\item\label{triangleCharges} if $\rho$ denotes the number of reserved vertices adjacent to at least one vertex of $S$, then 
$$ch_3(S) \geq \left( -\frac{1}{5} - \frac{7}{5}\varepsilon \right) \rho + \left( \frac{1}{5} - \frac{11}{5}\varepsilon \right)  (3-\rho).$$
\end{enumerate}
\end{LMA}
\begin{proof}
We first prove the first statement that for every $i\ne j\in\{1,2,3\}$, $\Gclone{x_j}{x_i}$ is 2-tight and ungemmed. Suppose not. Without loss of generality, suppose $\Gclone{x_2}{x_1}$ is not 2-tight or not ungemmed. Since $x_3$ is adjacent to both $x_1, x_2$, the vertex $x_3$ has degree 5 in $\Gclone{x_2}{x_1}$. From \Cref{cloning}, it Follows that $\Gclone{x_2}{x_1}$ contains a $6$-Ore subgraph $H$. By counting degrees, we see that $x_1, \tilde{x_1}, x_3 \in V^\ast(H)$. Applying \Cref{oneCinH} to $x_3$, we find that $x_3$ is in the same cluster as $x_1$ in $\Gclone{x_2}{x_1}$, and therefore in $G$, which contradicts the definition of $S$. This proves the fist statement.

Next we prove that all of (1)-(4) hold. First suppose (1) does not hold, that is, $x_1, x_2$ have a common neighbor $w$ of degree 6. Note that then $\Gclone{x_2}{x_1}$ contains a cluster $\{ x_1, \wt{x_1}\}$ with a neighbor $x_3$ having degree 5 and a neighbor $w$ having degree 6, which contradicts \Cref{clusterNeighbor6}. So we may assume that (1) holds.

Let $\{w_1, w_2, w_3\} = N(x_1)\setminus S$. Let $G' = \Gclone{x_2}{x_1}$ and consider $G'' = \clone{G'}{x_3}{x_1}$. Note that $G''$ contains a cluster of size 3, so it contains a 6-Ore subgraph $H$ with $\{x_1, w_1, w_2, w_3\} \subseteq V^\ast(H)$. Let $H_1$ be the kite contained in $H$ with $V^\ast(H_1) = \{x, w_1, w_2, w_3\}$. Since $x_1$ has no common neighbors of degree 6 with $x_2, x_3$, and $x_2,x_3 \notin V(H_1)$, we see that $x_1$ dangles from both $x_2, x_3$, and therefore \Cref{kite} applies to $x_1$. Symmetrically, there exist kites $H_2$ and $H_3$, so $x_2$ dangles from $x_1,x_3$ and $x_3$ dangles from $x_1,x_2$, so \Cref{kite} applies to $x_2$ and $x_3$ as well.
	
First, notice that \Cref{kite} implies $x_1$ has reserve degree at most one and hence (\ref{triangle1Reserve}) holds. \Cref{kite} also implies that $ch_3(x_1) \geq \frac{1}{5} - \frac{11}{5}\varepsilon$ if $x_1$ neighbors no reserved vertices, and $ch_3(x_1) \geq -\frac{1}{5} - \frac{7}{5}\varepsilon$ if $x_1$ neighbors one reserved vertex. Using the same reasoning for $x_2$ and $x_3$, we sum the charges on $S$ to obtain
$$ch_3(S) = ch_3(x_1) + ch_3(x_2) + ch_3(x_3) \geq \left( -\frac{1}{5} - \frac{7}{5}\varepsilon \right) \rho + \left( \frac{1}{5} - \frac{11}{5}\varepsilon \right)  (3-\rho)$$
and hence (\ref{triangleCharges}) holds.
	
Finally suppose (3) does not hold. So we may assume without loss of generality that there exists a reserved vertex $w$ that is a neighbor of $x_1$ and that has neighbor $z_1$ of degree 5 such that $z_1$ is not a singleton in $D_5(G)$. Letting $G' = \Gclone{x_1}{x_2}$, we find that $w$ has degree 5 in $D_5(G')$. Let $T$ be the component containing $w$ in $D_5(G')$. Note that $N(w)\cap D_5(G) \setminus \{x_1\} \subseteq V(T)$. By \Cref{d5struct} applied to $D_5(G')$, $T$ has size at most 4, and if $T$ has size 4 then $T$ is a star. Since $w$ is reserved, we have that $w$ is adjacent to at least 3 upward components of $D_5(G)$. Thus it follows that $w$ has degree at least two in $T$ and that $N(w)\cap V(T)$ is an independent set. Since $z_1$ is not a singleton of $D_5(G)$, it follows that $z_1$ has degree at least two in $T$. Thus $T$ is a component of size 4 with two vertices $w, z_1$ of degree at least 2, which contradicts the fact that $T$ is a star.
\end{proof}

\begin{CORO}\label{sendsToTriangle}
If $w$ is a reserved vertex that is a neighbor of a vertex in a triangle $S$ of $D_5(G)$, then $|(N(w)\cap D_5(G))\setminus S|\ge 2$ and if $v\in (N(w)\cap D_5(G))\setminus S$, then $v$ is a singleton.
\end{CORO}

\subsection{Counting Charge}\label{subs:accounting}

Using the results from the previous section, we can show that almost all components of $D_5(G)$ are satisfied after \textsc{stage 3}. Before completing the final stages, we perform an accounting of which vertices still require charge.

\begin{PROP}\label{2ClusterSatisfied}
If $C$ is a cluster of size 2 and $C$ has no neighbor having degree 5, then $ch_3(C) \geq 0$.
\end{PROP}
\begin{proof}
If $C = \{x,y\}$ is contained in a proto-gadget of $G$, then \Cref{dischargeProtoGadget} implies that $ch_3(C) \geq ch_2(C) \geq 0$. Thus, if $ch_2(C) < 0$, then $C$ is not contained in a proto-gadget of $G$. By \Cref{clusterNeighbor6}, the neighbors of $C$ all have degree at least 7, and \Cref{unhappyUp} implies that $x$ and $y$ are both upward of the neighbors of $C$. Thus $x$ and $y$ each receive at least $\psi(7,7) = 1 - \frac{1}{7}\varepsilon$ charge from four vertices. So $ch_3(C) \geq -6 - 2\varepsilon + 2(4 - \frac{4}{7}\varepsilon) \geq 0$.
\end{proof}

\begin{PROP}\label{cluster5Neighbor}
Let $C = \{x,y\}$ be a cluster with a neighbor $z$ having degree 5, and let $S = \{x,y,z\}$. Then $ch_3(S) \geq -1 + \varepsilon$, and furthermore, $ch_3(S) \geq 0$ unless $z$ has reserve degree at least $3$.
\end{PROP}
\begin{proof}
Note that $\Gclone{z}{C}$ has a cluster of size 3. So $C$ is contained in a gadget $H$. By \Cref{dischargeProtoGadget}, $ch_2(C) \geq 2 + 2\varepsilon$, as the neighbor $z$ of $C$ not in $H$ has degree 5. Therefore $ch_3(S) = ch_3(C) + ch_3(z) \geq -1 + \varepsilon$. If any of the neighbors of $z$ is not a reserved degree-6 neighbor, then $z$ receives at least another $1 - \frac{1}{2}\varepsilon$ charge, whence $ch_3(S) \geq 0$ as desired.
\end{proof}

\begin{PROP}\label{2PathSatisfied}
Let $S = (x,y,z)$ be an induced path in $D_5(G)$. Then $ch_3(S) \geq 0$.
\end{PROP}
\begin{proof}
By \Cref{2pathEnds}, we have $ch_2(x), ch_2(z) \geq 2+2\varepsilon$. Therefore $ch_3(S) \geq ch_2(x) + ch_2(z) - 3 - \varepsilon \geq 1 + 3\varepsilon > 0$.
\end{proof}

\begin{PROP}\label{edgeOf5s}
Let $S = \{x,y\}$ be a component of size 2 in $D_5(G)$, not in a cluster, and let $S' = \{ v \in S: ch_3(v) < 0\}$. Let $\rho = \max_{v \in S'} r(v)$. Then $ch_3(S) \geq \min\{0, -\rho + \frac{3}{2}\}$. Furthermore, if $\rho \geq 3$, then we have a stronger bound that $ch_3(S) \geq -\rho + \frac{8}{5}$.
\end{PROP}
\begin{proof}
Suppose that $x$ is happy after \textsc{stage 3}, so $ch_3(x) \geq 2 + 2\varepsilon$. Since $x$ and $y$ have degree 5 and are adjacent, $\rho \leq 4$. We calculate that
\begin{align*}ch_3(S) = ch_3(x) + ch_3(y) &\geq 2 + 2\varepsilon - 3 - \varepsilon + (4 - \rho)\left( 1 - \frac{1}{2} \varepsilon\right) \\
& = -\rho + 3 - \left( 1-\frac{1}{2}\rho \right) \varepsilon \\
&\geq -\rho + \frac{8}{5}.
\end{align*}
If $\Gclone{y}{x}$ contains a 6-Ore subgraph, then $x$ is in a proto-gadget and \Cref{dischargeProtoGadget} implies that $x$ is happy after \textsc{stage 2}. Therefore, by \Cref{cloning,ungemmedClone,tightClone}, we may assume that $\Gclone{y}{x}$ and $\Gclone{x}{y}$ are both 2-tight and ungemmed.
	
$\Gclone{y}{x}$ contains a cluster $C = \{x, \wt{x}\}$ of size 2. \Cref{clusterNeighbor6} implies that $C$ has at most 1 neighbor of degree 6. Such a neighbor of $C$ is either a vertex of degree 5 in $G$ adjacent to $x$ but not $y$, or a vertex having degree 6 in $G$ adjacent to both $x$ and $y$. Since $x$ does not have a neighbor having degree 5 that is distinct from $y$, we deduce that $x,y$ have at most 1 common neighbor having degree 6. Therefore either \Cref{edge06} or \Cref{edge16} applies.
	
Suppose first that $x,y$ have no common neighbors of degree 6. Apply \Cref{edge06} to the vertex $x$. If \Cref{edge06}(1) holds, then $x$ is happy after \textsc{stage 3} and the proposition follows as above. 
	
So suppose \Cref{edge06}(2) holds. If $x$ receives at least $1 - \frac{1}{2}\varepsilon$ charge from each neighbor, then $x$ neighbors no reserved vertices and $\rho$ is the number of reserved vertices neighboring $y$. The amount of charge received by $y$ is at least $1  -\frac{1}{2}\varepsilon$ for each non-reserved vertex, so
$$ch_3(x) + ch_3(y) \geq -6 - 2\varepsilon + 4\left(1 - \frac{1}{2}\varepsilon \right) + (4 - \rho)\left(1 - \frac{1}{2}\varepsilon \right) = -\rho + 2 - \left(6 - \frac{1}{2}\rho \right) \varepsilon,$$
which is clearly greater than $-\rho + \frac{8}{5}$ as desired.

Hence we may assume that \Cref{edge06}(3) holds, that is $y$ dangles from $x$. Since \Cref{edge06} applies symmetrically to $y$, we find by symmemtry that $x$ dangles from $y$. Now \Cref{kite} implies that
\begin{enumerate}
\item $x$ neighbors at most two reserved vertices.
\item $ch_3(x) \geq -\frac{1}{5} - \frac{7}{5}\varepsilon$.
\item If $x$ neighbors at most one reserved vertex, then $ch_3(x) \geq \frac{1}{5} - \frac{11}{5}\varepsilon$.
\end{enumerate}
By symmetry, the same applies to $y$, so $\rho \leq 2$. Thus, if neither $x$ nor $y$ neighbors two reserved vertices, then $ch_3(S) \geq 2(\frac{1}{5} - \frac{11}{5}\varepsilon) \geq 0$. If $\rho = 2$, then we have $ch_3(S) \geq 2(-\frac{1}{5}-\frac{7}{5}\varepsilon) \geq -\frac{1}{2} = -2 + \frac{3}{2}$, as desired.
	
If $x,y$ have exactly one common neighbor having degree 6, then \Cref{edge16} implies that $x$ dangles from $y$ and vice versa. We have just shown that we have the desired result in this case.
	
Notice that in all the cases where $\rho \geq 3$ is possible, the stronger bound $ch_3(S) \geq -\rho + \frac{8}{5}$ holds.
\end{proof}

\subsection{Global Discharging}\label{subs:global}
We can now employ a global discharging argument to resolve the singletons and all other unsatisfied vertices.

Let $U$ be the set of components $S$ of $D_5(G)$ such that $ch_3(S) < 0$. From each component $S$ of $U$, let $v_S$ denote a vertex of $S$ whose reserve degree is maximum over vertices in $S$. Let $A$ denote this set of vertices; note that $A$ is an independent set. Let $B$ denote the set of reserved vertices neighboring $A$. We apply \Cref{A2B} to $A$ and $B$, and deduce that $|E_G(A,B)| \leq |A| + 2|B|$.

%\begin{UCORO}[\ref{A2B}]
%Let $G$ be a $k$-critical graph, and let $A \subseteq V(G)$ be an independent set with $d_G(v) = k-1$ for all $v \in A$. Let $B$ be the set of neighbors of $A$ of degree $k$. Then $|E_G(A,B)| \leq |A| + 2|B|$.
%\end{UCORO}

We now perform the following rules in stages (namely we have separate \textsc{stages 4, 5 and 6}).
\begin{description}
\item[Rule 4]: Add $-\frac{4}{5} - \frac{1}{5}\varepsilon$ charge to every vertex in $A$, and distribute $\left( \frac{4}{5} + \frac{1}{5}\varepsilon \right)|A|$ charge among the vertices of $B$. The total charge is unchanged. 
\item[Rule 5]: Send $\frac{4}{5} + \frac{1}{5}\varepsilon$ charge along every edge in $E(A,B)$ towards $A$.
\item[Rule 6]: Redistribute the charge in $B$ so that every vertex of $B$ has at least $\frac{2}{5} - \frac{7}{5}\varepsilon$ charge. Every vertex in $B$ adjacent to a triangle sends $\frac{2}{5}-\frac{7}{5}\varepsilon$ charge to that triangle.
\end{description}

First, we verify that $B$ has sufficient charge stored after \textsc{stage 5} to implement \textsc{rule 6}. Observe that $ch_4(B) = ( \frac{4}{5} + \frac{1}{5}\varepsilon)|A| + (2-\varepsilon)|B|$. By \Cref{A2B}, the total charge sent in \textsc{stage 5} from $B$ to $A$ is at most
$$\left( \frac{4}{5} + \frac{1}{5}\varepsilon \right) |E(A,B)| \leq  \left( \frac{4}{5} + \frac{1}{5}\varepsilon \right) (|A| + 2|B|) = \left( \frac{4}{5} + \frac{1}{5}\varepsilon \right) |A| + \left(\frac{8}{5}  + \frac{2}{5} \varepsilon \right) |B|$$
Consequently,
\begin{align*}ch_5(B) &\geq  \left( \frac{4}{5} + \frac{1}{5}\varepsilon \right) |A| + \left(2-\varepsilon \right) |B| -\left( \frac{4}{5} + \frac{1}{5}\varepsilon \right) |A| - \left(\frac{8}{5} + \frac{2}{5} \varepsilon \right) |B| \\
&\geq \left( \frac{2}{5} - \frac{7}{5}\varepsilon \right) |B|.
\end{align*}
By \Cref{trianglePrelim}, a vertex in $B$ can be adjacent to at most one triangle. Thus, we have enough charge in $B$ after \textsc{stage 5} to carry out \textsc{rule 6}. which ensures that $ch_6(B) \geq 0$.
\begin{LMA}
A triangle neighboring $\rho$ reserved vertices receives $(\frac{2}{5}-\frac{7}{5}\varepsilon)\rho$ charge in \textsc{stage 6}.
\end{LMA}
\begin{proof}
Let $w$ be a reserved vertex neighboring a triangle $S$. By \Cref{sendsToTriangle}, $w \in B$, as $w$ is adjacent to a singleton in $A$. Therefore $w$ sends charge to $S$ in \textsc{stage 6}.
\end{proof}

We check that $ch_6(x) \geq 0$ for all vertices of degree 5.

\begin{PROP}
All vertices are satisfied after \textsc{stage 6}.
\end{PROP}
\begin{proof}
For vertices of degree at least 7, this follows immediately as our discharging rules never send more than the initial charge. Vertices of degree 6 also do not send more than their initial charge unless they belong to $B$. We have shown that $ch_6(B) \geq 0$.
	
It only remains to check the charge of components of $D_5(G)$. From \Cref{d5struct}, there are six possible components of $D_5(G)$. We verify that each type is satisfied.
\begin{description}
\item[Singletons:] A singleton $x \in D_5(G)$ has $ch_4(x) = -3 - \frac{4}{5} - (1 + \frac{1}{5})\varepsilon$. It receives at least $1 - \frac{1}{2}\varepsilon$ charge from its neighbors of degree at least 7, and $\frac{4}{5} + \frac{1}{5}\varepsilon$ from its neighboring reserved vertices. Therefore, it receives at least $\frac{4}{5} - \frac{1}{5}\varepsilon$ charge from every neighbor and
$$ch_5(x) \geq -3 - \frac{4}{5} - \left( 1 + \frac{1}{5} \right) \varepsilon + 5\left(\frac{4}{5} + \frac{1}{5}\varepsilon \right) = \frac{1}{5} - \frac{1}{5}\varepsilon \geq 0,$$
since $\varepsilon \leq 1$.
		
\item[Induced Paths of Length 2:] Let $S = (x,y,z)$ be an induced path of length 2 in $D_5(G)$. By \Cref{2PathSatisfied}, $ch_6(S) \geq ch_3(S) \geq 0$.
		
\item [Clusters of Size 2:] Let $C = \{x,y\}$ be a cluster of size 2. By \Cref{2ClusterSatisfied}, $ch_6(C) \geq ch_3(C) \geq 0$.
		
\item[Clusters of Size 2 with a Degree 5 Neighbor:] Let $S = \{x,y,z\}$, with $C = \{x,y\}$ a cluster of size 2. By \Cref{cluster5Neighbor}, $ch_3(S) \geq 0$ unless $z$ neighbors 3 reserved vertices, in which case $ch_3(S) \geq -1 + \varepsilon$. Then $z \in A$ and the reserved vertices neighboring $z$ are in $B$, so $z$ receives $3( \frac{4}{5} + \frac{1}{5}\varepsilon)$ charge in \textsc{stage 5}. Thus
$$ch_6(S) \geq ch_3(S) - \frac{4}{5} - \frac{1}{5}\varepsilon + 3\left( \frac{4}{5} + \frac{1}{5}\varepsilon \right) \geq \frac{3}{5} + \frac{8}{5}\varepsilon > 0,$$
		
since $\varepsilon \ge 0$.
		
\item[Two Adjacent Vertices:] Let $S = \{x,y\}$ be a component of $D_5(G)$, not in a cluster, with $ch_3(S) < 0$. We may assume that $ch_3(x) < 0$ and without loss of generality that the reserve degree of $x$ is at least as large as the reserve degree of $y$. Then $x \in A$, and the $\rho$ reserved vertices neighboring $x$ are in $B$, so $x$ receives at least $(\frac{4}{5} + \frac{1}{5}\varepsilon)\rho$ charge in \textsc{stage 5}. By \Cref{edgeOf5s}, we find that $ch_3(S) \geq -\rho + \frac{3}{2}$, and if $\rho \geq 3$, then $ch_3(S) \geq -\rho + \frac{8}{5}$. If $\rho \leq 2$, then
\begin{align*}
ch_6(S) \geq \left( -\rho + \frac{3}{2} \right) - \frac{4}{5} - \frac{1}{5}\varepsilon +  \left(\frac{4}{5} + \frac{1}{5}\varepsilon \right) \rho &= \frac{7}{10} - \frac{1}{5}\rho + \frac{\rho - 1}{5}\varepsilon \\
&\geq \frac{7}{10} - \frac{2}{5} + \frac{\rho - 1}{5}\varepsilon \\
& \geq 0,
\end{align*}
since $\rho \ge 0$ and $\varepsilon \le \frac{3}{2}$.
\\
If $3 \leq \rho \leq 4$, we have a similar calculation:
$$ ch_5(S) \geq \frac{4}{5} - \frac{1}{5}\rho + \frac{\rho - 1}{5}\varepsilon \geq \frac{\rho - 1}{5}\varepsilon \geq 0,$$
		
since $\varepsilon \ge 0$.
		
\item[Triangle:] Let $S$ be a triangle, and let $\rho$ be the number of reserved vertices that have a neighbor in $S$. By \Cref{trianglePrelim}, $ch_3(S) \geq \left( -\frac{1}{5} - \frac{7}{5}\varepsilon \right) \rho + \left( \frac{1}{5} - \frac{11}{5}\varepsilon \right)  (3-\rho)$. Thus, if $ch_3(S) < 0$, $\rho \geq 2$. By \Cref{sendsToTriangle}, a reserved vertex $w$ neighboring $S$ is also adjacent to a singleton of $D_5(G)$ that is unsatisfied after \textsc{stage 3}. Such a singleton is in $A$, so $w \in B$.
		
Let $x$ be the vertex of $S$ chosen to be in $A$. In \textsc{stage 4}, $x$ receives $-\frac{4}{5} - \frac{1}{5}\varepsilon$ charge, and then receives $\frac{4}{5} + \frac{1}{5}\varepsilon$ charge back from its adjacent reserved vertex in \textsc{stage 5}. Hence $ch_5(S) = ch_3(S)$. Since every reserved vertex neighboring $S$ is also in $B$, by the reasoning in the previous paragraph, $S$ receives $(\frac{2}{5} - \frac{7}{5}\varepsilon)\rho$ charge from \textsc{rule 6}. Hence
\begin{align*}
ch_6(S) \geq ch_5(S) + \left( \frac{2}{5} - \frac{7}{5}\varepsilon \right) \rho &= -\frac{2}{5}\rho + \frac{3}{5} - \frac{33}{5}\varepsilon + \frac{4\rho}{5}\varepsilon + \left( \frac{2}{5} - \frac{7}{5}\varepsilon \right) \rho \\
&= \frac{3}{5} -\frac{3\rho}{5}\varepsilon - \frac{33}{5}\varepsilon \geq 0,
\end{align*}
since $\rho \le 3$ and $\varepsilon \le \frac{1}{14}$.
\end{description}
\end{proof}

Therefore, the total charge on $G$ after \textsc{stage 5} is non-negative. However, the total charge is invariant and the initial charge was $-p(G)-\delta T(G) < 0$, a contradiction. This completes the proof of \Cref{MAINTHM}.
\bibliographystyle{siam}	
\bibliography{6CriticalK4Free}

\begin{thebibliography}{10}

\bibitem{D1953FM}
{\sc G.~Dirac}, {\em The structure of $k$-chromatic graphs}, Fund. Math., 40
  (1953), pp.~42--55.

\bibitem{D1957LMS}
\leavevmode\vrule height 2pt depth -1.6pt width 23pt, {\em Theorem of {R. L.
  Brooks} and a conjecture of {H. Hadwiger}}, Proceedings of the London
  Mathematical Society, 3 (1957), pp.~161--195.

\bibitem{G1963HAS}
{\sc T.~Gallai}, {\em Kritische graphen {II}}, Publications of the Mathematical
  Institute of the Hungarian Academy of Sciences, 8 (1963), pp.~373--395.

\bibitem{GLP}
{\sc R.~Gould, V.~Larsen, and L.~Postle}, {\em Structure in $k$-critical
  graphs}, manuscript.

\bibitem{KR2017JGT}
{\sc H.~A. Kierstead and L.~Rabern}, {\em Extracting list colorings from large
  independent sets}, Journal of Graph Theory, 86 (2017), pp.~315--328.

\bibitem{KSW1996DM}
{\sc A.~Kostochka, M.~Steibitz, and B.~Wirth}, {\em The colour theorems of
  {Brooks} and {Gallai} extended}, Discrete Mathematics, 162 (1996),
  pp.~299--303.

\bibitem{KY2014JCTb}
{\sc A.~Kostochka and M.~Yancey}, {\em Ore's conjecture on color-critical
  graphs is almost true}, Combin. Theory Ser. B, 109 (2014), pp.~73--101.

\bibitem{KY2018Combinatorica}
\leavevmode\vrule height 2pt depth -1.6pt width 23pt, {\em A
  \uppercase{b}rooks-type result for sparse critical graphs}, Combinatorica, 38
  (2018), pp.~887–--934.

\bibitem{L2015PHD_EMORY}
{\sc V.~Larsen}, {\em An Epsilon Improvement to the Asymptotic Density of
  k-Critical Graphs}, PhD thesis, Emory University, 2015.

\bibitem{LP2017JGT}
{\sc C.-H. Liu and L.~Postle}, {\em On the minimum edge-density of 4-critical
  graphs of girth five}, Journal of Graph Theory, 86 (2017), pp.~387--405.

\bibitem{O1967}
{\sc {\O}.~Ore}, {\em The Four Color Problem}, Academic Press, New York, 1967.

\bibitem{P2018Manuscript}
{\sc L.~Postle}, {\em The edge-density of 4-critical graphs of girth 5 is 5/3 +
  $\epsilon$}, manuscript.

\bibitem{P2015EDM}
\leavevmode\vrule height 2pt depth -1.6pt width 23pt, {\em On the minimum
  number of edges in triangle-free 5-critical graphs}, European Journal of
  Combinatorics, 66 (2017), pp.~264--280.

\end{thebibliography}

\end{document}